\documentclass{amsart}

\usepackage{amsmath, amsthm, amssymb, amstext, amsfonts}
\usepackage{enumerate}
\usepackage{graphicx}
\usepackage[applemac]{inputenc}

\theoremstyle{plain}

\newtheorem{thm}{Theorem}[section]
\newtheorem{lem}[thm]{Lemma}

\newtheorem{prop}[thm]{Proposition}

\theoremstyle{definition}

\newcommand{\N}{\ensuremath{\mathbb{N}}}

\newcommand{\sm}{\ensuremath{\smallsetminus}}

\newcommand{\ssim}{{\scriptscriptstyle \sim}}

\newcommand{\diam}{\textnormal{diam}}
\newcommand{\isom}{\ensuremath{\cong}}
\newcommand{\inv}{\ensuremath{^{-1}}}

\newcommand{\Aut}{\textnormal{Aut}}

\newcommand{\es}{\ensuremath{\emptyset}}

\newcommand{\sub}{\subseteq}

\newcommand{\perpp}{\mspace{1mu}{\perp}\mspace{1mu}}

\newcommand{\comment}[1]{}

\newcommand{\nat}{{\mathbb N}}

\newcommand{\ganz}{{\mathbb Z}}
\newcommand{\rat}{{\mathbb Q}}

\newcommand{\Hbb}{{\mathbb H}}

\newcommand{\AF}{\ensuremath{\mathcal A}}

\newcommand{\HF}{\ensuremath{\mathcal H}}
\newcommand{\IF}{\ensuremath{\mathcal I}}

\newcommand{\PF}{\ensuremath{\mathcal P}}

\newcommand{\RF}{\ensuremath{\mathcal R}}

\def\?#1{\vadjust{\vbox to 0pt{\vss\vskip-8pt\leftline{%
     \llap{\hbox{\vbox{\pretolerance=-1
     \doublehyphendemerits=0\finalhyphendemerits=0
     \hsize16truemm\tolerance=10000\small
     \lineskip=0pt\lineskiplimit=0pt
     \rightskip=0pt plus16truemm\baselineskip8pt\noindent
     \hskip0pt        %(without this, the first word is never hyphenated!)
     #1\endgraf}\hskip7truemm}}}\vss}}}

\newenvironment{txteq}
  {
    \begin{equation}
    \begin{minipage}[t]{0.85\textwidth} % set width to 0.9 x textwidth
    \em                                % switch on emph
  }
  {\end{minipage}\end{equation}\ignorespacesafterend}
\newenvironment{txteq*}
  {
    \begin{equation*}
    \begin{minipage}[t]{0.85\textwidth} % set width to 0.9 x textwidth
    \em                                % switch on emph
  }
  {\end{minipage}\end{equation*}\ignorespacesafterend}

\begin{document}

\title{Countable connected-homogeneous digraphs}
\author{Matthias Hamann}%\bigskip \\Fachbereich Mathematik\\Universit\"at Hamburg}
\address{Matthias Hamann, Department of Mathematics, University of Hamburg, Bundes\-stra\ss e~55, 20146 Hamburg, Germany}
\date{}%\today}
\maketitle

\begin{abstract}
A digraph is \emph{connected-homogeneous} if every isomorphism between two finite connected induced subdigraphs extends to an automorphism of the whole digraph.
In this paper, we completely classify the countable connected-homogeneous digraphs.
\end{abstract}

\section{Introduction}

A graph is \emph{homogeneous} if every isomorphism between two isomorphic finite induced subgraphs extends to an automorphism of the whole graph.
This restrictive property led to a complete classification of the countable homogeneous graphs \cite{Gard-HomogeneousGraphs,GK-kHomogeneous,LW-CountUltrahomGraphs,J-SmoothlyEmbeddable}.
Homogeneous graphs are in particular vertex-transitive.
Whereas vertex-transitive graphs are too rich to obtain a full classification, there are various ways to relax the notion of homogeneity naturally to obtain a larger set of graphs that still admits a full classification.
Examples of relaxations of homogeneity for graphs are \emph{distance-transitivity} where we require transitivity on pairs of vertices with the same distance, see~\cite{C-DistanceTransitive,HP-Transitivity,DistanceTransitive}, and \emph{set-homogeneity} where we require only that some isomorphism between every two isomorphic finite induced subgraph has to extend to an automorphism of the whole graph, see~\cite{DGMS-SetHomogeneous}.
In both these cases, there is no complete classification of the countable such graphs yet.
Another relaxation of homogeneity is the following: We call a graph \emph{connected-homogeneous}, or \emph{C-homogeneous} for short, if every isomorphism between two isomorphic finite induced connected subgraphs extends to an automorphism of the whole graph.
Countable C-homogeneous graphs have been classified in~\cite{Enomoto,Gard-HomConditions,GrayMacpherson,HP-Transitivity,HedmanPong}.

When it comes to digraphs, the analogous notions of \emph{homogeneity} and \emph{C-homo\-geneity} apply.
Countable homogeneous digraphs have been classified in~\cite{Cherlin-HomImprimitive,Cherlin-CountHomDigraphs,L-FiniteHomDigraphs,L-Tournaments}.
In this paper we will complete the classificaion of the countable C-homogeneous digraph, which was started by Gray and M\"oller~\cite{GM-CHomDigraphs} and continued in~\cite{FinConHomDigraphs, HH-ConHomDigraphs}.
So far, the connected C-homogeneous digraphs of finite degree and those with more than one end have been classified.
So the purpose of this paper is to classify the countable C-homogeneous digraphs that have precisely one end and thereby complete the classification of the countable C-homogeneous digraphs.

Another structure for which homogeneity and and C-homogeneity have been considered are partial orders.
Schmerl~\cite{Schmerl-HomogeneousPO} classified the countable homogeneous partial orders and Gray and Macpherson~\cite{GrayMacpherson} classified the countable C-homogeneous partial orders.
For more details on homogeneous structures we refer to Macpherson's survey~\cite{Macpherson-Survey}.

\medskip

On our way to the classification of the countable C-homogeneous digraphs, we shall use classification results of various other homogeneous structures: we shall use the classifications of the countable homogeneous digraphs~\cite{Cherlin-CountHomDigraphs}, of the countable homogeneous bipartite graphs~\cite{GGK}, and of the countable homogeneous $2$-partite digraphs~\cite{Hom2PartiteDi}.
%The classification of the countable homogeneous $2$-partite digraphs is also part of this paper.
%It will be done in Section~\ref{sec_Hom2PartiteDigraphs}.

\medskip

The paper is structured as follows:
After introducing in Section~\ref{sec_basics} all necessary notations for the remainder of the paper, we state the classification result of the countable C-homogeneous digraph (Theorem~\ref{thm_main}) and give brief descriptions of the involved digraphs in Section~\ref{sec_MainThm}.
In Section~\ref{sec_Overview}, we shall give a rough overview of the proof of Theorem~\ref{thm_main}.
Then we state the classification of the countable homogeneous bipartite graphs, of the countable homogeneous $2$-partite digraphs, and of the countable homogeneous digraphs in Section~\ref{sec_PreviousClassifications}, which will all be part of our proof of Theorem~\ref{thm_main}.
In Section~\ref{sec_reachability}, we will introduce and discuss the reachability relation, another tool for the proof of our main theorem.
Then, we have everything we need to tackle the proof of our main theorem, which will be done in Sections~\ref{sec_ProofPartI} and~\ref{sec_ProofPartII}.\looseness-1

\section{Basics}\label{sec_basics}

A \emph{digraph} $D$ is a pair of a non-empty set $VD$ of \emph{vertices} and an asymmetric (i.e.\ irreflexive and antisymmetric) binary relation $ED$ on $VD$, its \emph{edges}.
For a subset of vertices $X\sub VD$, let $D[X]:=(X,ED\cap X\times X)$ be the digraph \emph{induced by~$X$}.
Two vertices $x,y\in VD$ are \emph{adjacent} if either $xy\in ED$ or $yx\in ED$.
The \emph{out-neighbours} or \emph{successors} of~$x\in VD$ are the elements of the \emph{out-neighbourhood} $N^+(x):=\{y\in VD\mid xy\in ED\}$ and its \emph{in-neighbours} or \emph{predecessors} are the elements of the \emph{in-neighbourhood} $N^-(x):=\{y\in VD\mid yx\in ED\}$.
Furthermore, let $D^+(x):=D[N^+(x)]$ and $D^-(x):=D[N^-(x)]$.
If $D$ is \emph{vertex-transitive}, that is, if the automorphisms of~$D$ act transitively on~$VD$, then the digraphs $D^+(x)$ and~$D^+(y)$ (the digraphs $D^-(x)$ and $D^-(y)$) are isomorphic for any two vertices $x,y$ and we denote by~$D^+$ (by~$D^-$, respectively) one element of their isomorphism class.
For induced subdigraphs $A$ and~$B$ of~$D$ and $x\in VD$, let $A+B$ be the digraph $D[VA\cup VB]$, let $A+x=D[VA\cup\{x\}]$, and let $A-x=D[VA\sm\{x\}]$.
If $B\sub A$, let $A-B=D[VA\sm VB]$.
An \emph{independent} vertex set is a set whose elements are pairwise non-adjacent.
By $I_k$ we denote an independent vertex set of cardinality~$k$ and also a digraph whose vertex set is an independet set of cardinality~$k$.
It will always be obvious from the context, whether $I_n$ describes a vertex set or a digraph.
A \emph{tournament} is a digraph such that each two of its vertices are adjacent.

For $k\in\nat$, a \emph{$k$-arc} is a sequence $x_0\ldots x_k$ or $k+1$ vertices with $x_ix_{i+1}\in ED$ for all $i\leq k-1$.
A \emph{path} (of \emph{length} $\ell\in\nat$) is a sequence $x_0\ldots x_\ell$ of $\ell+1$ distinct vertices such that for all $i\leq\ell-1$ the vertices $x_i$ and $x_{i+1}$ are adjacent.
If we have $x_ix_{i+1}\in ED$ for all $i\leq\ell-1$ then we call the path \emph{directed}.
Hence, a directed path of length~$\ell$ is an $\ell$-arc all whose vertices are distinct.
A digraph is \emph{connected} if each two vertices are joined by a path.
A vertex, vertex set, or subdigraph \emph{separates} a digraph if its deletion leaves more than one component.
It \emph{separates} two vertices, vertex sets, or subgraphs if these lie in distinct components after the deletion.

A \emph{cycle} (of \emph{length} $\ell\geq 3$) is a path of length $\ell-1$ whose end vertices are joined by an edge.
A \emph{directed} cycle, denoted by $C_\ell$, is a cycle $x_1\ldots x_{\ell-1}$ either with $x_ix_{i+1}\in ED$ and $x_{\ell-1}x_1\in ED$ or with $x_{i+1}x_i\in ED$ and $x_1x_{\ell-1}\in ED$.
Triangles are cycles of length~$3$.
Up to isomorphism, there are two distinct kinds of triangles.
We call those triangles that are not directed \emph{transitive}.
We also denote graphs that are cycles of length~$\ell$ by $C_\ell$.
It will always be clear from the context whether $C_\ell$ is a graph or a digraph.

For an equivalence relation $\sim$ on~$VD$ let $D_\ssim$ be the digraph whose vertices are the equivalence classes of~$\sim$ and where $XY\in ED_\ssim$ if and only if there are $x\in X$ and $y\in Y$ with $xy\in ED$.
We call $D_\ssim$ a \emph{quotient digraph} of~$D$ (induced by~$\sim$).
In general, this is not a digraph since it may have loops as well as edges $XY$ and $YX$.
However, we only consider equivalence relations $\sim$ such that $ED_\ssim$ is an asymmetric relation.
But in each situation in which we consider quotient digraphs $D_\ssim$ we will prove that $ED_\ssim$ is asymmetric.

The \emph{underlying undirected graph} of a digraph $D=(V,E)$ is the graph $G=(V,\{\{x,y\}\mid xy\in E\})$.
A \emph{tournament} is a digraph whose underlying undirected graph is a complete graph.

For the remainder of the paper, let $\nat^\infty=\nat\cup\{\omega\}$.
The \emph{diameter} of~$D$ is defined by
\[
\diam(D)=\inf\{n\in\nat^\infty\mid d(x,y)\leq n\text{ for all }x,y\in VD\}.
\]

A \emph{ray} in a graph is a one-way infinite path and a \emph{double ray} is a two-way infinite path.
Two rays are \emph{equivalent} if for every finite vertex set~$S$ both rays lie eventually in the same component of $G-S$.
This is an equivalence relation whose classes are the \emph{ends} of the graph.
\emph{Rays}, \emph{double rays}, and \emph{ends} of a digraph are those of its underlying undirected graph.
For abbreviation, we denote by $C_\infty$ the directed double ray.

If the underlying undirected graph of a digraph~$D$ is bipartite then $D$ is \emph{$2$-partite}.
If in addition all edges are directed from the same partition set to the other then we call~$D$ \emph{bipartite}.

\section{The main result}\label{sec_MainThm}

In this section, we state our main theorem, the classification of the countable C-homogeneous digraphs (Theorem~\ref{thm_main}).
Afterwards, we describe all the digraphs that occur in the list and that need some explanations.

\begin{thm}\label{thm_main}
A countable digraph is C-homogeneous if and only if it is a disjoint union of countably many copies of one of the following digraphs:
\begin{enumerate}[\rm (i)]
\item a countable homogeneous digraph;
\item $H[I_n]$ for some $n\in\nat^\infty$ and with either $H=S(3)$ or $H=T^\wedge$ for some countable homogeneous tournament $T\neq S(2)$;
\item $X_{\lambda}(T)$ for some countable homogeneous tournament $T$ and $\lambda\in\nat^\infty$;
\item a regular tree;
\item $DL(\Delta)$, where $\Delta$ is a bipartite digraph such that $G(\Delta)$ is one of
\begin{enumerate}[\rm (a)]
\item $C_{2m}$ for some integer $m\geq 2$,
\item $CP_k$ for some $k\in\nat^\infty$ with $k\geq 3$,
\item $K_{k,l}$ for $k,l\in\nat^\infty$, $k,l\geq 2$, or
\item the countable generic bipartite graph;
\end{enumerate}
\item $M(k,m)$ for some $k\in\nat^\infty$ with $k\geq 3$ and some integer $m\geq 2$;
\item $M'(2m)$ for some integer $m\geq 2$;
\item $Y_k$ for some $k\in\nat^\infty$ with $k\geq 3$;
\item $C_m[I_k]$ for some $k,m\in\nat^\infty$ with $m\geq 3$;
\item $\RF_m$ for some $m\in\nat^\infty$ with $m\geq 3$;
\item $X_2(C_3)_\ssim$, where $\sim$ is a non-universal $Aut(X_2(C_3))$-invariant equivalence relation on~$VX_2(C_3)$; or
\item the generic orientation of the countable generic bipartite graph.
\end{enumerate}
\end{thm}

Those countable homogeneous digraphs that are not explicitely mentioned within Theorem~\ref{thm_main} will be described in Section~\ref{sec_HomDigraphs}.

For a tournament $T$, let $T^+$ be $T$ together with a new vertex~$x$ such that $xv\in ET^+$ for all $v\in VT$.
Then $T^\wedge$ is the disjoint union of two copies $T^+\varphi_1$, $T^+\varphi_2$ with isomorphisms $\varphi_1,\varphi_2$ and with $v\varphi_1 u\varphi_2\in ED$ if and only if $uv\in ET^+$ and $v\varphi_2 u\varphi_1\in ED$ if and only if $uv\in ED$.

Let $VS(2)$ be a dense subset of the unit circle such that the angle between any two points is rational.
A vertex $x$ is the successor of a vertex~$y$ if the angle between them is smaller than $\pi$ modulo $2\pi$ (counterclockwise).
The resulting tournament is~$S(2)$.
Similarly, let $VS(3)$ be a dense subset of the unit circle such that the angle between any two points is rational, too.
Two vertices in~$S(3)$ are adjacent if the angle between them is smaller than $3\pi/2$ modulo $2\pi$ (counterclockwise).

For two digraphs $D, D'$ let the \emph{lexicographic product} $D[D']$ be the digraph with vertex set $VD\times VD'$ and edge set
\[
\{(x,x')(y,y')\mid xy\in ED\text{ or }(x=y\text{ and }x'y'\in ED')\}.\footnote{Note that if $X\sub VD$, then $D[X]$ is a subdigraph of~$D$ (the restiction of~$D$ onto~$X$) and, if $D'$ is a digraph, $D[D']$ is a new digraph (the lexicographic product).}
\]

For a homogeneous tournament $T\neq I_1$ and a cardinal $\lambda$, let $X_{\lambda}(T)$ be the digraph such that every vertex is a cut vertex and lies in $\lambda$ distinct blocks each of which is isomorphic to~$T$.

For a bipartite edge-transitive digraph $\Delta$, let $DL(\Delta)$ be the digraph such that every vertex is a cut vertex and lies in precisely two blocks each of which is isomorphic to~$\Delta$ and such that the vertex has its successors in one of the two blocks and its predecessors in the other.

The complete bipartite graph with one side of size~$k$ and the other of size~$\ell$ is $K_{k,\ell}$.
The \emph{(bipartite) complement of a perfect matching} $CP_k$ is a complete bipartite graph $K_{k,k}$ where the edges of a perfect matching are removed.
A \emph{generic} bipartite graph is a bipartite graph with partition $\{X,Y\}$ such that for each two disjoint subsets $A,B$ of the same side we find a vertex in the other partition set with $A$ inside and $B$ outside its neighbourhood.

A digraph is a \emph{tree} if its underlying undirected graph is a tree.
It is \emph{regular} if all vertices have the same in-degree and all vertices have the same out-degree (but these two values need not coincide).

An undirected tree is \emph{semiregular} if for the canonical bipartition $\{X,Y\}$ of the vertices of the tree the vertices in~$X$ have the same degree and the vertices in~$Y$ have the same degree.
If the degree of the vertices in~$X$ is $k\in\nat^\infty$ and those in~$Y$ is $\ell\in\nat^\infty$, then we denote the semiregular tree by $T_{k,\ell}$.

Given $2\le m\in\N$ and a some $k\in\nat^\infty$ with $k \ge 3$ consider the tree $T_{k,m}$ and let $\{X,Y\}$ be its canonical bipartition such that the vertices in~$X$ have degree~$m$.
Subdivide each edge once and endow the neighbourhood of each $x\in X$ with a cyclic order.
For each new vertex $v$ let $x_v$ be its unique neighbour in~$X$ and denote by~$\sigma(v)$ the successor of~$v$ in $N(x_v)$.
Then for each $y\in Y$ and each $w\in N(y)$ we add an edge directed from $w$ to all $\sigma(u)$ with $u\in N(y)\sm\{w\}$.
Finally, we delete the vertices of the $T_{k,m}$ together with all edges incident with such a vertex to obtain the digraph $M(k,m)$.

For $2\leq m \in\N$ consider the tree $T_{2,2m}$ and let $\{X,Y\}$ be its canonical bipartition such that the vertices in~$X$ have degree~$2m$.
Subdivide every edge once and enumerate the neighbourhood of each $x\in X$ from $1$ to~$2m$ in a such way that the two neighbours of each $y\in Y$ have distinct parity.
For each new vertex $v$ let $x_v$ be its unique neighbour in~$X$ and define $\sigma(v)$ to be the successor of~$v$ in the cyclic order of $N(x_v)$.
For any $y\in Y$ we have a neighbour $a_y$ with even index, and a neighbour $b_y$ with odd index.
Then we add edges from both $a_y$ and $\sigma(a_y)$ to both $b_y$ and $\sigma(b_y)$.
Finally we delete the vertices of $T_{2,2m}$ together with all edges incident with such a vertex.
By $M'(2m)$ we denote the resulting digraph.

A tripartite digraph $D$ is a digraph whose vertex set can be partitioned into three sets $V_1,V_2,V_3$ such that
\[VE\sub (V_1\times V_2)\cup (V_2\times V_3)\cup (V_3\times V_1).\]
The \emph{directed tripartite complement} of~$D$ is the digraph
\[(VD,(\bigcup_{i=1,2,3}(V_i\times V_{i+1}))\sm ED),\]
where $V_4=V_1$.

For $k\in\nat^\infty$, let $Y_k$ be the digraph with vertex set $V_1 \cup V_2 \cup V_3$ where the~$V_i$ denote pairwise disjoint independent sets of the same cardinality~$k$ such that the induced subdigraphs $Y_k[V_i, V_{i+1}]$ with vertex sets $V_i \cup V_{i+1}$ (for $i=1,2,3$ with $V_4=V_1$) are complements of perfect matchings such that all edges are directed from $V_i$ to~$V_{i+1}$ and such that the directed tripartite complement of~$Y_k$ is the disjoint union of~$k$ copies of the directed triangle~$C_3$.

The digraph $\RF_m$ for $m\in\nat^\infty$ with $m\geq 3$ is constructed as follows:
take $m$ pairwise disjoint countably infinite sets $V_i$ for $i=1\ldots m$ if $m$ is finite and $i\in\ganz$ otherwise.
Then $\RF_m$ has vertex set $\bigcup V_i$ and edges only between $V_i$ and $V_{i+1}$ (with $V_{m+1}=V_1$) such that the digraph induced by $V_i$ and $V_{i+1}$ is a countable generic bipartite digraph such that the edges are directed from $V_i$ to~$V_{i+1}$.

We call a $2$-partite digraph $D$ with partition $\{X,Y\}$ a \emph{generic orientation of the countable generic bipartite graph} if for all finite $A,B,C\sub X$ (and all finite $A,B,C\sub Y$) there is a vertex $v\in X$ (a vertex $v\in Y$, respectively) with $A\sub N^+(v)$ and $B\sub N^-(v)$ and such that $v$ is not adjacent to any vertex of~$C$.
A back-and-forth argument shows that, up to isomorphism, there is a unique generic orientation of the countable generic bipartite graph.
It is easy to verify that the underlying undirected graph of~$D$ is the countable generic bipartite graph (see Section~\ref{sec_HomBipartiteGraphs} for the definition of a generic bipartite graph).

\medskip

For most of the digraphs in Theorem~\ref{thm_main}, we refer to their proof of the C-homo\-geneity to~\cite{FinConHomDigraphs,HH-ConHomDigraphs}.
In some cases this was only done for finite menbers of their class (e.g.\ in the case of~$Y_k$, this was done only for $k\in\nat$), but the proof for the infinite members of the classes is completely analogous.
The only digraphs of Theorem~\ref{thm_main} we have to consider here are the digraphs $S(3)[I_n]$, the digraphs $\RF_m$, and the generic orientation of the countable generic bipartite graph.
Whereas the latter is a direct consequence of the fact, that it is a homogeneous $2$-partite digraph that has an automorphism that switches its partition sets, we only have to consider the digraphs $\RF_m$ and $S(3)[I_n]$.
The fact that $S(3)[I_n]$ is C-homogeneous follows from the homogeneity of~$S(3)$ and in the case of $\RF_m$, it is an easy consequence of the fact that $\RF_m[V_i\cup V_{i+1}]$ is the countable homogeneous bipartite digraph and that two vertices in finite induced subdigraphs lie in the same set $V_i$ if and only if any path between them has the same number of forward and backward directed edges modulo~$m$.

\section{Overview of the proof of Theorem~\ref{thm_main}}\label{sec_Overview}

Let us give a very brief overview of the proof of Theorem~\ref{thm_main}.
The main lemma that we shall use throughout the proof is Lemma~\ref{lem_outNeighHomogeneous} which says that the out-neighbourhood of each vertex as well as the in-neighbourhood of each vertex induce a homogeneous digraph.
With this in mind, we consider Cherlin's classification of the countable homogeneous digraphs and investigate each of its cases one after another.
If the out-neighbourhood of some vertex is not an independent set (Section~\ref{sec_ProofPartI}), then we can prove the outcome of each case relatively easy.
Interestingly, some of the ideas of the proofs of the corresponding cases for undirected C-homogeneous graphs~\cite{GrayMacpherson} carries over but have to deal with the new situation of directed edges.
These cases are for example the generic $\HF$-free digraphs (versus generic $K_n$-free graphs), the generic $I_n$-free digraphs (versus generic $I_n$-free graphs), and the (semi-)generic $n$-partite digraphs (versus the complete $n$-partite graphs).

In Section~\ref{sec_ProofPartII}, we consider the case that the successors of each vertex form an independent set.
By considering the results of Section~\ref{sec_ProofPartI} for the digraph with all edges directed in the inverse way, also the predecessors of each vertex form an independent set.
In this situation, we can make use of the notion of the reachability relation by Cameron et al.~\cite{CPW}.
They showed (see Proposition~\ref{prop_CPW}) that either this equivalence relation is universal, or every equivalence class induces a bipartite digraph.
In the latter case (Section~\ref{sec_D+IndepAFBipartite}), we use the classification of the C-homogeneous bipartite (di-)graphs (Theorem~\ref{thm_HHThm6.4}) in analogy to the situation where $D$ is locally finite.
The case where the reachability relation is universal does not occur for locally finite digraphs; but for digraphs of infinite degree, there are such examples.
We treat this case in Section~\ref{sec_D+IndepAFUniversal}.
The main tool for that part is the classification of the homogeneous $2$-partite digraphs.

In the Sections~\ref{sec_ProofPartI} and~\ref{sec_ProofPartII} we prove that no other digraphs but those listed in Theorem~\ref{thm_main} are C-homogeneous.
The converse implication, that is, that all the digraphs in Theorem~\ref{thm_main} are indeed C-homogeneous, was already treated in Section~\ref{sec_MainThm}.

\section{Reachability relation}\label{sec_reachability}

Let $D$ be a digraph.
A \emph{walk} is a sequence $x_0\ldots x_k$ of vertices such that $x_i$ and $x_{i+1}$ are adjacent for all $0\leq i<k$.
If $x_{i-1}\in N^+(x_i)\Leftrightarrow x_{i+1}\in N^+(x_i)$ for all $0<i<k$ then the walk is called \emph{alternating}.
Two edges on a common alternating walk are \emph{reachable} from each other.
This defines an equivalence relation, the \emph{reachability relation}~$\AF$.
For an edge $e\in ED$, let $\AF(e)$ be the equivalence class of~$e$ and let $\langle\AF(e)\rangle$ be the \emph{reachability} digraph of~$D$ that contains~$e$, that is, the vertex set incident with some edge in~$\AF(e)$ and edge set $\AF(e)$.
If $D$ acts transitively on the edges of~$D$, that is, if $D$ is \emph{$1$-arc transitive}, then the digraphs $\langle\AF(e)\rangle$ are isomorphic for all $e\in ED$ and we denote by $\Delta(D)$ one digraph of their isomorphism class.\looseness-1

The following proposition is due to Cameron et al.

\begin{prop}\cite[Proposition 1.1]{CPW}\label{prop_CPW}
Let $D$ be a connected $1$-arc transitive digraph.
Then $\Delta(D)$ is $1$-arc transitive and connected.
Furthermore, either
\begin{enumerate}[{\em (a)}]
\item $\AF$ is the universal relation on $ED$ and $\Delta(D)\isom D$, or
\item $\Delta(D)$ is a bipartite reachability digraph.\qed
\end{enumerate}
\end{prop}

We say that a cycle $C$ \emph{witnesses that $\AF$ is universal} if $C$ contains an induced $2$-arc and if there is an edge $e$ on~$C$ such that $C$ without the edge $e$ is an alternating walk.\looseness-1

\begin{lem}\label{lem_CycleWitnesses}
Let $D$ be a non-empty vertex-transitive and $1$-arc transitive digraph whose reachability relation $\AF$ is universal.
Then $D$ contains a cycle that witnesses that $\AF$ is universal.
\end{lem}

\begin{proof}
As $D$ is non-empty, it contains some edge $xy$ and, since $D$ is vertex-transitive, it also has some edge $yz$.
Hence, $D$ contains a (not necessarily induced) $2$-arc $xyz$.
By universality of~$\AF$, there must be a minimal alternating walk $P$ in~$D$ whose first edge is~$xy$ and whose last edge is~$yz$.
Either this walk is a cycle or there is a vertex incident with at least three edges of that walk.
If the walk is a cycle, then it obviously witnesses that $\AF$ is universal.
If the walk contains a vertex $v$ incident with three edges of the walk, then one edge incident with~$v$ is directed towards~$v$ and one is directed away from~$v$, as otherwise we have a contradiction to the minimality of the alternating walk.
So $v$ is the middle vertex of two $2$-arcs $uvw$ and either $u'vw$ or $uvw'$ in the digraph $(VP,EP)$, say $u'vw$.
Then we find a shorter alternating walk -- a proper subwalk of~$P$ -- either between $uv$ and~$vw$ or between $u'v$ and~$vw$ and we are done by induction.
(Note that this is not necessarily a contradiction since, e.g., $xyz$ might be an induced $2$-arc but $uvw$ induces a triangle.)
\end{proof}

Lemma~\ref{lem_CycleWitnesses} just tells us that we find some cycle witnessing that $\AF$ is universal.
Next, we show that we can even find an induced cycle with the same property.

\begin{lem}\label{lem_UniversalInducedCycle}
Let $D$ be a non-empty vertex-transitive and $1$-arc transitive digraph whose reachability relation $\AF$ is universal.
If $D$ contains some cycle witnessing that $\AF$ is universal, then it contains an induced such cycle of at most the same length.
\end{lem}

\begin{proof}
Let us suppose that none of the minimal cycles witnessing the universality of~$\AF$ is induced.
Let $C$ be such a cycle of minimal length.
This exists by Lemma~\ref{lem_CycleWitnesses}.
Let $xy\in EC$ such that $C$ without the edge $xy$ is an alternating walk~$P$.
Since $C$ is not induced, it has a chord~$uv$.
If $u$ and~$v$ lie in the same set of the canonical bipartition of~$VP$, then the subwalk $uPv$ together with the edge $uv$ is a smaller cycle witnessing that $\AF$ is universal.
By minimality of~$C$, this cannot be.
So $u$ and $v$ lie in distinct sets of the canonical bipartition of~$P$.
But then we also find a smaller cycle in~$C$ together with the edge $uv$: if the out-degree of~$v$ in~$P$ is~$0$, then we take $uv$ together with the subwalk of~$C$ that contains $xy$, and otherwise we take $uv$ together with~$uPv$.
This contradiction to the minimality of~$C$ shows the lemma.\looseness-1
\end{proof}

\section{Some classification results of homogeneous structures}\label{sec_PreviousClassifications}

\subsection{(C-)Homogeneous bipartite graphs and digraphs}\label{sec_HomBipartiteGraphs}

In this section, we cite the classifications of the countable (C-)homogeneous bipartite graphs.
For countable (C-)homogeneous bipartite digraphs, then the analogous theorems hold.

A bipartite graph $G$ (with bipartition $\{X,Y\}$) is {\em homogeneous bipartite} if every isomorphism between two isomorphic finite induced subgraphs $A$ and~$B$ of~$G$ that preserves the bipartition (that means that $VA\cap X$ is mapped onto $VB\cap X$ and $VA\cap Y$ is mapped onto $VB\cap Y$) extends to an automorphism of~$G$ that preserves the bipartition.
We call $G$ {\em connected-homogeneous bipartite}, or simply {\em C-homogeneous bipartite}, if every isomorphism between two isomorphic finite induced connected subgraphs $A$ and~$B$ of~$G$ that preserves the bipartition extends to an automorphism of~$G$ that preserves the bipartition.
The same notions apply to bipartite and $2$-partite digraphs.

We begin with the classification of the homogeneous bipartite graphs.

\begin{thm}\label{thm_GGK}\cite[Remark 1.3]{GGK}
A countable bipartite graph is homogeneous if and only if it is isomorphic to one of the following graphs:
\begin{enumerate}[{\rm (i)}]
\item a complete bipartite graph;
\item an empty bipartite graph;
\item a perfect matching;
\item the bipartite complement of a perfect matching; or
\item the countable generic bipartite graph.\qed
\end{enumerate}
\end{thm}

The \emph{generic} bipartite graph is the bipartite graph $G$ with bipartition $\{X,Y\}$ such that for every two finite subsets $U_X,W_X\sub X$ and every two finite subsets $U_Y,V_Y\sub Y$ there exists $x\in X$ and $y\in Y$ with $U_X\sub N(y)$ and $V_X\cap N(y)=\es$ and with $U_Y\sub N(x)$ and $V_Y\cap N(x)=\es$.

The following theorem is the classification result of the countable C-homogeneous bipartite graphs.
Its proof is uses the just stated classification of the countable homogeneous bipartite graphs, Theorem~\ref{thm_GGK}.

\begin{thm}\label{thm_HHThm6.4}{\em \cite[Theorem 6.4]{HH-ConHomDigraphs}}
Let $G$ be a countable connected graph.
Then $G$ is C-homoge\-neous bipartite if and only if it is isomorphic to one of the following graphs:
\begin{enumerate}[{\em (i)}]
\item a cycle $C_{2m}$ for some $m\in\nat$ with $m\geq 2$;
\item an infinite semiregular tree $T_{k,\ell}$ for some $k,\ell\in\nat^\infty$ with $k,\ell\geq 2$;
\item a complete bipartite graph $K_{m,n}$ for some $m,n\in\nat^\infty$ with $m,n\geq 1$;
\item a complement of a perfect matching $CP_k$ for some $k\in\nat^\infty$ with $k\geq 3$; or
\item the countable generic bipartite graph.\qed
\end{enumerate}
\end{thm}

Note that Theorems~\ref{thm_GGK} and~\ref{thm_HHThm6.4} also apply to homogeneous and C-homo\-geneous bipartite digraphs, but not to $2$-partite digraphs.
The $2$-partite digraphs are in the case of homogeneity subject of the next section.

\subsection{Homogeneous $2$-partite digraphs}\label{sec_Hom2PartiteDigraphs}

As mentioned before, a $2$-partite digraph $D$ with partition $\{X,Y\}$ is \emph{homogeneous} if every isomorphism $\varphi$ between finite induced subdigraphs $A$ and~$B$ with $(VA\cap X)\varphi\sub X$ as well as $(VA\cap Y)\varphi\sub Y$ extends to an automorphism $\alpha$ of~$D$ with $X\alpha=X$ and $Y\alpha=Y$.
Let us state the classification of the countable $2$-partite digraphs:

\begin{thm}\label{thm_ClassHom2Partite}
\cite[Theorem~3.1]{Hom2PartiteDi}
Let $D$ be a countable $2$-partite digraph with partition $\{X,Y\}$.
Then $D$ is homogeneous if and only if one of the following cases holds:
\begin{enumerate}[\rm (i)]
\item $D$ is a homogeneous bipartite digraph;
\item $D\isom CP_k'$ for some $k\in\nat^\infty$ with $k\geq 2$;
\item $D$ is the countable generic $2$-partite digraph; or
\item $D$ is the generic orientation of the countable generic bipartite graph.\qed
\end{enumerate}
\end{thm}

For $k\in\nat^\infty$ with $k\geq 2$, let $CP_k'$ be the $2$-partite digraph with partition $\{X,Y\}$ such that $ECP_k'\cap (X\times Y)$ induces a $CP_k$ on~$VCP_k'$ and $ECP_k'\cap (Y\times X)$ induces a perfect matching on~$VCP_k'$.
Note that its underlying undirected graph is a complete bipartite graph.

We call a $2$-partite digraph $D$ with partition $\{X,Y\}$ \emph{generic} if for every finite $A,B\sub X$ (for every finite $A,B\sub Y$) there is a vertex $v\in Y$ (a vertex $v\in X$, respectively) with $A\sub N^+(v)$ and $B\sub N^-(v)$.
A back-and-forth argument shows that there is a unique countable generic $2$-partite digraph (up to isomorphism).
Similarly, we call a $2$-partite digraph $D$ with partition $\{X,Y\}$ a \emph{generic orientation of a generic bipartite graph} if for all pairwise disjoint finite subsets $A_X,B_X,C_X\sub X$ and $A_Y,B_Y,C_Y\sub Y$ there are vertices $y\in Y$ and $x\in X$ with $A_X\sub N^+(y)$, $B_X\sub N^-(y)$ and $C_X\sub y^\perp$ as well as with $A_Y\sub N^+(x)$, $B_Y\sub N^-(x)$ and $C_Y\sub x^\perp$.
It is easy to verify that its underlying undirected graph is a generic bipartite graph.

\subsection{Homogeneous digraphs}\label{sec_HomDigraphs}

In this section, we state Cherlin's classification of the countable homogeneous digraphs.

\begin{thm}\label{thm_CountHomDigraphs}\cite[5.1]{Cherlin-CountHomDigraphs}
A countable digraph is homogeneous if and only if it is isomorphic to one of the following digraphs:
\begin{enumerate}[{\rm (i)}]
\item $I_n$ for some $n\in\nat^\infty$;
\item $T[I_n]$ for some homogeneous tournament $T\neq I_1$ and some $n\in\nat^\infty$;
\item $I_n[T]$ for some homogeneous tournament $T\neq I_1$ and some $n\in\nat^\infty$;
\item the countable generic $\HF$-free digraphs for some set $\HF$ of finite tournaments;
\item the countable generic $I_n$-free digraphs for some integer $n\geq 3$;
\item $T^{\wedge}$ for some tournament $T\in \{I_1,C_3,\rat,T^\infty\}$;% (generalized version of the finite digraph $H=I_1^\wedge$);
\item the countable generic $n$-partite digraph for some $n\in\nat^\infty$ with $n\geq 2$;
\item the countable semi-generic $\omega$-partite digraph;
\item $S(3)$;
\item the countable generic partial order $\PF$; or
\item $\PF(3)$.\qed
\end{enumerate}
\end{thm}

The homogeneous tournaments are the already defined tournaments $I_1$, $C_3$, and $S(2)$ together with two more (see \cite[Theorem~3.6]{L-Tournaments}):
one is the generic tournament~$T^\infty$ that is the Fra\"iss\'e limit (see~\cite{Macpherson-Survey} for more on these limits) of all finite tournaments, so the unique homogeneous tournament that embeds all finite tournaments.
The remaining tournament is the tournament $\rat$ with vertex set $\rat$ and edges $xy$ if and only if $x<y$.

For a set $\HF$ of finite tournaments, the countable \emph{generic} $\HF$-free digraph is the Fra\"iss\'e limit of the class of all finite $\HF$-free digraph.
Similarly, for $n\in\nat$, the countable \emph{generic} $I_n$-free digraph is the Fra\"iss\'e limit of the class of all finite $I_n$-free digraphs and the countable \emph{generic} $n$-partite digraph is the Fra\"iss\'e limit of all orientations of finite complete $n$-partite graphs (where some partition classes may have no element).

The countable \emph{semi-generic} $\omega$-partite digraph is the Fra\"iss\'e limit of those finite complete $\omega$-partite digraphs that have the additional property that
\begin{txteq}\label{eq_omegaPartite}for each two pairs $(x_1,x_2),(y_1,y_2)$ from distinct classes, the number of edges from $\{x_1,x_2\}$ to~$\{y_1,y_2\}$ is even.
\end{txteq}

By $\PF$ we denote the countable \emph{generic} partial order, the Fra\"iss\'e limit of all finite partial orders.
Every partial order~$P$ is in a canonical way a digraph: for two elements $x,y$ of~$P$ we have $xy\in EP$ if and only if $x<y$.
We call digraphs that are obtained from partial orders in this way also partial orders.
Note that no partial order contains an induced $2$-arc.

It remains to define the variant $\PF(3)$ of~$\PF$. This digraph was first described in~\cite{Cherlin-CountHomDigraphs}.
A subset $X$ of~$V\PF$ is \emph{dense} if for all $a,b\in V\PF$ with $ab\in E\PF$ there is a vertex $c\in X$ with $ac,cb\in E\PF$.
Let $\{P_0,P_1,P_2\}$ be a partition of~$V\PF$ into three dense sets.
For this definition, let $x\perpp y$ if $x$ and~$y$ are not adjacent.
Let $\Hbb=(P_0,P_1,P_2)$ be the digraph on $V\PF$ such that for all $x,y\in P_i$ we have
\begin{alignat*}{2}
xy\in E\Hbb &\text{ if and only if } xy\in E\PF \\
\intertext{and such that for all $x\in P_i$ and $y\in P_{i+1}$ we have}
xy\in E\Hbb &\text{ if and only if } yx\in E\PF,\\
yx\in E\Hbb &\text{ if and only if } x\perpp y\in E\PF,\text{ and}\\
x\perpp y\in E\Hbb &\text{ if and only if } xy\in E\PF.
\end{alignat*}
Let $p$ be an element not in~$V\PF$.
Then $\PF(3)$ is the digraph on the vertex set $V\PF\cup\{p\}$ such that $(p^\perp,p^\rightarrow, p^\leftarrow)=\Hbb$, where
\begin{align*}
p^\perp&:=V\PF\sm N(p),\\
p^{\rightarrow}&:=N^+(p)\text{, and}\\
p^{\leftarrow}&:=N^-(p).
\end{align*}

\section{The case: $D^+\not\isom I_n\not\isom D^-$}\label{sec_ProofPartI}

%\subsection{General lemmas}

In this section we will investigate the situation that $D^+$ contains some edge.
Before we tackle this situation, we first show some general lemmas.
The following is our key lemma, which underlines our interest in the homogeneous digraphs:

\begin{lem}\label{lem_outNeighHomogeneous}\cite[Lemma~4.1]{FinConHomDigraphs}
For every C-homogeneous digraph $D$, the two digraphs $D^+$ and $D^-$ are homogeneous digraphs.\qed
\end{lem}

By this lemma, we are able to go through the list of countable homogeneous digraphs and look at each of them one by one, which is the general strategy for the proof of our main theorem.
\begin{comment}{
Each of these cases will have its own methods.
The presence of edges in $D^-$ or $D^+$ allows us to get some structure in~$D$ which helps us for the classification.
The absence of such edges forces us to use other and stronger methods.
We will look at this latter situation in Section~\ref{sec_ProofPartII}.
}\end{comment}

\begin{lem}\label{lem_OutNeighFinite}
Let $D$ be a countable connected C-homogeneous digraph with infinite out-degree.
Then either the in-degree is also infinite or $D^+$ is isomorphic to either $I_\omega$ or~$I_\omega[C_3]$.
\end{lem}

\begin{proof}
The claim follows directly from Theorem~\ref{thm_CountHomDigraphs} and Lemma~\ref{lem_outNeighHomogeneous}.
\end{proof}

As the locally finite C-homogeneous digraphs have already been classified~\cite{FinConHomDigraphs}, the previous lemma allows to concentrate (mostly) on digraphs with infinite $D^+$.

\begin{lem}\label{lem_Pentagon}
Let $D$ be a C-homogeneous digraphs such that it contains isomorphic copies of every orientation of~$C_5$. Then the diameter of~$D$ is~$2$.
\end{lem}

\begin{proof}
Since $D$ contains some orientation of~$C_5$, it contains two non-adjacent vertices.
Hence, the diameter of~$D$ is at least~$2$.
Let us suppose that $D$ does not have diameter~$2$.
Let $x$ and~$y$ be vertices of distance $3$ in~$D$ and $P$ be a shortest path between them.
Then there is an injection from~$P$ into one of the orientations of~$C_5$.
Let $C$ be a copy of this orientation in~$D$.
By C-homogeneity, we find an automorphism $\alpha$ of~$D$ that maps $P$ into~$C$.
Let $z$ be the vertex on~$C$ that is adjacent to the end vertices of~$P\alpha$.
Then $z\alpha\inv$ is adjacent to~$x$ and~$y$, which is a contradiction to the choice of these two vertices.
Thus, $D$ has diameter~$2$.
\end{proof}

\begin{lem}\label{lem_FourCycle}
Let $D$ be a C-homogeneous digraph such that it contains isomorphic copies of all orientations of~$C_4$.
Then each two non-adjacent vertices of~$D$ have a common successor and a common predecessor.
\end{lem}

\begin{proof}
Let $a$ and~$b$ be two non-adjacent vertices of~$D$.
By Lemma~\ref{lem_Pentagon} there is a vertex $x$ that is adjacent to~$a$ and~$b$.
Since every orientation of~$C_4$ embeds into~$D$, there is one such copy that has an isomorphic image of $D[a,x,b]$ in~$D$ such that the images of $a$ and~$b$ are the predecessors of the fourth vertex~$y$, and there is one such image such that the images of~$a$ and~$b$ are the successors of the fourth vertex~$y'$.
By C-homogeneity, we can map $D[a,x,b]$ onto such copies by automorphisms $\alpha, \beta$ of~$D$.
Then $y\alpha\inv$ and $y'\beta\inv$ verify the assertion.
\end{proof}

\subsection{\boldmath Generic $I_n$-free digraphs as~$D^+$}\label{sec_GenericInFree}

Throughout this section, let $D$ be a countable connected C-homogeneous digraph such that $D^+$ is isomorphic to the countable generic $I_n$-free digraph for some integer $n\geq 3$.
(Note that $n=2$ implies that $D^+$ is a tournament.
We consider this case in a later section.)
Our first step is to show that $D^+$ and $D^-$ are isomorphic.

\begin{lem}\label{lem_I_nFree+=-}
We have $D^+\isom D^-$.
\end{lem}

\begin{proof}
Let $F$ be any finite $I_n$-free digraph.
Then we find an isomorphic copy of~$F$ in~$D^+$ and, in addition, we find a vertex $x\in VD^+$ with $yx\in ED$ for all $y\in VF$.
Hence, $D^-$ contains an isomorphic copy of~$F$.

Since $D^-$ contains every finite $I_n$-free digraph, it is a direct consequence of Theorem~\ref{thm_CountHomDigraphs} that $D^-$ is either a generic $I_m$-free digraph for some $m\geq n$ or a generic $\HF$-free digraph with $\HF=\es$.
The latter or the first with $m>n$ is impossible since they contain a vertex with $n$ independent successors.
So $D^-$ is also the countable generic $I_n$-free digraph.
\end{proof}

Our next aim is to show that every finite induced $I_n$-free subdigraph of~$D$ lies in $D^+(x)$ for some $x\in VD$.
We do this in two steps and begin with the case that the subdigraph is some $I_m$ with $m<n$.

\begin{lem}\label{lem_GM13}
If $H\sub D$ is an isomorphic copy of~$I_m$ for some $m<n$, then there exist vertices $x,y\in VD$ with $H\sub D^+(x)$ and with $H\sub D^-(y)$.
\end{lem}

\begin{proof}
Note that $d^-\neq 0$.
So for $m=1$ the assertion is obvious and for $m=2$ it follows from Lemma~\ref{lem_FourCycle}.
%We proceed by induction.
Let $m\geq 3$ and let $H\isom I_m$ be a subdigraph of~$D$ with $VH=\{x_1,\ldots,x_m\}$.
By induction, we find $a,b\in VD$ with $\{x_1,\ldots,x_{m-1}\}\sub N^+(a)$ and $\{x_2,\ldots,x_m\}\sub N^+(b)$.
The digraph $F:=H+a+b$ is connected because of $m\geq 3$.
Since $F$ is $I_n$-free, $D^+(a)$ contains an isomorphic copy $F'$ of~$F$.
Applying C-homogeneity, we find an automorphism $\alpha$ of~$D$ that maps $F'$ to~$F$.
So we have $H\sub F'\sub D^+(a\alpha)$.

By an analogous argument, we find $y\in VD$ with $H\sub D^-(y)$.
\end{proof}

\begin{lem}\label{lem_GM14}
If $H\sub D$ is a finite induced $I_n$-free digraph, then there exist vertices $x,y\in VD$ with $H\sub D^+(x)$ and $H\sub D^-(y)$.
\end{lem}

\begin{proof}
If $H\isom I_m$ for some $m<n$, then the assertion follows from Lemma~\ref{lem_GM13}.
So we may assume that $H$ has a vertex~$a$ with $N^+(a)\cap VH\neq\emptyset$.
By induction, there is a vertex $u$ in~$D$ with $H-a\sub D^+(u)$.
Thus, $H+u$ is connected and $I_n$-free.
Applying an analogous argument as in the proof of Lemma~\ref{lem_GM13}, we find a vertex~$x$ in~$D$ with $H\sub H+u\sub D^+(x)$.

The existence of~$y$ follows analogously.
\end{proof}

Our next aim is to show that, for any two disjoint finite induced $I_n$-free digraphs $A$ and~$B$, we find a vertex~$x$ in~$D$ with $A\sub D^+(x)$ and $B\sub D^-(x)$.
We do not know whether this is true, even if we assume that $A$ is maximal $I_n$-free in $A+B$.
In particular, we need more structure on $D[N(x)]$ than we have till now.
But if we make the additional assumption that we find an isomorphic situation \emph{somewhere} in~$D$, that is, if we find subdigraphs $A'$ and $B'$ such that there exists an isomorphism $\varphi\colon A+B\to A'+B'$ with $A\varphi=A'$ and $B\varphi=B'$ and if $A'+B'$ has the claimed property, then we find such a vertex~$x$ without any further knowledge on the structure of $D[N(x)]$.

\begin{lem}\label{lem_I_nFreeSubdigraphPlusAdjacentVertex}
Let $A$ be a finite induced $I_n$-free subdigraph of~$D$ and let $z\in VD$ such that $z$ has a predecessor in~$A$.
Then there exists a vertex $x\in VD$ with $A\sub D^-(x)$ and $z\in N^+(x)$.
\end{lem}

\begin{proof}
Due to Lemma~\ref{lem_GM14}, we find a vertex $v\in VD$ with $A\sub D^+(v)$.
Let $a\in VA$ be a predecessor of~$z$.
Let $x_1,\ldots,x_{n-1}$ be $n-1$ independent vertices in $N^+(v)$ with $a'x_i\in ED$ for all $a'\in VA$.
These vertices exist as $D[A,x_1,\ldots,x_{n-1}]$ is $I_n$-free by construction and as $D^+$ is the generic $I_n$-free digraph.
All the vertices $z,x_1,\ldots,x_{n-1}$ lie in $N^-(a)$, so they cannot be independent.
By the choice of the~$x_i$, we know that $z$ must be adjacent to at least one of them, say $x_i$.
As $A+z$ is not $I_n$-free, we do not have $zx_i\in ED$.
Hence, we have $x_iz\in ED$ and $x_i$ is a vertex we are searching for.
\end{proof}

\begin{lem}\label{lem_TwoI_nFreeSubdigraphs}
Let $A,B,A',B'$ be finite induced $I_n$-free subdigraphs of~$D$ such that an isomorphism $\varphi\colon A'+B'\to A+B$ with $A'\varphi=A$ and $B'\varphi=B$ exists.
If $A$ is maximal $I_n$-free in~$A+B$ and if $D$ has a vertex $v$ with $A'\sub D^+(v)$ and $B'\sub D^-(v)$, then there exists $x\in VD$ with $A\sub D^+(x)$ and $B\sub D^-(x)$.
\end{lem}

\begin{proof}
If $A+B$ is connected, then the assertion is a direct consequence of C-homo\-geneity and, if $B$ has no vertex, then the assertion follows from Lemma~\ref{lem_GM14}.
So we may assume that there is some $z\in VB$.
Let $z'=z\varphi\inv$.

By induction, we find a vertex $w$ with $A\sub D^+(w)$ and $B-z\sub D^-(w)$.
Hence, we can map $A'+B'-z'+v$ onto $A+B-z+w$ by an automorphism $\alpha$ of~$D$ with $u\alpha=u\varphi$ for all $u\in V(A+B-z)$.
Taking $A'\alpha$, $B'\alpha$, $z'\alpha$, and $v\alpha$ instead of $A'$, $B'$, $z'$, and $v$ shows that
\begin{txteq}
we may assume $A'=A$ and $B'-z'=B-z$.
\end{txteq}

Let $u\in VA$ be in a component of $A+B$ that does not contain~$z$.
Because of $n\geq 3$ and $z'v\in ED$, the subdigraph $D[v,z,z']$ is $I_n$-free.
So by Lemma~\ref{lem_I_nFreeSubdigraphPlusAdjacentVertex} we find a vertex $y$ with $v,z,z'\in N^-(y)$ and $u\in N^+(y)$.
The digraphs $(A+y)+B$ and $(A+y)+B'$ are isomorphic and have less components than $A+B$.
As $A'+y\sub D^+(v)$ and $B'\sub D^-(v)$, we find $x\in VD$ with $A+y\sub D^+(x)$ and $B\sub D^-(x)$ by induction on the number of components, which finishes the proof.
\end{proof}

Now we are able to prove the main result of this section:

\begin{prop}
Let $D$ be a countable connected C-homogeneous digraph such that $D^+$ is the countable generic $I_n$-free digraph for some $n\geq 3$.
Then $D$ is homogeneous.
\end{prop}

\begin{proof}
Let $A$ and $B$ be two finite isomorphic induced subdigraphs of~$D$ and let ${\varphi\colon A\to B}$ be an isomorphism.
If $A$ is conntected, then $\varphi$ extends to an automorphism of~$D$ by C-homogeneity.
So let us assume that $A$ is not connected.
Let $A_1\sub A$ be maximal $I_n$-free with vertices from at least two distinct components of~$A$ and let $A_2\sub A-A_1$ be maximal $I_n$-free such that for some $x\in VD$ there is an isomorphic copy of~$D[A_1,A_2]$ in $D[N(x)]$ such that the image of~$A_1$ lies in~$D^+(x)$ and the image of~$A_2$ lies in~$D^-(x)$.
For $B_1:=A_1\varphi\sub B$ and $B_2:=A_2\varphi\sub B$, the corresponding statements hold.
According to Lemma~\ref{lem_TwoI_nFreeSubdigraphs}, we find two vertices $x_A$ and~$x_B$ with $A_1\sub D^+(x_A)$ and $A_2\sub D^-(x_A)$ and with $B_1\sub D^+(x_B)$ and $B_2\sub D^-(x_B)$.
Then $x_A$ has no neighbour in $A-(A_1+A_2)$ by the maximalities of~$A_1$ and~$A_2$ and, analogously, $x_B$ has no neighbour in $B-(B_1+B_2)$.
Hence, $\varphi$ extends to an isomorphism $\varphi'$ from $A+x_A$ to $B+x_B$ and these two subdigraphs of~$D$ have less components than $A$ and~$B$.
By induction on the number of components, $\varphi'$ extends to an automorphism of~$D$ and so does~$\varphi$.
\end{proof}

\subsection{\boldmath Generic $\HF$-free digraphs as $D^+$}\label{sec_GenericHFree}

In the following, let $D$ be a countable connected C-homogeneous digraph such that $D^+$ is the countable generic $\HF$-free digraph for some set $\HF$ of finite tournaments on at least three vertices.
(If we exclude the tournament on two vertices, then $D^+$ is an edgeless digraph. These will be investigated in Section~\ref{sec_ProofPartII}.)
In this section, we investigate the largest class of homogeneous digraphs: the class of the countable generic $\HF$-free digraphs contains uncountably many elements, as Henson~\cite{Henson-CountHom} proved, whereas all the other classes contain only countably many elements.

\begin{lem}\label{lem_DMinusForHFree}
There is a set $\HF'$ of finite tournaments on at least three vertices such that $D^-$ is the generic $\HF'$-free digraph.
\end{lem}

\begin{proof}
With a similar argument as in the proof of Lemma~\ref{lem_I_nFree+=-}, the assertion follows from Theorem~\ref{thm_CountHomDigraphs}.
\end{proof}

For the remainder of this section, let $\HF'$ be the finite set of tournaments we obtain from Lemma~\ref{lem_DMinusForHFree}.

Our next aim is to show that every finite induced $\HF$-free subdigraph of~$D$ lies in $D^+(x)$ for some $x\in VD$.

\begin{lem}\label{lem_CombineTournamentsHFree}
For every two disjoint finite induced $\HF$-free tournaments $A$ and~$B$ in~$D$, there exists a
vertex $x$ with $A+B\sub D^+(x)$.
%finite subdigraph $C$ of~$D$ such that $A+B+C$ is connected and $\HF$-free.
\end{lem}

\begin{proof}
If $|VA|=1=|VB|$, then the assertion follows directly from Lemma~\ref{lem_FourCycle}, because $D^+$ embeds every orientation of~$C_4$.
So we may assume $|VA|\geq 2$ and $|VA|\geq|VB|$.
Let $a\in VA$ such that $a$ has a successor in $A^-:=A-a$.
By induction on $|VA|+|VB|$, we find a vertex $v$ with $A^-+B\sub D^+(v)$.
Since $D^+$ is generic $\HF$-free, there is a vertex $w\in N^+(v)$ that has precisely one successor~$a'$ in~$A^-$ and one successor~$b$ in~$B$.
If $a$ and $w$ are not adjacent, then $A+B+w$ is connected and $\HF$-free.
Hence, the out-neighbourhood of some vertex of~$D$ contains an isomorphic copy of~$A+B+w$ and, by C-homogeneity, there exists a vertex $x$ with $A+B+w\sub D^+(x)$.
So we assume in the following that $w$ and~$a$ are adjacent.
Note that the only triangle in $A+B+w$ is the transitive triangle $D[a,a',w]$.
Hence, if $\HF$ does not contain the transitive triangle, then $A+B+w$ is $\HF$-free and connected and we find a vertex~$x$ with $A+B+w\sub D^+(x)$.
So we assume for the remainder of this proof that $\HF$ contains the transitive triangle.

First, we consider the case $|VA|=2$ and $|VB|=1$.
If $wa\in ED$, then $A+B\sub D^+(w)$ and $w$~is a vertex we are searching for.
If $aw\in ED$, let $w'\in N^+(w)$ with $w'a',w'b\in ED$, which exists as $D^+(w)$ is generic $\HF$-free.
If $aw'\in ED$, then $D[w,w',a']$ is a transitive triangle that lies in~$D^+(a)$, which is impossible by the choice of~$\HF$.
Hence, either $w'a\in ED$ or $a$ and $w'$ are not adjacent, and the assertion follows as before, just with~$w$ instead of~$w'$.

The next case that we look at is $|VA|=2=|VB|$.
Let $b'$ be the second vertex in~$B$.
As $D^+(v)$ is generic $\HF$-free, we find a vertex $c\in N^+(v)$ with $cb'\in ED$ but that is adjacent to neither $a'$ nor~$b$.
If $a$ and~$c$ are adjacent, then $D[a,a',c,b',b]$ is connected and $\HF$-free, so we find $x\in VD$ with $A+B+c\sub D^+(x)$.
Thus, let us assume that $a$ and~$c$ are not adjacent.
Then let $d\in N^+(v)$ with $da',dc\in ED$ and $b,b'\notin N(d)$, which exists as $D^+(v)$ is generic $\HF$-free.
If $a$ and~$d$ are not adjacent, then $D[a,a',d,c,b',b]$ is connected and $\HF$-free, so we find $x\in VD$ with $A+B+c+d\sub D^+(x)$ as before.
Hence, we may assume that $a$ and~$d$ are adjacent.
Considering the edge between $b$ and~$b'$ and the edge between $a$ and~$d$, we find by induction a vertex~$v'$ with $a,b',d\in N^+(v')$.
The connected subdigraphs $D[a',d,v',b',b]$ and $D[a',a,v',b',b]$ are isomorphic, so we find by C-homogeneity and automorphism $\alpha$ of~$D$ that fixes $a',v',b'$, and~$b$ and maps $d$ to~$a$.
Then $A+B=(A^-+B+d)\alpha\sub D^+(v\alpha)$ proves the assertion in this case.

The only remaining case is $|VA|\geq 3$.
Let $\hat{a}$ be a vertex in $N^+(v)$ such that there exists an isomorphism from $A^-+B+\hat{a}$ to $A+B$ that fixes $A^-+B$.
This vertex exists as $A+B$ is $\HF$-free and hence has an isomorphic copy in the generic $\HF$-free digraph $D^+(v)$.
As $D^+(v)$ is homogeneous, we then may assume that this copy conincides with $A+B$ on $A^-+B$.
Note that we may have chosen $w$ such that $w$ and $\hat{a}$ are not adjacent.
Let $c\in VA^-$ be a vertex that is not adjacent to~$w$.
If $a$ and~$\hat{a}$ are adjacent, let $F=D[\hat{a},a,w,b]$ and let $F=D[\hat{a},c,a,w,b]$ otherwise.
Then $F$ is connected contains no triangle, so it is $\HF$-free and we find a vertex $x$ with $F\sub D^+(x)$.
Then there is an isomorphism from $A^-+B+x+\hat{a}$ to $A+B+x$ that fixes $A^-+B+x$.
This isomorphism extends to an automorphism $\alpha$ of~$D$ by C-homogeneity.
Then $A+B=(A^-+B+\hat{a})\alpha\sub D^+(v\alpha)$ shows the remaining case of the lemma.\looseness-1
\end{proof}

\begin{lem}\label{lem_ConHFree}
For every finite induced $\HF$-free subdigraph  $A$ of~$D$, there exists a vertex $x$ with $A\sub D^+(x)$.
\end{lem}

\begin{proof}
If $A$ is connected, then we find an isomorphic copy of~$A$ in some $D^+(y)$, as $D^+$ is generic $\HF$-free.
So C-homogeneity implies the assertion.
Next, let us assume that $A$ has precisely two components $A_1$ and~$A_2$.
If both these components are tournaments, then Lemma~\ref{lem_CombineTournamentsHFree} implies the assertion.
So we may assume that $A_1$ has two non-adjacent vertices $a_1$ and~$a_2$.
Furthermore, we may assume that $A_1^-:=A_1-a_1$ is connected.
By induction, there exists a vertex $v\in VD$ with $A_1^-+A_2\sub D^+(v)$.
As $D^+(v)$ is generic $\HF$-free, we find a vertex $w\in N^+(v)$ with precisely one neighbour in~$A_2$ and such that $a_2$ is its only neighbour in~$A_1^-$.
As $a_1$ and~$a_2$ are not adjacent, the digraph $A+w$ is connected and $\HF$-free.
So we find a vertex $x$ of~$D$ with $A\sub A+w\sub D^+(x)$.

Let us now assume that $A$ consists of more than two components $A_1,\ldots,A_n$ with $n\geq 3$.
Let $a\in VA_1$.
By induction, we find a vertex $v\in VD$ with $A-a\sub D^+(v)$.
As $D^+(v)$ is generic $\HF$-free, there is a vertex $w\in N^+(v)$ that has no neighbour in~$A_1-a$ and precisely one neighbour in each $A_i$ for $2\leq i\leq n$.
Then $A+w$ is $\HF$-free and has at most two components.
By the previous cases, we find a vertex $x$ with $A\sub A+w\sub D^+(x)$ as claimed.
\end{proof}

Note that we also obtain with the same arguments as in the proofs of Lemma~\ref{lem_CombineTournamentsHFree} and~\ref{lem_ConHFree} that for every finite induced $\HF'$-free subdigraph $A$ of~$D$ we find some $x\in VD$ with $A\sub D^-(x)$.

\begin{lem}\label{lem_ConHH'Free}
Let $A$ and~$A'$ be finite induced $\HF$-free subdigraphs of~$D$ and let $B$ and~$B'$ be finite induced $\HF'$-free subdigraphs of~$D$ such that an isomorphism $\varphi\colon A'+B'\to A+B$ with $A'\varphi=A$ and $B'\varphi=B$ exists.
If $A$ is maximal $\HF$-free in~$A+B$ and if $D$ has a vertex $v$ with $A'\sub D^+(v)$ and $B'\sub D^-(v)$, then there exists a vertex $x\in VD$ with $A\sub D^+(x)$ and $B\sub D^-(x)$.
\end{lem}

\begin{proof}
If $A+B$ is connected, then the assertion is a direct consequence of C-homogeneity and, if $|VB|=0$, then the assertion follows from Lemma~\ref{lem_ConHFree}.
So let us assume that $A+B$ is not connected and that $B$ has some vertex~$z$.
Let $z'=z\varphi$.
As in the proof of Lemma~\ref{lem_TwoI_nFreeSubdigraphs},
\begin{txteq}
we may assume $A'=A$ and $B'-z'=B-z$.
\end{txteq}
By maximality of~$A$ in $A+B$ being $\HF$-free, we conclude $z\notin N^+(v)$ and that $A$ contains from each component of~$A+B$ at least one vertex.
Let $a\in VA$ be in a component of $A+B$ that does not contain~$z$.
Note that we may assume $z\notin N^-(v)$, as otherwise $v$ is a vertex we are searching for.
Hence, $z$ and~$v$ are not adjacent.
So $D[a,v,z,z']$ is $\HF'$-free and we find $y\in VD$ with $D[a,v,z,z']\sub D^-(y)$ due to the corresponding statement of Lemma~\ref{lem_ConHFree} for $\HF'$ instead of~$\HF$.
Because of $vy\in ED$, we know that $A'+y$ is $\HF$-free.
Note that there exists an isomorphism from $(A+y)+B$ to $(A'+y)+B'$ extending~$\varphi$ and that $A'+y\sub D^+(v)$ and $B'\sub D^-(v)$.
By induction on the number of components of~$A+B$, we find a vertex $x\in VD$ with $A+y\sub D^+(x)$ and $B\sub D^-(x)$.
This shows the assertion.
\end{proof}

Now we are ready to prove the main result of this section:

\begin{prop}
Let $D$ be a countable connected C-homogeneous digraph such that $D^+$ is the countable generic $\HF$-free digraph for some set~$\HF$ of finite tournaments.
Then $D$ is homogeneous.
\end{prop}

\begin{proof}
Let $A$ and $B$ be two isomorphic finite induced subdigraphs of~$D$ and let ${\varphi\colon A\to B}$ be an isomorphism.
Let $A^+$ be a maximal induced $\HF$-free subdigraph of~$A$.
Note that $A^+$ contains at least one vertex from each component of~$A$.
Let $A^-\sub A-A^+$ be maximal $\HF'$-free such that for some vertex $v$ of~$D$ there exists an embedding $\psi\colon A^++A^-\to D[N(v)]$ with $A^+\psi\sub D^+(v)$ and $A^-\psi\sub D^-(v)$.
According to Lemma~\ref{lem_ConHH'Free}, there is a vertex $x\in VD$ with $A^+\sub D^+(x)$ and $A^-\sub D^-(x)$.
By the maximimal choices of $A^+$ and $A^-$, we conclude that $x$ is not adjacent to any vertex of $A$ outside $A^++A^-$.
Let $B^+=A^+\varphi$ and $B^-=A^-\varphi$.
By the same argument as above, there is also a vertex $y$ with $B^+\sub D^+(y)$ and $B^-\sub D^-(y)$ such that no other vertex of~$B$ is adjacent to~$y$.
So $\varphi$ extends to an isomorphism $\varphi'$ from $A+x$ to $B+y$.
Since $A+x$ is connected, we can extend $\varphi'$, and thus also $\varphi$, to an automorphism of~$D$ by C-homogeneity.
\end{proof}

\subsection{\boldmath Generic $n$-partite or semi-generic $\omega$-partite digraph as $D^+$}

Within this section, let us assume that $D$ is a countable connected C-homogeneous digraph such that $D^+$ is either a countable generic $n$-partite digraph for some $n\in\nat^\infty$ with $n\geq 2$ or the countable semi-generic $\omega$-partite digraph.

\begin{lem}\label{lem_SemGenPartiteD+=D-}
We have $D^+\isom D^-$.
\end{lem}

\begin{proof}
First, let us assume that $D^+$ is either generic $n$-partite for some $n\geq 3$ or semi-generic $\omega$-partite.
Since for every $k<n$ every finite complete $k$-partite digraph (with the property~(\ref{eq_omegaPartite}) if $D^+$ is semi-generic $\omega$-partite) lies in~$D^-(y)\cap D^+(x)$ for some edge $xy\in ED$, we conclude from Theorem~\ref{thm_CountHomDigraphs} that $D^-$ is either a countable generic $\HF$-free digraph for some set $\HF$ of finite tournaments or a countable generic $m$-partite digraph for some $m\geq n-1$ or the countable semi-generic $\omega$-partite digraph.
The first digraph is excluded by Section~\ref{sec_GenericHFree}.

If $D^+$ is generic $n$-partite, then we can also exclude the countable semi-generic $\omega$-partite digraph for~$D^-$, since $D^-$ contains \emph{every} finite complete $k$-partite digraph.
For $xy\in ED$, we find some $(k+1)$-partite digraph in $D^-(y)$: the digraph $A+x$ where $A$ is an arbitrary complete $k$-partite digraph in $D^+(x)\cap D^-(y)$.
Hence, we have $m\geq n$ and by symmetry we also have $n\geq m$, so $D^+\isom D^-$.

If $D^+$ is semi-generic $\omega$-partite, then we find for every $k<\omega$ some complete $k$-partite digraph in~$D^-$, so $D^-$ is either generic or semi-generi $\omega$-partite.
We exclude the first possibility by our previous situation.
Thus, we have also $D^+\isom D^-$ in this case.

Now we consider the remaining situation, that is, that $D^+$ is the countable generic $2$-partite digraph.
Then, for every edge $xy\in ED$, the digraph $D^-(y)$ contains the complete $2$-partite digraph with $x$ on one side and with infinitely many successors of~$x$ on the other side.
Due to Theorem~\ref{thm_CountHomDigraphs}, we conclude that the only possibilities for $D^-$ are $\PF$, $\PF(3)$, $T[I_\omega]$ for some homogeneous tournament $T\neq I_1$, the generic $\HF$-free digraphs, which are excluded by Section~\ref{sec_GenericHFree}, or the (semi-)generic $n$-partite digraph, which must be the generic $2$-partite digraph due to our previous situations.
If $D^-$ is either~$\PF$ or~$\PF(3)$, then $D^-$ has a vertex with three successors in~$D^-$ that induce an edge with an isolated vertex.
Since this digraph does not lie in the countable generic $2$-partite digraph, $D^-$ is neither $\PF$ nor~$\PF(3)$.\looseness-1

If $D^-\isom T[I_\omega]$ for an infinite homogeneous tournament~$T$, then $D^+$ contains an arbitrarily large tournament, which cannot lie in any $2$-partite digraph.
Let us suppose $D^-\isom C_3[I_\omega]$.
Let $x\in VD$ and $D[v_1,v_2,v_3]$ be a directed triangle in~$D^-(x)$.
Considering $D^+(v_i)$, we know that $v_i$ has successors in precisely one set of the $2$-partition of~$D^+(x)$.
Hence for two $v_i$, these sets coincide.
Applying C-homogeneity to fix $x$ and rotate $D[v_1,v_2,v_3]$ by an automorphism of~$D$, we conclude that these sets coincide for all~$v_i$ and, applying C-homogeneity once more, we know that the same holds for all directed triangles in~$D^-(x)$.
Thus, all vertices in $N^-(x)$ have their successors in~$N^+(x)$ in the same partition set of $D^+(x)$, which contradicts C-homogeneity, as we can fix $x$ and map one vertex of~$N^+(x)\cap N^+(v_1)$ onto one of its successors in~$D^+(x)$ by an automorphism of~$D$ since $D$ is C-homogeneous.
So we have $D^-\not\isom C_3[I_\infty]$.
Hence, we have shown the assertion in this case, too.
\end{proof}

Now we are able to prove the main result of this section:

\begin{prop}
Let $D$ be a countable connected C-homogeneous digraph such that $D^+$ is either the countable generic $n$-partite digraph for some $n\in\nat^\infty$ with $n\geq 2$ or the countable semi-generic $\omega$-partite digraph.
Then $D$ is homogeneous.
\end{prop}

\begin{proof}
Let $x\in VD$ and $a,b\in N^+(x)$ with $ab\in ED$.
As $D^-\isom D^+$ holds by Lemma~\ref{lem_SemGenPartiteD+=D-}, we have
\begin{equation}\label{eq_nPartiteMain1}
N^-(b)\sm N(x)\sub N(a).
\end{equation}
Note that all partition sets of~$D^-(b)$ except for the one containing~$x$ have elements in~$N^+(x)$.
A direct consequence is the following:
\goodbreak
\begin{txteq}\label{eq_nPartiteMain1Concl}
For every maximal tournament in $D^+(x)$ that contains~$b$ and has no edge directed away from~$b$, this tournament has vertices of each partition set of~$D^-(b)$ except for the one containing~$x$.
\end{txteq}

Let us show that also
\begin{equation}\label{eq_nPartiteMain2}
N^+(b)\sm N(x)\sub N(a)
\end{equation}
holds.%
\begin{comment}{Dies f\"ur den Fall, dass doch aufgespalten werden muss.
For this, we distinguish the cases whether $D^+$ is generic or semi-generic $n$-partite.
}\end{comment}
Let us suppose that (\ref{eq_nPartiteMain2}) does not hold.
Then we find $y\in N^+(b)$ that is adjacent to neither $a$ nor~$x$.
As an induced directed cycle of length~$4$ embeds into~$D^+$, C-homogeneity implies the existence of a vertex $c\in N^+(y)\cap N^-(a)$ such that $b$ and~$c$ are not adjacent and, furthermore, we find a vertex $z\in VD$ with $D[a,b,y,c]\sub D^+(z)$ by C-homogeneity.
The structure of~$D^-(a)$ implies that $x$ is adjacent to either $c$ or~$z$.
First, let us assume that $x$ and~$z$ are adjacent.
Since $D[a,x,y]$ does not embed into~$D^+(z)$, we have $xz\in ED$ and, as $D[a,b,z]$ is a triangle in~$D^+(x)$, we have $n\geq 3$ if $D^+$ is generic $n$-partite.
Let $\{v_i\mid i\in I\}$ be a maximal set in~$N^+(z)$ such that $X:=\{a,b,v_i\mid i\in I\}$ induces a tournament and such that $D[a,b,c,y]\sub D^+(v_i)$ for all $i\in I$.
By its maximality and due to the structure of~$D^+$, the set $X$ contains vertices from each maximal independent set in~$N^+(z)$.
Due an analogue of~(\ref{eq_nPartiteMain1Concl}) for $z$ instead of~$x$, we know that $X$ meets every maximal independent set of $N^-(b)$ but the one that contains~$z$.
So $x$ must be non-adjacent to some~$v_i$.
As $D[v_i,c,x]\sub D^-(a)$, we conclude that $x$ and~$c$ are adjacent.
So if we replace $z$ by~$v_i$ if necessary, we may assume that $x$ and~$c$ are adjacent but $x$ and~$z$ are not.

Because $D[x,y,z]$, a digraph on three vertices with precisely one edge, cannot lie in~$D^-(c)$, we have $xc\notin ED$.
So $cx\in ED$ and $D[x,b,y,c]$ is an induced directed cycle.
As $C_4$ embeds into~$D^+$, we find $z'\in VD$ with $D[x,b,y,c]\sub D^+(z')$ by C-homogeneity.
Considering $D^-(b)$, we conclude that $z$ and~$z'$ are adjacent.
The corresponding edge is not $z'z$, as $D[x,y,z]$ cannot lie in~$D^+(z')$.
Hence, we have $zz'\in ED$.
Because $D[a,y,z']$ lies in~$D^+(z)$, we know that $a$ and~$z'$ are adjacent and, as $D[a,x,y]$ cannot lie in~$D^+(z')$, the corresponding edge must be~$az'$.
Since $D^+(z)$ contains the triangle $D[b,y,z']$, we have $n\geq 3$ if $D^+$ is generic $n$-partite.
Similarly as above, we choose a maximal set $\{w_i\mid i\in I\}$ in~$N^+(z')$ such that the set $X=\{b,y,w_i\mid i\in I\}$ induces a tournament and such that $D[b,y,c,x]\sub D^+(w_i)$ for all $i\in I$.
By its maximality, the set~$X$ contains vertices from each maximal independent set in~$N^+(z')$.
Then an analogue of~(\ref{eq_nPartiteMain1Concl}) for~$z'$ instead of~$x$ implies that $X$ meets every maximal independent set of $N^-(b)$ but the one that contains~$z'$.
So $a$ must be non-adjacent to some~$w_i$ and $z$ is adjacent to every~$w_j$, in particular to~$w_i$.
But $zw_i\in ED$ is impossible, as $D[a,w_i,y]$ does not embed into~$D^+(z)$, and $w_iz\in ED$ is impossible, as $D[x,y,z]$ does not embed into~$D^+(w_i)$.
This contradiction proves~(\ref{eq_nPartiteMain2}).

Now we have shown $N(b)\sm N(x)\sub N(a)$.
For an induced directed cycle $x_1x_2\ldots x_m$ (with $m\leq 5$) in~$N^+(x)$ with $x_{m-1}=a$ and $x_1=b=x_m$, we use C-homogeneity to find an automorphism that fixes~$x$ and rotates the cycle backwards so that we can conclude inductively
\[N(x_m)\sm N(x)\sub N(x_{m-1})\sm N(x)\sub\ldots\sub N(x_2)\sm N(x)\sub N(x_1)\sm N(x).\]
Because of $x_1=x_m$, all inclusions are equalities of the involved sets.
In particular, we have $N(a)\sm N(x)=N(b)\sm N(x)$.
Note that any two vertices in~$N^+(x)$ lie on an induced directed cycle of length at most~$4$.
Hence, we can apply the above argument and obtain
\begin{equation}\label{eq_nPartiteMain3}
N(u)\sm N(x)=N(v)\sm N(x)\text{ for all }u,v\in N^+(x).
\end{equation}
By symmetry and as $D^+\isom D^-$ due to Lemma~\ref{lem_SemGenPartiteD+=D-}, we have
\begin{equation}\label{eq_nPartiteMain3'}
N(u)\sm N(x)=N(v)\sm N(x)\text{ for all }u,v\in N^-(x).
\end{equation}

Let us show for $A:=N(a)\sm N(x)$ the following:
\begin{txteq}\label{eq_nPartiteMain4}
$A$ is an independent set.
\end{txteq}
Let us suppose that there are two vertices $u,v\in A$ with $uv\in ED$.
Note that $b$ is adajcent to~$u$ and~$v$ by~(\ref{eq_nPartiteMain3}).
We find $w\in N^+(u)\cap N^+(v)$.
The analogue of~(\ref{eq_nPartiteMain3}) for~$u$ instead of~$x$ gives us $N(v)\sm N(u)=N(w)\sm N(u)$, which shows that $w$ is not adjacent to~$x$.
If $av\in ED$, then we obtain a contradiction to an analogue of~(\ref{eq_nPartiteMain3'}) %for the triple $(v,u,a)$ 
as $x$ lies in $N(a)\sm N(v)$ but not in $N(u)\sm N(v)$.
Thus, we have $va\in ED$ and we conclude $vb\in ED$ analogously.
Due to the structure of $D^+(v)$ we know that $w$ has to be adjacent to either~$a$ or~$b$.
First, let us assume that $a$ and~$w$ are adjacent.
If $aw\in ED$, then we conclude $x\in N(a)\sm N(w)=N(v)\sm N(w)$ by an analogue of~(\ref{eq_nPartiteMain3'})% for $(w,v,a)$
, which contradicts $v\in A$, and if $wa\in ED$, then $x$ is not adjacent to both end vertices of~$vw$, which is impossible in $D^-(a)$.
We obtain analogous contradictions if~$w$ and~$b$ are adjacent.
Hence, we have shown~(\ref{eq_nPartiteMain4}).

Let us show
\begin{equation}\label{eq_nPartiteMain5}
VD=A\cup N(x).
\end{equation}
First, let $y\in N(x)$ and let $u$ be a neighbour of~$y$.
If $u$ lies outside $N(x)$, then we find a vertex $v$ with $D[x,y,u]\sub D^-(v)$ due to C-homogeneity and as $D^-$ contains an isomorphic copy of $D[x,y,u]$.
So we conclude $u\in A$ due to~(\ref{eq_nPartiteMain3}).
Now let $y\in A$ and let $u$ be a neighbour of~$y$.
If $u$ is adjacent to~$a$, then $u\in A\cup N(x)$.
So let us assume that $a$ and~$u$ are not adjacent.
Then we find by C-homogeneity a vertex~$v$ with $D[a,y,u]\sub D^-(v)$.
As $v$ is adjacent to~$a$, it lies in $N(x)\cup A$ and as it is adjacent to~$y$, it cannot lie in~$A$ due to~(\ref{eq_nPartiteMain4}).
So $v$ lies in~$N(x)$ and by the first case we conclude that $u$ lies in $A\cup N(x)$.
This shows~(\ref{eq_nPartiteMain5}).

Our last step, before we show the homogeneity of~$D$, is to show that
\begin{txteq}\label{eq_nPartiteMain7}
$D$ is complete $m$-partite for some $m\in\nat^\infty$.
\end{txteq}
Let $\IF$ be the set of maximal independent sets in~$N^+(x)$.
Let $A'=A\cup\{x\}$ and, for every $I\in \IF$, let $I'$ be a maximal independent set in~$D$ that contains~$I$.
Due to~(\ref{eq_nPartiteMain3}), every vertex of~$A'$ is adjacent to all vertices of~$N^+(x)$.
As $D[x,a,a']$ with $a'\in A$ embeds into $D^-(x)$, we find by C-homogeneity a vertex~$v$ with $D[x,a,a']\sub D^-(v)$.
So every vertex of~$A'$ is adjacent to some vertex of $N^-(x)$ and hence by~(\ref{eq_nPartiteMain3'}) to every vertex of $N^-(x)$.
So by~(\ref{eq_nPartiteMain5}), every vertex of~$A'$ is adjacent to every vertex outside~$A'$.
As $D$ is vertex-transitive, the same holds for every maximal independent vertex set of~$D$.
Thus, (\ref{eq_nPartiteMain7}) holds.

To show that $D$ is homogeneous, let $F$ and~$H$ be two isomorphic induced subdigraphs of~$D$.
If they are connected, then C-homogeneity implies that every isomorphism from~$F$ to~$H$ extends to an automorphism of~$D$.
So we may assume that they are not connected.
As $D$ is complete $m$-partite, we conclude that $VF$ is an independent set and the same is true for~$VH$.
Then we find $u_F$ and $u_H$ with $VF\sub N(u_F)$ and $VH\sub N(u_H)$.
Note that due to the structure of~$D^+(x)$, we find subdigraphs $F'$ and $H'$ of~$D^+(x)$ that are isomorphic to~$F+u_F$ and $H+u_H$, respectively.
By C-homogeneity, we find an automorphism $\varphi_F$ of~$D$ that maps $F+u_F$ to~$F'$ and an automorphism $\varphi_H$ that maps $H+u_H$ to~$H'$.
Then $F+x\varphi_F\inv$ and $H+x\varphi_H\inv$ are connected and every isomorphism from~$F$ to~$H$ extends to an isomorphism from $F+x\varphi_F\inv$ to $H+x\varphi_H\inv$, so C-homogeneity implies the assertion.
\end{proof}

\subsection{\boldmath The digraphs $T^\wedge$ as $D^+$}

In this section, we investigate countable connected C-homogeneous digraphs $D$ with $D^+\isom T^\wedge$ for some $T\in\{I_1,C_3,\rat,T^\infty\}$.
If $T$ is either $I_1$ or~$C_3$, then we obtain from Lemma~\ref{lem_OutNeighFinite} that $D$ is locally finite and due to Lemmas~4.2 and 4.3 of~\cite{FinConHomDigraphs} we obtain that no such C-homogeneous digraph exists.
Hence, it suffices to consider only the cases $T\isom \rat$ and $T\isom T^\infty$ in the proof of Proposition~\ref{prop_TWedge}.

\begin{prop}\label{prop_TWedge}
No countable connected C-homogeneous digraph $D$ with $D^+\isom T^\wedge$ for any $T\in\{I_1,C_3,\rat,T^\infty\}$ exists.
\end{prop}

\begin{proof}
Let us suppose that some countable connected C-homogeneous digraph $D$ with $D^+\isom T^\wedge$ exists for some $T\in\{\rat, T^\infty\}$.
Note that it was already proven in~\cite{FinConHomDigraphs} that no such digraph exists if $T\in\{I_1,C_3\}$, as we have already mentioned earlier.
Due to Theorem~\ref{thm_CountHomDigraphs} and the previous sections, the only possibilities for $D^-$ are $I_n[T_0]$, $T_0[I_n]$, $S(3)$, $T_0^\wedge$, $\PF$, or $\PF(3)$, where $n\in\nat^\infty$ and $T_0$ is some homogeneous tournament.
Because the latter two digraphs contain the complete bipartite digraph $K_{1,3}$, but $T^\wedge$ contains no three independent vertices, we know that $D^-$ is one of the first four digraphs.
Since the first three digraphs in that list do not contain the digraph $D'$ depicted in Figure~\ref{pic_DPrime}, we have the following:
\begin{txteq}\label{eq_TWedge1}
if $D'$ embeds into $D^-$, then $D^-\isom T_0^\wedge$ for some infinite homogeneous tournament~$T_0$.
\end{txteq}
\begin{figure}[h]
\begin{center}
\includegraphics[width=.3\textwidth]{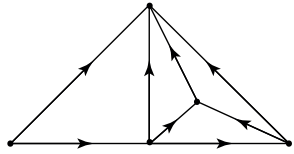}
\caption[Figure 1]{The digraph $D'$}\label{pic_DPrime}
\end{center}\end{figure}

Let $xy\in ED$.
Note that $D^+(x)\cap D^+(y)\isom T$.
The first statement that we shall show is the following:
\begin{txteq}\label{eq_TWedge2}
There is a unique pair of vertices $v,\mathring{v}$ in $D^+(y)$ that are not adjacent and each of which is not adjacent to~$x$.
\end{txteq}
For each $z\in N^+(y)$, let $\mathring{z}$ denote the unique vertex in~$D^+(y)$ that is not adjacent to~$z$.
For every $z\in N^+(x)\cap N^+(y)$, either $\mathring{z}\in N^-(x)$ or $\mathring{z}$ is not adjacent to~$x$, since $D^+(x)\cap D^+(y)$ is a tournament.
Let us suppose that $\mathring{z}$ is not adjacent to~$x$.
By C-homogeneity, the same holds for every $\mathring{u}$ with $u\in N^+(x)\cap N^+(y)$.
Let $u_1,u_2,u_3\in N^+(x)\cap N^+(y)$ with $u_iu_j\in ED$ for $i<j\leq 3$ and with $u_iz\in ED$ for all $i\leq 3$.
These vertices exist as every vertex of~$\rat$ and~$T^\infty$ contains the directed triangle in its in-neighbourhood, so the same holds for~$z$ in $D^+(x)\cap D^+(y)$.
The digraph $D[x,y,u_1,\mathring{z},\mathring{u_3}]$ is isomorphic to~$D'$ and lies in~$D^-(u_2)$.
Due to~(\ref{eq_TWedge1}), we have $D^-\isom T_0^\wedge$ for some infinite homogeneous tournament~$T_0$.
Hence, $D^-(u_2)$ contains a unique vertex that is not adjacent to~$x$ which contradicts the fact that $\mathring{z}$ and $\mathring{u}_3$ are not adjacent to~$x$ even though they lie in~$N^-(u_2)$.
This contradiction shows $\mathring{z}x\in ED$.
By C-homogeneity, we conclude that for any $w\in N^+(y)$ that is not adjacent to~$x$ also the vertex $\mathring{w}$ is not adjacent to~$x$.
Indeed, if not, then we have $\mathring{w}x\in ED$ by the previous situation.
Hence, some automorphism of~$D$ fixes $x$ and~$y$ and maps $\mathring{z}$ to~$\mathring{w}$ and we obtain $xw\in ED$, contrary to the choice of~$w$.
Since $D^+$ contains an induced $2$-arc, there is a vertex in~$N^+(y)$ that is not adjacent to~$x$, which shows the existence of a pair of vertices as described in~(\ref{eq_TWedge2}).
It remains to show that this pair is unique.

Let us suppose that $N^+(y)\sm N(x)$ contains two vertices $v,w$  with $vw\in ED$.
Among the vertices $v,\mathring{v},w$, and~$\mathring{w}$, we find two adjacent ones, say $v$ and~$w$ with ${vw\in ED}$ such that there are two vertices $u_1,u_2\in N^+(x)\cap N^+(y)$ with $u_1,u_2\in N^-(v)\cap N^-(w)$ and with $u_1u_2\in ED$.
The digraph $D[x,y,u_1,\mathring{v},\mathring{w}]$ is isomorphic to~$D'$ and lies in~$D^-(u_2)$.
Due to~(\ref{eq_TWedge1}), we have $D^-\isom T_0^\wedge$ for some infinite homogeneous tournament~$T_0$.
Note that $T_0^\wedge$ does not contain a subdigraph on three vertices with precisely one edge.
But $D[\mathring{v},\mathring{w},x]$ is such a digraph, which lies in $D^-(u_2)\isom T_0^\wedge$.
This contradiction shows the uniqueness of the vertex pair in~(\ref{eq_TWedge2}), as every maximal independent vertex set in~$D^+(y)$ has precisely two vertices.

Let $N=N^+(x)\cap N^+(y)$.
In the following, let $v$ and $\mathring{v}$ be the vertices of~(\ref{eq_TWedge2}).
Our next step is to show
\begin{equation}\label{eq_TWedge3}
N\sub N^+(v)\quad\text{or}\quad N\sub N^-(v).
\end{equation}

Let us suppose that we find vertices $a\in N^+(v)\cap N$ and $b\in N^-(v)\cap N$.
Note that $a$ and~$b$ are adjacent, since both lie in the tournament $D^+(x)\cap D^+(y)$.
Since $T$ contains a transitive triangle, let $c\in N$ such that $D[a,b,c]$ is a transitive triangle.
Then either $c\in N^+(v)$ or $c\in N^-(v)$.
If $c\in N^+(v)$, then we find an automorphism of~$D$ that fixes $x$ and $y$ and maps the edge between $a$ and~$b$ to the edge between $a$ and~$c$ by C-homogeneity.
If $c\in N^-(v)$, then we find an automorphism of~$D$ that fixes $x$ and~$y$ and maps the edge between $a$ and~$b$ to the edge between $b$ and~$c$.
Any of these automorphisms can neither fix $v$ nor map it to~$\mathring{v}$ even though its image must lie in $\{v,\mathring{v}\}$ by~(\ref{eq_TWedge2}).
This contradiction shows~(\ref{eq_TWedge3}).

By symmetry, we may assume $N\sub N^+(v)$ and hence $N\sub N^-(\mathring{v})$.
Since $D$ is C-homogeneous, we find an automorphism $\alpha$ of~$D$ that fixes $x$ and~$y$ and maps $v$ to~$\mathring{v}$.
Since $\alpha$ fixes $x$ and~$y$, we have $N\alpha=N$ and hence
\[
N\alpha=N\sub N^+(v)=(N^+(\mathring{v}))\alpha.
\]
Thus, we have $N\sub N^+(\mathring{v})$.
This is a contradiction to $N\sub N^-(\mathring{v})$, which shows the assertion.
\end{proof}

\subsection{\boldmath The digraph $S(3)$ as $D^+$}

In this section, we show that no countable connected C-homogeneous digraphs $D$ has the property $D^+\isom S(3)$.
Our strategy in the proof is to exclude all countable homogeneous digraphs for~$D^-$.

\begin{prop}\label{prop_D+S3}
No countable connected C-homogeneous digraph $D$ with $D^+\isom S(3)$ exists.
\end{prop}

\begin{proof}
Let us suppose that some countable connected C-homogeneous digraph $D$ with $D^+\isom S(3)$ exists.
Since $D^+\isom S(3)$, we have $D^+(x)\cap D^+(y)\isom\rat$ for every edge $xy\in ED$.
Let $v\in N^+(x)\cap N^+(y)$.
As $D^+$ contains a transitive triangle, C-homogeneity implies the existence of some $z\in VD$ with $D[x,y,v]\sub D^+(z)$.
In $D^+(z)$ we find a vertex $u$ with $u\in N^+(y)\cap N^+(v)$ that is not adjacent to~$x$.
By C-homogeneity, we can map $xyu$ onto any other induced $2$-arc $xya$ and obtain
\begin{txteq}\label{eq_D+S3_1}
$N^-(a)\cap N^+(x)\cap N^+(y)\neq\es$ for every $a\in N^+(y)\sm N(x)$.
\end{txteq}
As $D^+(x)\cap D^+(y)\isom\rat$ is a proper subdigraph of~$D^+(y)\isom S(3)$, we find a predecessor $w$ of~$v$ in~$N^+(y)$ that lies outside $N^+(x)$ and has only successors in $N^+(x)\cap N^+(y)$.
The vertices $x$ and~$w$ are adjacent due to~(\ref{eq_D+S3_1}).
As $w\notin N^+(x)$, we have $wx\in ED$.
Thus, $D^-(v)$ contains the directed triangle $D[x,y,w]$.

Note that $v$ has some predecessor $w'$ in $N^+(x)\cap N^+(y)$.
This vertex must be adjacent to~$w$ as each two predecessors of~$v$ are adjacent by the structure of~$S(3)$.
As $N^-(w)$ contains no vertex of $D^+(x)\cap D^+(y)$, we have $w'\in N^+(w)$.
Note that we also have $D[x,y,w,w']\sub D^-(v)$.

Since $D^-$ contains a copy of $D[x,y,w,w']$ and a copy of~$\rat$, Theorem~\ref{thm_CountHomDigraphs} implies that the only possibilities for $D^-$ are either $\PF(3)$, $I_n[T^\infty]$, or $T^\infty[I_n]$ for some $n\in\nat^\infty$ by the previous sections.
We cannot have $D^-\isom\PF(3)$, since $\PF(3)$ contains a vertex with three independent successors, but $D^+$ contains no independent set of three vertices.
So we have $D^-\isom I_n[T^\infty]$ or $D^-\isom T^\infty[I_n]$.
But then $D^-$ contains a vertex with a directed triangle in its out-neighbourhood.
This is impossible, since $S(3)$ contains no directed triangle.%
\begin{comment}{
Um $S(2)$ bei $I_n[T_0]$ und $T_0[I_n]$ auszuschliessen:
So we have $D^-\isom I_n[S(2)]$ or $D^-\isom S(2)[I_n]$.
Since $D[x,y,w,w']$ does not embed into any of them, we have no possibility left for $D^-$.
}\end{comment}
As no possibility is left for $D^-$, we have shown the assertion.
\end{proof}

\subsection{\boldmath The digraph $\PF(3)$ as $D^+$}

In this section, we show that no countable connected C-homogeneous digraph $D$ has the property $D^+\isom \PF(3)$.

\begin{prop}
No countable connected C-homogeneous digraph $D$ with $D^+\isom\PF(3)$ exists.
\end{prop}

\begin{proof}
Let us suppose that there is a countable connected C-homogeneous digraph~$D$ with $D^+\isom \PF(3)$.
Since the in-neighbourhood of any vertex contains every finite partial order, we have $D^-\isom \PF$ or $D^-\isom \PF(3)$.
Furthermore, we have $D^-(y)\cap D^-(x)\isom \PF$ for every edge $xy\in ED$.
As $D^+$ contains a directed triangle, C-homogeneity implies the existence of a vertex $a\in N^+(y)$ such that $D[x,y,a]$ is a directed triangle.
Let
\begin{align*}
a^\perp&:=\{b\in N^+(y)\mid a\text{ not adjacent to }b\},\\
a^{\rightarrow}&:=N^+(a)\cap N^+(y)\text{, and}\\
a^{\leftarrow}&:=N^-(a)\cap N^+(y).
\end{align*}
So we have $\Hbb(a):=(a^\perp,a^{\rightarrow},a^{\leftarrow})\isom\mathbb{H}$.
Note that $D^+(x)$ has an edge with both its incident vertices in the same set $a^\perp$, $a^\rightarrow$, or $a^\leftarrow$, as $D^+(x)\cap D^+(y)$ contains a tournament on four vertices.
If either $a^\perp$ or $a^\leftarrow$ contains an edge $uv$ of~$D^+(x)$, then we find an edge $u'v'$ in $D^-(u)\cap D^-(v)$ with $u',v'\in a^\rightarrow$ due to the structure of~$\PF(3)$.
If either $u'$ or~$v'$ does not lie in $N^+(x)$, then $xy$ together with this vertex induce either a $2$-arc or a directed triangle in $D^-(u)\cap D^-(v)\isom \PF$, which is impossible.
So we may assume that there are two adjacent vertices $b$ and~$c$ of~$N^+(x)$ in~$a^\rightarrow$.
Then $D[a,x,y]$ is a directed triangle in $D^-(b)\cap D^-(c)$, which is impossible.
\end{proof}

\subsection{\boldmath Generic partial order $\PF$ as $D^+$}

Within this section, let $D$ be a countable connected C-homogeneous digraph with $D^+\isom\PF$.
Before we are able to prove that $D$ is homogeneous in this situation, we will prove several lemmas.
Our first one determines~$D^-$.

\begin{lem}\label{lem_D+=D-=PF}
We have $D^-\isom\PF$.
\end{lem}

\begin{proof}
Since, for every edge $xy\in ED$, the digraph $D^+(x)\cap D^-(y)$ contains every finite partial order, the assertion follows from Theorem~\ref{thm_CountHomDigraphs} together with the previous sections.
\end{proof}

Our general strategy to prove that $D$ is homogeneous is similar to those of the Sections~\ref{sec_GenericInFree} and~\ref{sec_GenericHFree}.
In particular, one step is to show that every finite partial order in~$D$ lies in $D^+(x)$ for some $x\in VD$ (Lemma~\ref{lem_FinPOinD+}).
As in the other two cases, we prove it by induction.
In this situation, the base case (Lemma~\ref{lem_GenPO2Vertices}) turns out to be the most complicated part of the proof.

\begin{lem}\label{lem_GenPO2Vertices}
Any two vertices in~$D$ have a common predecessor.
\end{lem}

\begin{proof}
If $D$ contains no induced $2$-arc, then any induced path is an alternating walk and lies in the out-neighbourhood of some vertex by C-homogeneity.
Hence, any two vertices have a common predecessor.

Thus, we assume that $D$ contains induced $2$-arcs.
Our first aim is to show that
\begin{txteq}\label{eq_GenPO2Vertices1}
the end vertices of any induced $2$-arc have a common predecessor or a common successor.
\end{txteq}
In order to prove (\ref{eq_GenPO2Vertices1}) we investigate for $xy\in ED$ the three sets:
\begin{align*}
x^\perp&:=\{z\in N^+(y)\mid x\text{ not adjacent to }z\},\\
x^\rightarrow&:=N^+(y)\cap N^+(x),\text{ and}\\
x^\leftarrow&:=N^+(y)\cap N^-(x).
\end{align*}
If $ba\in ED$ for some $a\in x^\rightarrow$ and some $b\in x^\perp$, then $xyb$ is an induced $2$-arc in $D^-(a)$.
As $D^-\isom \PF$ by Lemma~\ref{lem_D+=D-=PF} and $\PF$ contains no induced $2$-arc, we have shown:
\begin{txteq}\label{eq_GenPO2Vertices1a}
no vertex in~$x^\perp$ has successors in~$x^\rightarrow$.
\end{txteq}

If $ba\in ED$ for some $a\in x^\rightarrow$ and some $b\in x^\leftarrow$, then the directed triangle $D[x,y,b]$ lies in $D^-(a)\isom$, which is not possible.
Thus, we have
\begin{txteq}\label{eq_GenPO2Vertices1aa}
no vertex in~$x^\leftarrow$ has successors in~$x^\rightarrow$.
\end{txteq}

Let us suppose that no $a\in x^\rightarrow$ and $b\in x^\perp$ are adjacent.
In $D^+(y)$, we find a common predecessor~$c$ and a common successor~$c'$ of~$a$ and~$b$.
Since neither of them can lie in $x^\perp$ or in $x^\rightarrow$ by assumption, both lie in~$x^\leftarrow$.
Any predecessor of~$c$ in $D^+(y)$ is also a predecessor of~$a$ and~$b$ and thus must lie in~$x^\leftarrow$.
By C-homogeneity, we find an automorphism $\alpha$ of~$D$ that fixes $x$ and~$y$ and maps $c$ to~$c'$.
This is impossible, as $c'=c\alpha$ has predecessors in~$D^+(y)$ that lie outside $x^\leftarrow=(x^\leftarrow)\alpha$.
Thus, we have shown that some vertex of~$x^\rightarrow$ has a neighbour in $x^\perp$.
By C-homogeneity and due to~(\ref{eq_GenPO2Vertices1a}), we have
\begin{txteq}\label{eq_GenPO2Vertices1b}
every vertex in~$x^\perp$ has a predecessor in~$x^\rightarrow$ and every vertex in~$x^\rightarrow$ has a successor in~$x^\perp$.
\end{txteq}

If any vertex $a$ in~$x^\perp$ has a predecessor in~$x^\leftarrow$, then the end vertices of the induced $2$-arc $xya$ have a common predecessor.
Thus, we have shown:
\begin{txteq}\label{eq_GenPO2Vertices2}
if~{\rm (\ref{eq_GenPO2Vertices1})} does not hold, then no vertex of~$x^\perp$ has a predecessor in~$x^\leftarrow$.
\end{txteq}

Let us assume that we have $ab\in ED$ for all $a\in x^\rightarrow$ and all $b\in x^\perp$.
Because of $D^+\isom\PF$, we find a vertex $z\in N^+(y)$ that is adjacent to neither~$a$ nor~$b$.
Hence, $z$ lies neither in~$x^\perp$ nor in~$x^\rightarrow$.
Thus, we have $z\in x^\leftarrow$.
Let $u$ be a common successor of~$z$ and~$b$ in~$D^+(y)$.
We have $u\notin x^\rightarrow$ by~(\ref{eq_GenPO2Vertices1a}) because of $bu\in ED$.
By~(\ref{eq_GenPO2Vertices2}), the edge $zu$ implies that either (\ref{eq_GenPO2Vertices1}) holds or $u\notin x^\perp$.
So we may assume $u\in x^\leftarrow$.
Then~(\ref{eq_GenPO2Vertices2}) implies~(\ref{eq_GenPO2Vertices1}) as $b\in x^\perp$ has the predecessor $u\in x^\leftarrow$.
Due to~(\ref{eq_GenPO2Vertices1a}), we have shown
\begin{txteq}\label{eq_GenPO2Vertices2a}
if {\rm (\ref{eq_GenPO2Vertices1})} does not hold, then for every vertex in~$x^\perp$ there is some vertex in~$x^\rightarrow$ such that these two vertices are not adjacent.
\end{txteq}

Since every two vertices in~$N^+(y)$ have a common predecessor, the existence of a vertex $z_1$ in~$x^\leftarrow$ and a vertex $z_2$ in~$x^\rightarrow$ that are not adjacent implies that the end vertices of the induced $2$-arc $z_1xz_2$ have a common predecessor.
Together with~(\ref{eq_GenPO2Vertices1aa}), this implies that
\begin{txteq}\label{eq_GenPO2Vertices2b}
if~{\rm (\ref{eq_GenPO2Vertices1})} does not hold, then every vertex of~$x^\rightarrow$ is a predecessor of every vertex of~$x^\leftarrow$.
\end{txteq}

Let $ab\in ED$ with $a\in x^\rightarrow$ and $b\in x^\perp$.
This edge exists due to~(\ref{eq_GenPO2Vertices1b}).
By~(\ref{eq_GenPO2Vertices2a}), we may assume that there is some vertex $c\in x^\rightarrow$ with $cb\notin ED$.
Then~(\ref{eq_GenPO2Vertices1a}) implies that $c$ and~$b$ are not adjacent.

If $a$ and~$c$ are adjacent, then $ca\notin ED$ because we have $cb\notin ED$ and $D^+(y)$ contains no induced $2$-arc.
So let us assume $ac\in ED$.
In $D^+(y)$, we find a vertex $c'\in N^-(c)$ that is adjacent to neither $a$ nor~$b$.
We have $c'\notin x^\perp$ due to~(\ref{eq_GenPO2Vertices1a}) because of $c\in x^\rightarrow$.
By~(\ref{eq_GenPO2Vertices2b}), either (\ref{eq_GenPO2Vertices1}) holds or $c'\notin x^\leftarrow$.
Thus, we may assume $c'\in x^\rightarrow$.
Taking $c'$ instead of~$c$, we may assume that $a$ and~$c$ are not adjacent.
Thus, the end vertices of~$D[c,x,a,b]$ lie in~$D^+(y)$ and hence have a common predecessor.
By a symmetric argument, we obtain that
\begin{txteq}\label{eq_GenPO2Vertices3}
if~{\rm (\ref{eq_GenPO2Vertices1})} does not hold, then the end vertices of any induced path isomorphic to either $D[c,x,a,b]$ or the digraphs obtained from $D[c,x,a,b]$ by reversing the directions of all its edges have a common predecessor.
\end{txteq}

Let $\alpha$ be an automorphism of~$D$ that fixes $x$ and~$y$ and interchanges $a$ to~$c$.
For $b':=b\alpha$ we have $b\neq b'\in (x^\perp)\alpha=x^\perp$.
Since $ab'\notin ED$ and $D^+(y)\isom\PF$, we have $bb'\notin ED$ and, symmetrically, we have $b'b\notin ED$.
Hence, $b$ and~$b'$ are not adjacent.
Let $u\in N^+(y)$ with $a,b,c\in N^-(u)$ and such that $u$ and~$b'$ are not adjacent.
If $u\in x^\perp$, then (\ref{eq_GenPO2Vertices3}) applied to $D[x,c,u,b]$ implies~(\ref{eq_GenPO2Vertices1}), since $x$ and~$b$ are the end vertices of the induced $2$-arc $xyb$.
Due to~(\ref{eq_GenPO2Vertices1a}), the vertex $u$ does not lie in~$x^\rightarrow$.
Hence, we may assume $u\in x^\leftarrow$.
Let $v$ be a predecessor of~$b'$ in $D^+(y)$ that has no neighbour in $\{a,b,c,u\}$.
Since $v$ and $u$ are not adjacent, (\ref{eq_GenPO2Vertices2b}) implies either (\ref{eq_GenPO2Vertices1}) or $v\notin x^\rightarrow$.
By~(\ref{eq_GenPO2Vertices2}) and as $vb'\in ED$, either~(\ref{eq_GenPO2Vertices1}) holds or $v\notin x^\leftarrow$.
Thus, we may assume $v\in x^\perp$.
Then (\ref{eq_GenPO2Vertices3}) applied to $D[x,c,b',v]$ shows that the end vertices of the induced $2$-arc $xyv$ have a common predecessor.
This shows~(\ref{eq_GenPO2Vertices1}).

\medskip

Due to~(\ref{eq_GenPO2Vertices1}), every two vertices of distance~$2$ have a common successor or a common predecessor.
If they have a common successor, then these three vertices induce a connected finite partial order and, by C-homogeneity, we find a common predecessor of all three vertices.
Hence, we have shown
\begin{txteq}
any two vertices of distance~$2$ have a common predecessor.
\end{txteq}

To show the lemma, it thus suffices to show
\begin{equation}\label{eq_GenPO2Vertices4}
\diam(D)=2.
\end{equation}
We consider all possible induced paths $P$ of length~$3$, not necessarily directed, one by one and show that the end vertices of such a path have distance~$2$.
If $P$ is an alternating walk, then it is a partial order and, for every $x\in VD$, the subdigraph $D^+(x)$ contains an isomorphic copy of~$P$.
By C-homogeneity, we find a vertex $z$ with $P\sub D^+(z)$ and the claim follows directly.

Let $a_1,a_2,a_3,a_4$ be the vertices of~$P$.
Let us assume that $a_1a_2$, $a_2a_3$, and~$a_4a_3$ are the edges on~$P$.
Since $D[a_2,a_3,a_4]$ is a connected partial order, we find a vertex~$x$ with $a_2,a_3,a_4\in N^+(x)$.
If $a_1$ and~$x$ are adjacent, then we have $d(a_1,a_4)=2$.
If $a_1$ and~$x$ are not adjacent, then $D[a_1,a_2,x,a_4]$ is a connected partial order that lies in~$D^+(z)$ for some $z\in VD$ by C-homogeneity.
Thus, also in this case, $a_1$ and~$a_4$ have a common neighbour.
Similar orientations like in this case (e.g., with edges $a_2a_1$, $a_2a_3$, and $a_3a_4$) follow by symmetric arguments.

The only remaining case is that $P$ is an induced $3$-arc.
Then we find a common predecessor of the first and the third vertex on~$P$ and obtain -- either directly or by the previous case -- that the end vertices of~$P$ have distance~$2$.
This shows~(\ref{eq_GenPO2Vertices4}) and, as previously mentioned, the lemma.
\end{proof}

\begin{lem}\label{lem_FinPOinD+}
For every finite partial order $A$ in~$D$, there exists some $x\in VD$ with $A\sub D^+(x)$.
\end{lem}

\begin{proof}
If $A$ is connected, then the assertion is a direct consequence of $C$-homo\-geneity, as for every $x\in VD$ the subdigraph $D^+(x)$ contains an isomorphic copy of~$A$.
So let us assume that $A$ is not connected.
If $|VA|=2$, then the assertion follows from Lemma~\ref{lem_GenPO2Vertices}.
So we may assume $|VA|\geq 3$.
If $VA$ is an independent set, let $a$ be an arbitrary vertex of~$A$.
If $A$ has an edge, let $a\in VA$ such that $a$ has a successor in~$A$ but no predecessor.
By induction on~$|A|$, we find $x\in VD$ with $A-a\sub N^+(x)$.
If $xa\in ED$, then $x$ is the vertex we are searching for.
So let us assume either that $ax\in ED$ or that $a$ and~$x$ are not adjacent.
In each case, $A+x$ is a partial order and it has less components than~$A$.
Thus, the assertion holds by induction on the number of components of~$A$.
\end{proof}

\begin{lem}\label{lem_2FinPOinN}
Let $A,A',B,B'$ be finite induced partial orders in~$D$ such that an isomorphism $\varphi\colon A'+B'\to A+B$ with $A'\varphi=A$ and $B'\varphi=B$ exists.
If $A$ is a maximal partial order in~$A+B$ and if $D$ has a vertex $v$ with $A'\sub D^+(v)$ and $B'\sub D^-(v)$, then there exists $x\in VD$ with $A\sub D^+(x)$ and $B\sub D^-(x)$.
\end{lem}

\begin{proof}
If $A+B$ is connected or if $B$ is empty, then the assertion follows either by C-homogeneity or by Lemma~\ref{lem_FinPOinD+}.
So let us assume that $A+B$ has at least two components and that $B$ is not empty.
By induction and similar to the proof of Lemma~\ref{lem_TwoI_nFreeSubdigraphs}, we may assume that there are $z\in VB$ and $z'\in VB'$ such that $A=A'$ and $B-z=B'-z'$.
Furthermore, we may assume that $z$ does not lie in~$N^-(v)$, because the assertion follows directly in that case.
Since $A$ is a maximal partial order in~$A+B$, we know that $A$ contains vertices from each component of~$A+B$.
Let $a_1,\ldots,a_n\in VA$ such that $\{a_1,\ldots,a_n,z\}$ has precisely one vertex from each component of~$A+B$.
By Lemma~\ref{lem_FinPOinD+}, we find a vertex $y$ with $\{a_1,\ldots,a_n,z,z'\}\sub N^+(y)$.
The digraphs $A+B+y$ and $A'+B'+y$ are connected and isomorphic to each other.
By C-homogeneity, there is an automorphism $\alpha$ of~$D$ that fixes $y$ and all vertices of $A$ and $B-z$ and maps $z$ to~$z'$.
Hence, $v\alpha$ is a vertex we are searching for.\looseness-1
\end{proof}

\begin{prop}
Let $D$ be a countable connected C-homogeneous digraph with $D^+\isom\PF$.
Then $D$ is homogeneous.
\end{prop}

\begin{proof}
Let $A$ and $B$ be isomorphic finite induced subdigraphs of~$D$ and $\varphi\colon A\to B$ be an isomorphism.
Let $A_1$ be a maximal partial order of~$A$ and $A_2$ be a maximal partial order of $A\sm A_1$ such that for some vertex $x\in VD$ there is an embedding~$\tau$ from $A_1+A_2$ to $D^+(x)+D^-(x)$ such that $A_1\tau\sub D^+(x)$ and $A_2\tau\sub D^-(x)$.
Note that $A_1$ contains vertices from each component of~$A$ by its maximality.
Let $B_1=A_1\varphi$ and $B_2=A_2\varphi$.
Due to Lemma~\ref{lem_2FinPOinN}, we find a vertex $y$ with $A_1\sub D^+(y)$ and $A_2\sub D^-(y)$ and a vertex $z$ with $B_1\sub D^+(z)$ and $B_2\sub D^-(z)$.
By maximalities of $A_1$ and~$A_2$, we know that no vertex of $A\sm (A_1+A_2)$ is adjacent to~$y$ and, similarly, no vertex of $B\sm (B_1+B_2)$ is adjacent to~$z$.
The isomorphism $\varphi$ extends canonically to an isomorphism $\varphi'\colon A+y\to B+z$.
Since $A+y$ and $B+z$ are connected, we can extend $\varphi'$, and hence also $\varphi$, to an automorphism $\alpha$ of~$D$ by C-homogeneity.
\end{proof}

\subsection{\boldmath The digraphs $T[I_n]$ as $D^+$}\label{sec_TI_n}

In this section, let $D$ be a countable connected C-homogeneous digraph with $D^+\isom T[I_n]$ for some countable homogeneous tournament $T\neq I_1$ and some $n\in\nat^\infty$.
Our first aim in this section is to determine~$D^-$.

\begin{lem}\label{lem_D+TInThenD-TInPreInfo}
If $n\geq 2$, then $D^-\isom T'[I_m]$ for some countable homogeneous tournament $T'\neq I_1$ and some $m\in\nat^\infty$.
\end{lem}

\begin{proof}
Let $xz\in ED$.
Note that $VD^-$ is not an independent set, since $z$ has a predecessor in $D^+(x)$.
As $n\geq 2$, there are two non-adjacent vertices $y_1,y_2\in N^+(x)\cap N^-(z)$.
Since the digraph $D[x,y_1,y_2]\sub D^-(z)$ cannot be embedded into $I_k[T']$ for any countable homogeneous tournament $T'\neq I_1$ and any $k\in\nat^\infty$, Theorem~\ref{thm_CountHomDigraphs} together with the previous sections imply the assertion.
\end{proof}

\begin{lem}\label{lem_D+TInThenD-TIn}
If $D^-\isom T'[I_m]$ for some countable homogeneous tournament $T'\neq I_1$ and some $m\in\nat^\infty$, then $D^+\isom D^-$.
\end{lem}

\begin{proof}
To show $m=n$, let $x\in VD$.
As $T\neq I_1$, any vertex in $D^+(x)$ has $n$ independent predecessors in~$D^+(x)$.
Hence, we conclude $m\geq n$.
By a symmetric argument we also have $n\geq m$.
To show $D^+\isom D^-$ it thus suffices to show $T=T'$.

Note that $T=C_3$ implies $T'=C_3$ and vice versa because in any countable infinite homogeneous tournament, we have arbitrarily large finite tournaments in the out- and in the in-neighbourhood of every vertex.

Let us now show $T=T'$ in the case $T=T^\infty$.
Let $x\in VD$ and let $F$ be a finite tournament in $D^+(x)$.
As $T^\infty$ is homogeneous and embeds every finite tournament, we find a vertex $y\in N^+(x)$ with $F\sub D^-(y)$.
Thus, $T'$ contains every finite tournament.
So we have $T'=T^\infty=T$.

Next, we assume $T=\rat$.
Let us suppose $T\neq T'$.
Then we obtain from the previous cases $T'=S(2)$.
Let $xy\in ED$.
As $x$ has a predecessor in~$D^-(y)$, let $a\in N^-(x)\cap N^-(y)$.
Since $D^-(x)$ contains a directed triangle and is homogeneous, we find $b,c\in N^-(x)$ with $ab,bc,ca\in ED$.
Since $D^+(a)\isom \rat[I_n]$, we have $by\in ED$.
Similarly, we conclude $cy\in ED$.
The digraph $D[x,a,b,c]$ cannot be embedded into $S(2)[I_m]$ even though it lies in~$D^-(y)$.
This contradiction shows $T=T'$ if $T=\rat$ and finishes the proof of the lemma.
\end{proof}

We remark that we will see in Section~\ref{sec_InT}, that the assumption $D^-\isom T'[I_m]$ in Lemma~\ref{lem_D+TInThenD-TIn} is not only satisfied if $n\geq 2$ (due to Lemma~\ref{lem_D+TInThenD-TInPreInfo}) but also if $n=1$ (due to Lemma~\ref{lem_InT_m=n}).

If either $n\geq 2$ or $D^+\isom T\isom D^-$, then the next lemma will exclude the possibility $T=S(2)$:

\begin{lem}\label{lem_TIn0}
If $D^+\isom D^-$, then $T\neq S(2)$.
\end{lem}

\begin{proof}
Let us suppose $T=S(2)$.
Let $x\in VD$ and let $a,b,c\in N^+(x)$ with $ab,bc,ca\in ED$.
Since $D[x,a,b]$ can be embedded into~$D^+$, we find a vertex $y\in VD$ with $D[x,a,b]\sub D^+(y)$ by C-homogeneity.
Since $D^-(a)\isom S(2)[I_n]$ and $c$ and~$y$ do not both lie either in~$D^-(x)$ or in $D^+(x)$, these two vertices must be adjacent.
Because $D[x,a,b,c]$ does not embed into~$D^+$, this edge cannot be $yc$, so it is $cy$.
In $D^-(b)$ we find a vertex $z$ with $z\in N^+(a)\cap N^+(x)\cap N^-(y)$.

Since $D$ is C-homogeneous, we find an automorphism $\alpha$ of~$D$ that fixes $x$ and~$y$ and maps $ca$ to~$zb$.
Since $b$ lies in $N^+(a)\cap N^-(c)$, its image $b\alpha$ lies in $N^+(b)\cap N^-(z)$.
Considering $D^+(x)$, we know that $b\alpha$ cannot lie in $N^+(a)$ as $D^+(x)\cap D^+(a)$ contains no directed triangle $D[b,b\alpha,z]$ but $b\alpha$ must be adjacent to~$a$.
So we have $b\alpha a\in ED$.
But then $D[a,b,b\alpha,x]$ is a digraph which lies in $D^+(y)$ even though it cannot be embedded into $S(2)[I_n]$.
This contradiction shows the assertion.
\end{proof}

The following lemma shows that we can restrict ourselves to the situation $n=1$ in the remainder of this section: all the other C-homogeneous digraphs that satisfy the assumptions of this section and that have the property $n\geq 2$ arise from those with $n=1$ in a canonical way.

\begin{lem}\label{lem_TIn}
If $D^+\isom D^-$, then there is a countable connected C-homogeneous digraph $D'$ with $D'^+\isom T\isom D'^-$ and with $D'[I_n]\isom D$.
\end{lem}

\begin{proof}
Let $x\in VD$.
Let us first show that
\begin{txteq}\label{eq_TIn1}
$N^-(a)=N^-(b)$ for each two non-adjacent vertices $a,b\in N^+(x)$.
\end{txteq}
Let $y\in N^-(a)$.
First, let us assume that $x$ and~$y$ are adjacent.
If $y\in N^+(x)$, then it is an immediate consequence of $D^+(x)\isom T[I_n]$ that $y$ lies in~$N^-(b)$.
So let us assume $yx\in ED$.
If $D$ contains no directed triangle, then it contains a transitive triangle and, by C-homogeneity, we find a vertex $z\in VD$ with $D[x,y,a]\sub D^-(z)$.
Then $D^+(x)$ shows $bz\in ED$ and $D^-(z)$ shows $yb\in ED$.
If $T$ contains a directed triangle, let $z\in N^-(a)$ such that $D[x,y,z]$ is a directed triangle.
Let $a^\perp$ be the set of vertices in~$D$ that are not adjacent to~$a$.
Due to the structure of~$D^+(y)$, we observe $N^+(y)\cap a^\perp\sub N^+(x)\cap a^\perp$ and conclude
\[
N^+(y)\cap a^\perp\sub N^+(x)\cap a^\perp\sub N^+(z)\cap a^\perp\sub N^+(y)\cap a^\perp.
\]
So all inclusions are equalities, which shows $yb\in ED$.

Now we assume that $x$ and~$y$ are not adjacent.
Then we find $z\in N^-(a)$ with $x,y\in N^+(z)$.
So we have due to the previous situation that $z$ lies in $N^-(b)$ and hence that $y$ lies in~$N^-(b)$.
This shows~(\ref{eq_TIn1}).

Let us define a relation $\sim$ on~$VD$ via
\begin{txteq}\label{eq_TIn2}
$u\sim v\ :\Longleftrightarrow\ N^-(u)=N^-(v)$ for all $u,v\in VD$.
\end{txteq}
Then $\sim$ is obviously an $\Aut(D)$-invariant equivalence relation with no two adjacent vertices in the same equivalence class.
Let $A,B$ be two equivalence classes and let $a_1,a_2\in A$ and $b_1,b_2\in B$ with $a_1b_1\in ED$.
By definition, we know $a_1b_2\in ED$.
Let $c\in N^-(a_1)\cap N^-(b_1)$.
By definition of~$\sim$, we conclude $ca_2,cb_1\in ED$.
So we have $D[a_1,a_2,b_1,b_2]\sub D^+(c)$.
Due to the structure of~$D^+(c)$ and as $a_1$ and~$a_2$ are not adjacent, $a_2$ is a predecessor of~$b_1$ and of~$b_2$.
Thus, we have shown that
\begin{txteq}\label{eq_TIn3}
each two equivalence classes induce either a complete or an empty bipartite digraph.
\end{txteq}
Thus, $D_\ssim$ is a digraph.
Note that (\ref{eq_TIn3}) implies that $D_\ssim$ inherits C-homogeneity from~$D$.
By~(\ref{eq_TIn1}), we conclude $D\isom D_\ssim[I_n]$ and $D^+\isom T$.
\end{proof}

Now we are able to complete the investigation for $D$ if $D^+\isom C_3[I_n]\isom D^-$.

\begin{lem}\label{lem_TIn1}
If $D^+\isom C_3[I_n]\isom D^-$, then $D\isom C_3^\wedge[I_n]$.
\end{lem}

\begin{proof}
By Lemma~\ref{lem_TIn}, it suffices to show $D\isom C_3^\wedge$ if $n=1$.
Note that $D$ is locally finite, if $n=1$.
So we obtain the assertion from~\cite[Lemma~4.5]{FinConHomDigraphs}.
\end{proof}

In the following we only have to look closer at the cases $T=T^\infty$ and $T=\rat$.
So we assume for the remainder of this section that $T$ is one of those two tournaments.
In both cases we obtain (among others) digraphs that are similar to those that we obtain in the case of $T=C_3$: the digraphs $T^\wedge[I_n]$.
The situation in which they occur (in the case $n=1$) is that every edge lies on precisely two induced $2$-arcs, once as the first edge and once as the last edge:

\begin{lem}\label{lem_TInOne2Arc}
If $n=1$, if $D^+\isom D^-$, and if every edge of~$D$ is on precisely one induced $2$-arc the first edge and on precisely one induced $2$-arc the last edge, then $D\isom T^\wedge$.
\end{lem}

\begin{proof}
Let $x\in VD$.
We first show that
\begin{txteq}\label{eq_TInOne2Arc1}
there exists a unique vertex $x^\perp$ such that every induced $2$-arc that starts at~$x$ ends at $x^\perp$.
\end{txteq}
Suppose (\ref{eq_TInOne2Arc1}) does not hold.
Then we find two distinct $2$ arcs $xyz$ and $xuv$ in~$D$.
By assumption, we have $y\neq u$.
Since $y$ and~$u$ lie in the tournament $D^+(x)$, they are adjacent.
So we may assume $yu\in ED$.
Because there is a unique induced $2$-arc whose second edge is~$uv$, we know that $y$ and~$v$ are adjacent.
As $x$ and~$v$ are not adjacent, $v$ cannot lie in~$D^-(y)$, so we have $v\in N^+(y)$.
But then the edge $xy$ lies on the two induced $2$-arcs $xyz$ and $xyv$.
This contradiction to the assumption shows~(\ref{eq_TInOne2Arc1}).

Next, we show
\begin{equation}\label{eq_TInOne2Arc2}
(x^\perp)^\perp=x.
\end{equation}
Let $xyx^\perp$ be an induced $2$-arc.
Let $a\in N^+(y)\cap N^-(x^\perp)$.
Since $xya$ cannot be an induced $2$-arc by assumption, $a$ and~$x$ are adjacent.
This edge must be $xa$ because of $D^+(a)\isom T$.
So there exists $b\in VD$ with $x,a\in N^+(b)$
Since $xax^\perp$ is an induced $2$-arc, the edge $ax^\perp$ cannot lie on a second induced $2$-arc $bax^\perp$.
Hence, $b$ and~$x^\perp$ are adjacent.
Note that $x^\perp$ does not lie in~$N^+(b)$ because of $D^+(b)\isom T$ and $x\in N^+(b)$.
So $x^\perp b\in ED$ and $x^\perp bx$ is an induced $2$-arc that shows~(\ref{eq_TInOne2Arc2}).

Let us show that
\begin{txteq}\label{eq_TInOne2Arc3}
$\diam(D)=2$ and $x^\perp$ is the only vertex in~$D$ that is not adjacent to~$x$.
\end{txteq}
Since $D$ contains induced $2$-arcs, its diameter is at least~$2$.
Let $xux^\perp$ and $x^\perp vx$ be induced $2$-arcs.
Any neighbour of~$x^\perp$ except for $u$ and~$v$ must be adjacent to either $u$ or~$v$ because of $D^+\isom T\isom D^-$, so its distance to~$x$ is at most~$2$.
Because of $D^+\isom T\isom D^-$, any two vertices $a,b$ with $d(a,b)=2$ must be the end vertices of an induced $2$-arc.
Hence, (\ref{eq_TInOne2Arc1}) and~(\ref{eq_TInOne2Arc2}) show that every neighbour of~$x^\perp$ must be adjacent to~$x$.
This shows~(\ref{eq_TInOne2Arc3}).

Now we are able to show $D\isom T^\wedge$.
Due to~(\ref{eq_TInOne2Arc3}), we know that $D$ is the union of $D_1:=D^+(x)+x$ and $D_2:=D^-(x)+x^\perp$.
Furthermore, we have $D^-(x)=D^+(x^\perp)$ because $x$ and~$x^\perp$ have no common successor and no common predecessor.
Let us define
\[
\varphi\colon D_1\to D_2, \quad y\mapsto y^\perp.
\]
Since $D_1$ and $D_2$ are tournaments, $y^\perp$ does not lie in~$D_1$ for any $y\in VD_1$, so $\varphi$ is well-defined.
Similarly, $\varphi$ is surjective.
Due to~(\ref{eq_TInOne2Arc1}) and~(\ref{eq_TInOne2Arc2}), we also have that $\varphi$ is injective.
Let $uv\in ED_1$.
Then $vu^\perp\in ED$ and $u^\perp v^\perp\in ED$ as $D^-$ is a tournament.
This shows that $\varphi$ is an isomorphism.
Let $a\in D_1$ and $b\in D_2$.
If $ab\in ED$, then $ba^\perp\in ED$ and, if $ba\in ED$, then $a^\perp b\in ED$.
Thus, we have shown $D\isom T^\wedge$.
\end{proof}

Now we determine $D$ in the case $T=T^\infty$ if $D^+\isom D^-$.

\begin{lem}\label{lem_TIn2}
If $D^+\isom T^\infty[I_n]\isom D^-$, then either $D\isom (T^\infty)^\wedge[I_n]$ or $D\isom T'[I_n]$ for some countable homogeneous tournament~$T'$.
\end{lem}

\begin{proof}
Let us assume that $n=1$ and that $D$ is not a homogeneous tournament.
As any induced subdigraph of a tournament is connected, C-homogeneity implies that $D$ is no tournament at all.

Since $D^+$ and $D^-$ are tournaments, we find between each two vertices $x$ and~$y$ of distance~$2$ an induced $2$-arc $xyz$ in~$D$.
Our aim is to apply Lemma~\ref{lem_TInOne2Arc}.
Therefore, we prove that
\begin{txteq}\label{eq_TIn2_1}
there is no $z'\neq z$ in~$VD$ such that $xyz'$ is an induced $2$-arc.
\end{txteq}
Let us suppose that we find a vertex $z'\neq z$ such that $xyz'$ is an induced $2$-arc.
Since $D^+(y)\isom T^\infty$, the vertices $z$ and $z'$ are adjacent, say $zz'\in ED$.
Let $a\in N^+(y)\cap N^+(x)$.
Because $D^-(a)$ is a tournament, neither $z$ nor~$z'$ lies in~$N^-(a)$.
Since $D^+(y)$ is a tournament, $a$ is adjacent to~$z$ and to~$z'$.
Thus, $z$ and~$z'$ lie in~$N^+(a)$.
In $D^+(y)\isom T^\infty$, we find a vertex $b$ with $ba,bz,z'b\in ED$.
Considering $D^-(a)$, we know that $b$ and~$x$ are adjacent, but neither $bx$ nor $xb$ is an edge of~$D$ since neither $D^+(b)$ can contain $x$ and~$z$ nor $D^-(b)$ can contain $x$ and~$z'$.
This contradiction shows~(\ref{eq_TIn2_1}).

By an analogous proof as above, there is precisely one induced $2$-arc whose second edge is~$xy$.
Thus, the assertion follows from Lemma~\ref{lem_TInOne2Arc}.
\end{proof}

It remains to determine $D$ in the case $T=\rat$.

\begin{lem}\label{lem_TIn3}
If $D^+\isom \rat[I_n]\isom D^-$, then $D$ is isomorphic to one of the following digraphs:
\begin{enumerate}[\rm (i)]
\item $\rat^\wedge[I_n]$;
\item $S(3)[I_n]$; or
\item $T'[I_n]$ for some countable homogeneous tournament~$T'$.
\end{enumerate}
\end{lem}

\begin{proof}
As in the proof of Lemma~\ref{lem_TIn2}, we assume $n=1$ and that $D$ is not a (homogeneous) tournament.
If for every edge $xy$ there is precisely one induced $2$-arc whose first edge is~$xy$ and precisely one induced $2$-arc whose second edge is~$xy$, then Lemma~\ref{lem_TInOne2Arc} implies $D\isom \rat^\wedge$.
By symmetry, let us assume that $xy$ lies on two induced $2$-arcs $xyz$ and $xyz'$.

Considering $D^+(y)$, the vertices $z$ and~$z'$ are adjacent.
We may assume $zz'\in ED$.
Let $z''\in N^+(y)\cap N^+(x)$.
Note that $D^-(z'')\isom\rat$ implies that neither $z$ nor~$z'$ lies in~$N^-(z'')$.
But as $z,z',z''\in N^+(y)$, the vertex $z''$ is adjacent to~$z$ and to~$z'$.
Hence, we have $z''z\in ED$ and $z''z'\in ED$.
By C-homogeneity, we find an automorphism $\alpha$ of~$D$ that fixes $y$ and~$z'$ and maps $z''$ to~$z$.
Since $x\in N^-(z'')$ but $x\notin N(z')$, we conclude $x':=x\alpha\neq x$.
Note that $x$ and $x'$ must be adjacent as both vertices lie in~$D^-(y)$ but $x'x\notin ED$ because not both of the two non-adja\-cent vertices $x$ and~$z$ can lie in~$D^+(x')$.
Thus, we have $xx'\in ED$ and $xx'zz'$ is an induced $3$-arc.
Thus, we have shown that
\begin{txteq}\label{eq_TIn3_1}
the end vertices of any induced $2$-arc are also end vertices of an induced $3$-arc.
\end{txteq}

Let us show the following:
\begin{txteq}\label{eq_TIn3_2}
$D$ contains either an induced directed cycle or an induced directed double ray.
\end{txteq}
First, let us assume that there is an integer $m$ such that $D$ contains an induced $m$-arc but no induced $(m+1)$-arc.
Let $m$ be smallest possible.
Due to~(\ref{eq_TIn3_1}), we have $m\geq 3$.
Let $a_0\ldots a_m$ be an induced $m$-arc and $a_{m+1}\in VD$ such that $a_1\ldots a_{m+1}$ is also an induced $m$-arc.
To see that such a vertex $a_{m+1}$ exists, take an automorphism~$\alpha$ of~$D$ that maps $a_0\ldots a_{m-1}$ to $a_1\ldots a_m$, which exists by C-homogeneity, and set $a_{m+1}:=a_m\alpha$.
By the choice of~$m$, we know that $a_0\ldots a_{m+1}$ is not an induced $(m+1)$-arc.
If $a_0=a_{m+1}$, then $a_0\ldots a_m$ is an induced directed cycle.
So $a_0$ and $a_{m+1}$ are distinct but adjacent.
As $m\geq 2$, the vertices $a_0$ and $a_m$ are not adjacent.
Hence, $a_0$ cannot lie in the tournament $D^-(a_{m+1})$.
Thus, we have $a_{m+1}a_0\in ED$ and the vertices $a_0,\ldots, a_{m+1}$ form an induced directed cycle.

If no such $m$ exists, then $D$ contains an induced $n$-arc for every $n\in\nat$, as it contains an induced $3$-arc by~(\ref{eq_TIn3_1}).
Hence, $D$ contains an induced directed double ray: by C-homogeneity, we can enlarge every $n$-arc $a_1\ldots a_{n+1}$ to an $(n+2)$-arc $a_0\ldots a_{n+2}$ in a similar way we enlarged the $m$-arc in the previous case.
Continuing in this way we obtain an induced directed double ray, which shows~(\ref{eq_TIn3_2}).

Next, we show that
\begin{txteq}\label{eq_TIn3_2b}
$D$ contains no induced $4$-arc.
\end{txteq}
Let us suppose that $D$ contains an induced $4$-arc $a_0\ldots a_4$.
By~(\ref{eq_TIn3_1}) and C-homo\-geneity, we find a vertex $b$ such that $a_0ba_3$ in an induced $2$-arc.
Since $D^-(b)$ does not contain two non-adjacent vertices, we have $a_4b\notin ED$.
So either $a_0ba_4$ is an induced $2$-arc or $a_0ba_3a_4$ is an induced $3$-arc and we find by~(\ref{eq_TIn3_1}) a vertex $c$ such that $a_0ca_4$ is an induced $2$-arc.
For simplicity, set $c:=b$ if $a_0ba_4$ is an induced $2$-arc.
Considering $D^+(a_0)$ we know that $a_1$ and~$c$ are adjacent. As an edge $ca_1$ is a contradiction to $D^+(c)\isom \rat$, we have $a_1c\in ED$ and we conclude as before that $a_2$ and~$c$ are adjacent.
But an edge $a_2c$ implies that $D^-(c)$ contains the two non-adjacent vertices $a_0$ and~$a_2$ and an edge $ca_2$ implies that $D^+(c)$ contains the two non-adjacent vertices $a_2$ and~$a_4$.
This contradiction shows~(\ref{eq_TIn3_2b}).

A direct consequence of~(\ref{eq_TIn3_2b}) is that
\begin{txteq}\label{eq_TIn3_3}
$D$ contains neither an induced directed double ray nor an induced directed cycle of length at least~$6$.
\end{txteq}

The next step is to show that
\begin{txteq}\label{eq_TIn3_4}
$D$ contains no directed triangle.
\end{txteq}
Let $xy\in ED$.
For every $a\in N^-(y)$, we define\pagebreak
\begin{align*}
a^\rightarrow&=\{v\in N^+(y)\mid av\in ED\},\\
a^\leftarrow&=\{v\in N^+(y)\mid va\in ED\},\text{ and}\\
a^\perp&=\{v\in N^+(y)\mid a\text{ not adjacent to }v\}.
\end{align*}
Let $a_1\in a^\rightarrow$, $a_2\in a^\leftarrow$ and $a_3\in a^\perp$.
These three vertices form a transitive triangle as they lie in $D^+(y)\isom \rat$.
Since $D^+(a_2)$ is a tournament and $a\in N^+(a_2)$, we have $a_3a_2\in ED$ and, since $D^-(a_1)$ is a tournament and $a\in N^-(a_1)$, we have $a_1a_3\in ED$.
As $D[a_1,a_2,a_3]$ is transitive, we conclude $a_1a_2\in ED$.
So we have $a^\rightarrow\cup a^\perp\sub N^-(a_2)$ and $a^\leftarrow\cup a^\perp\sub N^+(a_1)$.

Let us suppose that $D$ contains some directed triangle.
Let $z,z'\in x^\perp$ with $zz'\in ED$, let $u\in x^\rightarrow$, and let $v\in x^\leftarrow$.
As $D$ contains a directed triangle, we find a vertex $w$ such that $D[w,y,u]$ is such a triangle.
As we have $w^\rightarrow\cup w^\perp\sub N^-(w')$ for every $w'\in w^\leftarrow$ and as $u\in w^\leftarrow$, we conclude $N^+(y)\cap N^+(u)\sub w^\leftarrow$.
In particular, we have $x^\perp\sub N^+(y)\cap N^+(u)\sub w^\leftarrow$.
In particular, we have $z'w\in ED$.
By C-homogeneity, we find an automorphism~$\alpha$ of~$D$ that fixes $y$ and~$z$ and maps~$v$ to~$z'$.
Then we have $x\alpha\neq x$, as $v\alpha=z'\in x^\perp$ but $v\notin x^\perp$.
Since $w$ and~$x\alpha$ lie in $D^-(y)$, they are adjacent to~$x$.
But neither of them lies in $N^+(x)$, because both lie in $N^+(z')$ and $D^-$ is a tournament.
Note that $z\in x^\perp\cap (x\alpha)^\perp$.
Thus, we have $x^\perp\not\sub (x\alpha)^\leftarrow$ and hence we do not find any automorphism of~$D$ that fixes $x$ and~$y$ and maps $w$ to~$x\alpha$.
This contradiction to C-homogeneity shows~(\ref{eq_TIn3_4}).

We know by~(\ref{eq_TIn3_2})--(\ref{eq_TIn3_4}) that the only induced directed cycles in~$D$ have length either $4$ or~$5$.
Next, we show that
\begin{txteq}\label{eq_TIn3_5}
$D$ contains a directed cycle of length~$4$.
\end{txteq}
If $D$ contains no induced directed cycle of length~$4$, then $D$ contains only induced directed cycles of length~$5$.
Let $a_1\ldots a_5a_1$ be such a cycle.
Due to~(\ref{eq_TIn3_1}), there is an induced $2$-arc $a_1ua_4$ in~$D$.
Since $a_1ua_4a_5a_1$ is not an induced directed cycle of length~$4$, the vertices $u$ and~$a_5$ must be adjacent.
But an edge $ua_5$ implies that $a_1ua_5a_1$ is a directed triangle and an edge $a_5u$ implies that $ua_4a_5u$ is a directed triangle.
These contradictions to~(\ref{eq_TIn3_4}) show~(\ref{eq_TIn3_5}).

Let us show that
\begin{txteq}\label{eq_TIn3_6}
for every directed cycle $C$ of length~$4$ every vertex of~$D$ outside~$C$ has a predecessor $u$ and a successor $w$ on~$C$ with $uw\in ED$.
\end{txteq}
First, let $v$ be a vertex outside~$C$ that has a neighbour on~$C$.
If $v$ has a predecessor on~$C$, then there are at most two predecessors of~$v$ on~$C$, since $D^-(v)$ is a tournament.
Let $u$ be that predecessor of~$v$ on~$C$ whose successor $w$ on~$C$ does not lie in~$N^-(v)$.
Since $v$ and~$w$ lie in~$N^+(u)$, they are adjacent and by the choice of~$u$ we have $vw\in ED$.
If $v$ has a successor on~$C$, then an analogous argument shows the assertion for~$v$.
Since $D^+\isom \rat\isom D^-$, any neighbour of~$v$ that does not lie on~$C$ must be adjacent to some neighbour of~$v$ on~$C$ -- either to a predecessor or a successor.
Thus, every vertex of~$D\sm C$ is adjacent to some vertex of~$C$ and we have shown~(\ref{eq_TIn3_6}).

A consequence of~(\ref{eq_TIn3_6}) is the following:
\begin{txteq}\label{eq_TIn3_7}
the vertices that are not adjacent to a given vertex induce a tournament.
\end{txteq}
Let $C=x_1x_2x_3x_4x_1$ be a directed cycle of length~$4$, which exists by~(\ref{eq_TIn3_5}), and let $u$ and~$v$ be two vertices that are not adjacent to~$x_1$.
By~(\ref{eq_TIn3_6}) we know that each of $u$ and $v$ has a predecessor on~$C$, which cannot be $x_4$ since $D^+(x_4)$ is a tournament.
Furthermore, each of $u$ and $v$ has a successor on~$C$, which cannot be $x_2$ since $D^-(x_2)$ is a tournament.
If $u$ and~$v$ are not adjacent, then we may assume that $x_3u,ux_4\in ED$ and $x_2v,vx_3\in ED$ as $D^+$ and~$D^-$ are tournaments.
Note that neither $vx_4$ nor $x_2u$ lies in~$ED$ as $u$ and~$v$ are not adjacent.
Thus, $ux_4x_1x_2v$ is an induced $4$-arc.
This contradiction to~(\ref{eq_TIn3_2b}) proves~(\ref{eq_TIn3_7}).

\medskip

We are now able to show $D\isom S(3)$.
To show this, it suffices to show that $D$ is homogeneous, because the only homogeneous digraph with $D^+\isom \rat$ that has two distinct induced $2$-arcs $xyz$ and $xyz'$ is~$S(3)$.

Let $A$ and~$B$ two isomorphic finite induced subdigraphs of~$D$ and $\varphi\colon A\to B$ be an isomorphism.
If $A$ is connected, then $\varphi$ extends to an automorphism of~$D$ by C-homogeneity.
So let us assume that $A$ has at least two components.
Then~(\ref{eq_TIn3_7}) shows that $A$ has precisely two components $A_1$ and~$A_2$ both of which are tournaments.
Furthermore, each component can be embedded into~$\rat$ since $D$ contains no directed triangle by~(\ref{eq_TIn3_4}).
Let $a_1\in VA_1$ such that $A_1-a_1\sub D^+(a_1)$ and let $a_2\in VA_2$ such that $A_2-a_2\sub D^-(a_2)$.
Let $C$ be a directed cycle of length~$4$.
This exists by~(\ref{eq_TIn3_5}).
By C-homogeneity, we may assume $a_1\in VC$.
Due to~(\ref{eq_TIn3_6}), we know that $D$ contains either an induced $2$-arc from~$a_1$ to~$a_2$ or an induced $2$-arc from~$a_2$ to~$a_1$.
Indeed, if $auvw$ is the cycle~$C$, then $a_2$ has a predecessor on~$C$ by~(\ref{eq_TIn3_6}) which cannot be $w$ since $D^+(w)$ does not contain two non-adjacent vertices.
Similarly, $u$ is not a successor of~$a_2$.
Hence, either $a_1ua_2$ or $a_2wa_1$ is the induced $2$-arc we are searching for.
Since $a_1$ and $a_2$ lie on an induced $2$-arc, C-homogeneity implies that we may also assume $a_2\in VC$.
So we find a vertex $a\in VC\cap N^+(a_1)\cap N^-(a_2)$.
Note that $a\notin VA$.
Then every vertex $a_1'\in A_1\sm\{a_1\}$ must be adjacent to~$a$ since $a$ and~$a_1'$ lie in $D^+(a_1)\isom\rat$.
If $a$ is a predecessor of~$a_1'$, then $D^+(a)$ contains the two non-adjacent vertices $a_2$ and~$a_1'$, which is impossible.
Hence, $a$ is a successor of~$a_1'$.
Similarly, we obtain that $a$ is a predecessor of every vertex $a_2'\in VA_2$.
So we have $A_1\sub D^-(a)$ and $A_2\sub D^+(a)$.
Similarly, we find a vertex $b\in VD$ with $A_1\varphi\sub D^-(b)$ and $A_2\varphi\sub D^+(b)$.
Then $\varphi$ extends to an isomorphism from $A+a$ to $B+b$ and hence by C-homogeneity to an automorphism of~$D$.
So we obtain that $D$ is homogeneous and hence isomorphic to~$S(3)$.
\end{proof}

Let us summarize the results of this section:

\begin{prop}\label{prop_TIn}
Let $D$ be a countable connected C-homogeneous digraph with $D^+\isom T[I_n]$ for some countable homogeneous tournament $T$ and some $n\in\nat^\infty$.
If either $n\geq 2$ or if $D^+\isom T\isom D^-$, then $D$ is isomorphic to one of the following digraphs:
\begin{enumerate}[\rm (i)]
\item $T^\wedge[I_n]$ if $T\in\{C_3, \rat, T^\infty\}$; or
\item $S[I_n]$, where either $S=S(3)$ or $S$ is some countable homogeneous tournament.
\end{enumerate}
\end{prop}

\begin{proof}
Note that $D^+\isom D^-$ also holds if $n\geq 2$ due to Lemmas~\ref{lem_D+TInThenD-TInPreInfo} and~\ref{lem_D+TInThenD-TIn}.
Then the assertion directly follows from Lemmas~\ref{lem_TIn0}, \ref{lem_TIn1}, \ref{lem_TInOne2Arc}, \ref{lem_TIn2}, and~\ref{lem_TIn3}.
\end{proof}

We will see in Section~\ref{sec_InT} (Lemma~\ref{lem_InT_m=n}) that $D^+\isom T$ implies $D^-\not\isom I_m[T']$ for any $m\in\nat^\infty$  with $m\geq 2$ and any countable homogeneous tournament $T'\neq I_1$.
Thus, we have $D^-\isom T'[I_m]$ for some countable homogeneous tournament $T'\neq I_1$ and some $m\in\nat^\infty$.
So Lemma~\ref{lem_D+TInThenD-TIn} implies $D^+\isom D^-$ and hence Proposition~\ref{prop_TIn} covers this situation.

\subsection{\boldmath $D^+\isom I_n[T]$ with $T\neq I_1$}\label{sec_InT}

In this section, let $D$ be a countable connected C-homogeneous digraph with $D^+\isom I_n[T]$ for some countable homogeneous tournament $T\neq I_1$ and some $n\in\nat^\infty$ with $n\geq 2$.
A direct consequence of the previous sections together with the fact that $T$ contains some edge is the following lemma:

\begin{lem}\label{lem_InTD-}
We have $D^-\isom I_m[T']$ for some $m\in\nat^\infty$ and some countable homogeneous tournament $T'\neq I_1$.\qed
\end{lem}

Our next lemma says that $T$ and~$T'$ are infinite tournaments.
Note that we do not know so far whether $m>1$ or not.
We will see this in Lemma~\ref{lem_InT_m=n}.

\begin{lem}\label{lem_InTNotC3}
We have $T\neq C_3\neq T'$.
\end{lem}

\begin{proof}
Seeking for a contradiction, let us suppose $T=C_3$.
Let $xy\in ED$ and let $a,b\in N^+(x)$ with $ya,ab,by\in ED$.
Let $z$ be a common predecessor of~$x$ and~$y$.
Considering $D^-(y)$, the vertices $z$ and~$b$ lie in the same component, which is a tournament.
Thus, they are adjacent.
As an edge $zb$ gives us a transitive triangle $D[x,y,b]$ in $D^+(z)$ and as this is not possible, we have $bz\in ED$.
Hence, the directed triangle $D[a,b,y]\sub D^+(x)$ contains one successor and one predecessor of~$z$.
So if the third vertex is either a successor or a predecessor of~$z$, then we can find an automorphism of~$D$ that fixes $x$ and~$z$ and rotates the directed triangle $D[a,b,y]$.
More precisely, the automorphism maps $a$ to either $b$ or~$y$ and hence it must leave the component $D[a,b,y]$ of~$D^+(x)$ invariant.
Applying the same automorphism once more, we obtain that the whole triangle $D[a,b,y]$ lies either in $D^+(z)$ or in $D^-(z)$.
As neither of these two cases can occur, the third vertex of $D[a,b,y]$ is not adjacent to~$z$.

Thus, in the directed triangle $D[a,b,y]\sub D^+(x)$, we find a predecessor of~$z$, a~successor of~$z$ and a vertex not adjacent to~$z$.
By C-homogeneity, we find the same in each directed triangle in~$D^+(x)$.
Indeed, if $u$ is a vertex in another directed triangle in~$D^+(x)$, then we have $D[u,x,z]\isom D[v,x,z]$ for some $v\in\{a,b,y\}$.
Thus, $x$ together with $n\geq 2$ independent successors lies in~$D^+(z)$, which is impossible.
This shows $T\neq C_3$.
So $T$ is an infinite tournament and $D^+(x)\cap D^-(y)$ contains a transitive triangle.
Thus, we also have $T'\neq C_3$.
\end{proof}

Now we can describe the structure of the neighbourhood of any vertex:

\begin{lem}\label{lem_InT_m=n}
For every $x\in VD$, the digraph $D^+(x)+D^-(x)$ is a disjoint union of isomorphic homogeneous tournaments. 
Each of its components consists of one component of~$D^+(x)$ and one component of~$D^-(x)$.

In particular, we have $m=n$.
\end{lem}

\begin{proof}
For every $u\in N^-(x)$, there is a unique component of~$D^+(x)$ that contains successors of~$u$ because of $D^+(u)\isom I_n[T]$.
We denote this component by~$A_u$.

The first step is to show
\begin{txteq}\label{eq_InT_m=n1}
$A_u=A_v$ for all adjacent vertices $u,v\in N^-(x)$.
\end{txteq}
We may assume $uv\in ED$.
Since $T$ is infinite by Lemma~\ref{lem_InTNotC3}, it contains a transitive triangle.
Hence, there is a vertex $y\in N^+(x)\cap N^+(v)$ in $D^+(u)$.
This vertex~$y$ already shows us~$A_u=A_v$.

By C-homogeneity, there is for every component $C$ of~$D^+(x)$ some vertex $v\in D^-(x)$ with $C=A_v$.
Thus, (\ref{eq_InT_m=n1}) implies $n\leq m$.
Symmetrically, we obtain $m\leq n$.
Hence, we have $n=m$.

Let us show
\begin{txteq}\label{eq_InT_m=n2}
$N(v)\cap N^+(x)\sub A_v$ for every $v\in N^-(x)$.
\end{txteq}
Since $D^+(v)\cap D^+(x)$ is a tournament, we have $N^+(v)\cap N^+(x)\sub A_v$.
Let us suppose $N(v)\cap N^+(x)\not\sub A_v$.
Then we find a vertex $y\in N^+(x)\cap N^-(v)$ that lies outside $A_v$.
Let $C_v$ be the component of~$D^-(x)$ that contains~$v$.
Note that $y$ has no predecessor in~$C_v$ as $y\notin A_v$ and due to~(\ref{eq_InT_m=n1}).
If $v$ is the unique successor of~$y$ in~$C_v$, then we can find an automorphism of~$D$ that fixes $x$ and~$y$ and maps some predecessor $v^-$ of~$v$ in~$C_v$ to some successor $v^+$ of~$v$ in~$C_v$.
Note that neither $v^-$ nor $v^+$ is adjacent to~$y$ as we already mentioned.
This automorphism fixes $C_v$ setwise, so it must fix $v$, the unique neighbour of~$y$ in~$C_v$.
But we have $(v^-v)\alpha=v^+v\notin ED$ even though $v^-v\in ED$.
This shows that $y$ has a second successor $u\neq v$ in~$C_v$.
As $u$ and~$v$ are adjacent, we have $A_u=A_v$ by~(\ref{eq_InT_m=n1}).
Hence, we may assume $uv\in ED$.
By C-homogeneity, we find an automorphism~$\alpha$ of~$D$ that maps $yu$ to~$vx$.
Then $v$ has a predecessor $x\alpha$ in $N^+(x)$ that is adjacent to $v\alpha\in A_v$.
As $A_v$ contains some predecessor of~$v$, C-homogeneity implies that it contains every predecessor of~$v$ in $N^+(x)$ in contradiction to $y\notin A_v$.
Indeed, we find an automorphism that fixes~$x$ and~$v$ and maps $x\alpha$ to~$y$ and this automorphism does not fix $A_v$ setwise even though it fixes $x$ and~$v$.
This shows~(\ref{eq_InT_m=n2}).

Next, we show
\begin{txteq}\label{eq_InT_m=n3}
$A_v=N(v)\cap N^+(x)$ for every $v\in N^-(x)$.
\end{txteq}
If $A_v$ contains some vertex $y$ that is not adjacent to~$v$, then, by C-homogeneity, some automorphism of~$D$ maps $y$ to some vertex $z$ in $N^+(x)\sm A_v$ and fixes $x$ and~$v$.
Note that $z$ exists because of $n\geq 2$.
But then this automorphism does not fix $A_v$ setwise even though it fixes~$x$ and~$v$.
This contradiction shows~(\ref{eq_InT_m=n3}).

By symmetric arguments, there is for every $u\in N^+(x)$ a component $B_u$ of~$D^-(x)$ with $B_u=N(u)\cap N^-(x)$ and for each two vertices $u,v$ in the same component of $D^+(x)$, the components $B_u$ and $B_v$ coincide.
Thus, $D^+(x)+D^-(x)$ is a disjoint union of isomorphic tournaments and each component of $D^+(x)+D^-(x)$ consists of precisely one component of $D^+(x)$ and one component of $D^-(x)$.
That every component of $D^+(x)+D^-(x)$ is homogeneous is a direct consequence of C-homogeneity.
\end{proof}

Note that with Lemma~\ref{lem_InT_m=n}, we have completed the analysis of Section~\ref{sec_TI_n}.
Furthermore, we have all lemmas we need to finish the situation if $D^+$ is isomorphic to $I_n[T]$ for some $n\in\nat^\infty$ with $n\geq 2$ and some countable homogeneous tournament $T\neq I_1$.
(Note that the case $n=1$ was already completed in Section~\ref{sec_TI_n}.)

\begin{prop}\label{prop_InT}
If $D$ is a countable connected C-homogeneous digraph with $D^+\isom I_n[T]$ for some countable homogeneous tournament $T\neq I_1$ and some $n\in\nat^\infty$ with ${n\geq 2}$, then $D\isom X_{\lambda}(T')$ for some countable infinite homogeneous tournament~$T'$ and for some countable cardinal~$\lambda\geq 2$.
\end{prop}

\begin{proof}
For $x\in VD$, let $D_x:=D^+(x)+D^-(x)$.
Due to Lemma~\ref{lem_InT_m=n}, the digraph $D_x$ is a disjoint union of isomorphic infinite tournaments.
First, we show that
\begin{txteq}\label{eq_InT1}
for every $x\in VD$, no two components of~$D_x$ lie in the same component of $D-x$.
%$D-x$ is not connected for every $x\in VD$.
\end{txteq}
Let us suppose that we find a path in $D-x$ between vertices in distinct components of~$D_x$.
Let $P$ be such a path of minimal length and let $u$ and~$v$ be its end vertices.
If $ux\in ED$, let $a$ and~$b$ two vertices in~$N^+(u)$ such that $a\in N^-(x)$ and~$b\in N^+(x)$.
If $xu\in ED$, we choose $a$ and~$b$ in~$N^-(u)$ such that $a\in N^-(x)$ and $b\in N^+(x)$.
These vertices exist as $D^+(u)$ and $D^-(u)$ are disjoint unions of homogeneous tournaments.
If~$a$ or~$b$ has a neighbour $c$ on~$P$ other than~$u$, this neighbour must be the neighbour of~$u$ on~$P$ by the minimality of~$P$.
But then $a$, $b$, $c$, and~$x$ lie in the same component of~$D_u$, which is a tournament.
So $c$ is already adjacent to~$x$, which contradicts the minimality of~$P$.
Hence, the paths $vPua$ and $vPub$ are isomorphic and, by C-homo\-geneity, we can find an automorphism $\alpha$ of~$D$ that maps the first onto the second path by fixing $P$ pointwise and mapping $a$ to~$b$.
Since $a$ lies in~$N^-(x)$ and~$b$ lies in $N^+(x)$, we have $x\neq x\alpha$.
But as $x\alpha$ is adjacent to~$u$ and to~$b$, it lies in the same component of~$D_u$ as~$x$.
So $x$ and~$x\alpha$ are adjacent and $x\alpha$ lies in the same component of~$D_x$ as~$a$ and~$b$.
Since $x\alpha$ is a neighbour of~$v=v\alpha$, also $v$ lies in the same component of~$D_x$ as~$x\alpha$ and thus the vertices $u$ and~$v$ are adjacent.
This contradiction to the choice of~$u$ and~$v$ shows~(\ref{eq_InT1}).

For every $x\in VD$, each component of~$D_x$ is an infinite tournament and hence contains a ray.
Rays from distinct components of~$D_x$ cannot be equivalent as they lie in distainct components of~$D-x$ due to~(\ref{eq_InT1}).
Hence, $D$ has at least two ends.
Thus, the assertion follows from Theorem~7.6 in~\cite{HH-ConHomDigraphs}, the classification result of connected C-homogeneous digraphs with more than one end.
\end{proof}

\subsection{A first result}

By summarizing the propositions of the previous sections together with Cherlin's classification of the homogeneous digraphs, Theorem~\ref{thm_CountHomDigraphs}, we obtain the following theorem:

\begin{thm}\label{thm_Part1}
Let $D$ be a countable connected C-homogeneous digraph.
Then one of the following cases holds:
\begin{enumerate}[{\rm (i)}]
\item $D$ is homogeneous;
\begin{comment}{Der folgende Punkt f\"allt schon unter $D$ ist homogen:
\item $D\isom T[I_n]$ for some $n\in\nat^\infty$ and some homogeneous tournament~$T$;
}\end{comment}
\item $D\isom T^\wedge[I_n]$ for some $n\in\nat^\infty$ and some tournament~$T\in\{C_3,\rat,T^\infty\}$;
\item $D\isom S(3)[I_n]$ for some $n\in\nat^\infty$;
\item $D\isom X_{\lambda}(T)$ for some countable infinite homogeneous tournament~$T$  and for some countable cardinal~$\lambda \geq 2$; or
\item $D^+\isom I_n$ and $D^-\isom I_m$ for some $m,n\in\nat^\infty$.\qed
\end{enumerate}
\end{thm}

\section{The case: $D^+\isom I_n$ and $D^-\isom I_m$}\label{sec_ProofPartII}

Throughout this section, let $D$ be a countable connected C-homogeneous digraph with $D^+\isom I_n$ for some $n\in\nat^\infty$.
By the previous sections, we also have $D^-\isom I_{n'}$ for some $n'\in\nat^\infty$.

The following lemma is already proven in~\cite{FinConHomDigraphs}.
Therefore, we omit its proof here.

\begin{lem}\cite[Lemma~5.1]{FinConHomDigraphs}\label{lem_d+1Tree}
If $d^+=1$ or $d^-=1$, then $D$ is either an infinite tree or a directed cycle.\qed
\end{lem}

Since connected C-homogeneous digraphs with more than one end have already been classified~\cite{GM-CHomDigraphs,HH-ConHomDigraphs}, we assume for the remainder of this section that $D$ contains at most one end.

In~\cite[Lemmas 5.2 and 5.5]{FinConHomDigraphs}, the author showed that the reachability relation of every locally finite C-homogeneous digraph with at most one end and whose out-neighbourhood is independent is not universal.
If we consider such digraphs of arbitrary degree, this does no longer hold.
For example, the countable generic $2$-partite digraph is a C-homogeneous digraph with independent out-neighbourhood and with precisely one end and its reachability relation is universal.
In the following, we distinguish the two cases whether the reachability relation $\AF$ of~$D$ is universal or not.

\subsection{\boldmath Non-universal reachability relation}\label{sec_D+IndepAFBipartite}

Within this section, let $D$ be a countable connected C-homogeneous digraph with $D^+\isom I_n$ for some $n\in\nat^\infty$, with $D^-\isom I_{n'}$ for some $n'\in\nat^\infty$, with at most one end.
We assume that $\AF$ is not universal and, due to Lemma~\ref{lem_d+1Tree}, that $n,n'\geq 2$.
Hence, we obtain by Proposition~\ref{prop_CPW} that $\Delta(D)$ is bipartite.
That is the reason, why we turn our attention towards the classification of the C-homogeneous bipartite graphs.
The following lemma due to Gray and M\"oller~\cite{GM-CHomDigraphs} underlines our interest in the C-homogeneous bipartite graphs.

\begin{lem}\label{lem_GM4.3}\cite[Lemma 4.3]{GM-CHomDigraphs}
The digraph $\Delta(D)$ is a connected C-homoge\-neous bipartite digraph.\qed
\end{lem}

By Lemma~\ref{lem_GM4.3}, we know that $G(\Delta(D))$ belongs to one of the five classes described in Theorem~\ref{thm_HHThm6.4}.
In the following, we will treat these five possibilities one by one.
Let us start with the case $G(\Delta(D))\isom C_{2m}$ for some $m\geq 2$, where we notice that $D$ must be locally finite as every vertex lies in at most two reachability digraphs:

\begin{lem}\label{lem_In_DeltaCycle}
If $G(\Delta(D))\isom C_{2m}$ for some $m\geq 2$, then $D$ is locally finite.\qed
\end{lem}

Thus, if $G(\Delta(D))$ is an even cycle, then we obtain this part of the classification from Theorem~2.1 of~\cite{FinConHomDigraphs}.
In the following, we assume $G(\Delta(D))\not\isom C_{2m}$ for any $m\in\nat$.
Since locally finite C-homogeneous digraphs have already been classified, we may assume in the following that either $d^+=\omega$ or $d^-=\omega$.
By reversing the directions of each edge if necessary, we may assume $d^+=\omega$.

For a reachability digraph $\Delta$ of~$D$, two vertices or a set of vertices  of~$\Delta$ \emph{lie on the same side of~$\Delta$} if their out-degree, and hence also their in-degree, in~$\Delta$ is the same.

\begin{lem}\label{lem_bipReachInterGeq2}
For each two reachability digraphs $\Delta_1$ and~$\Delta_2$ of~$D$ we have either ${\Delta_1\cap\Delta_2=\es}$ or $|V(\Delta_1\cap\Delta_2)|\geq 2$.
\end{lem}

\begin{proof}
Let us suppose that the intersection of two distinct reachability digraphs $\Delta_1$ and $\Delta_2$ consists of precisely one vertex.
Since every vertex lies in precisely two reachability digraphs and since $D$ is vertex-transitive, each two distinct reachability digraphs either have trivial intersection or share precisely one vertex.

\smallskip

We distinguish the cases whether $C_3$ embeds into~$D$ or not.
First, we assume that $D$ contains no directed triangle.
Let $xy\in ED$ and $\Delta=\langle\AF(xy)\rangle$.
If $G(\Delta)\not \isom CP_k$, let $P$ be any path of minimal length from any successor~$u$ of~$y$ to~$x$ avoiding~$y$.
Such a path exists as the one-ended digraph $D$ cannot contain any cut-vertex.
If $G(\Delta)\isom CP_k$, let $P$ be any path of minimal length from any successor~$u$ of~$y$ to~$x$ that avoids~$y$ and the unique neighbour $\bar{y}$ of~$y$ in the bipartite complement of~$\Delta$.
As $k=d^+=\omega$, both of the two reachability digraphs $\langle\AF(yu)\rangle$ and $\Delta$ contain rays that avoid $y$ and~$\bar{y}$ and hence $y$ and~$\bar{y}$ separate neither these rays nor $u$ from~$x$.
Thus, we also know in this situation that $P$ exists.

By the minimality of~$P$, the only successor of~$y$ on~$P$ is~$u$.
If $y$ has a predecessor~$x'$ on~$P$, then $xyu$ and $x'yu$ are induced $2$-arcs, so we find an automorphism of~$D$ that maps one onto the other and we obtain a contradiction to the minimality of~$P$.
Thus, $y$ has no neighbour on~$P$ except for~$u$ and~$x$.
At most $|VP|$ vertices of~$\Delta$ that lie on the same side as~$y$ can have successors on~$P$, since any two such vertices with a common successor on~$P$ would lie in two common reachability digraphs.
Since $N^+(x)$ contains infinitely many vertices, all of which lie on the same side of~$\Delta$ as~$y$, we find one such vertex~$z$ that has no successor on~$P$.
If $G(\Delta)$ is either complete bipartite or the bipartite complement of a perfect matching, then every predecessor of~$z$ on~$P$ is also a predecessor of~$y$ by the assumption that in the case $G(\Delta)\isom CP_k$ the path $P$ does not contain $\bar{y}$.
Hence, $P$ contains predecessors of~$z$ only if $G(\Delta)$ is the generic bipartite graph or a tree $T_{k,\ell}$.
Note that any predecessor of~$z$ on~$P$ is a predecessor in~$\Delta$ of~$z$.
Thus, in these two cases we may have chosen $z$ among the infinitely many vertices of~$N^+(x)$ that have no predecessor on~$P$.
Let $v$ be the neighbour of~$u$ on~$P$.
Then both vertices $y$ and~$z$ have only one neighbour on~$vPx$, the vertex~$x$.
By C-homogeneity, we find an automorphism $\alpha$ of~$D$ that fixes $vPx$ and interchanges $y$ and~$z$.
Let $w=u\alpha$.

If $vu\in ED$, then $v$ and~$y$ lie on the same side of $\langle\AF(yu)\rangle$ and on this side lies also $y\alpha=z$ as $(vu)\alpha=vw$.
But then $y$ and~$z$ lie in two common reachability digraphs which contradicts the assumption.
Hence, we have $uv\in ED$ and $wv\in ED$.
The two $2$-arcs $xyu$ and $xzw$ induce a digraph that consists only of these two $2$-arcs:
as~$z$ and~$u$ are not adjacent, neither are $y=z\alpha$ and $w=u\alpha$.
Note that no successor of~$y$ can have $w$ or~$z$ as a predecessor because otherwise either $w$ or~$z$ lies in the two reachability digraphs $\langle\AF(yu)\rangle$ and either $\langle\AF(wv)\rangle$ or $\Delta$, which is impossible by assumption.
By the same assumption and similar as above, only finitely many successors of~$u$ have successors on the $1$-arc $zw$.
Since $d^+=\omega$, we find a vertex $u'\in N^+(y)$ that is adjacent to neither $w$ nor~$z$.
Note that $u'$ and~$x$ are not adjacent since $D$ contains no triangle.
Hence, we find by C-homogeneity an automorphism $\beta$ of~$D$ that fixes $D[x,y,z,w]$ pointwise and maps $u$ to~$u'$.
So $u'$ and~$w$ have a common successor $v\beta$ and thus $u$ and~$u'$ lie on the same side of $\langle\AF(uv)\rangle$ and of $\langle\AF(yu)\rangle$.
This contradiction shows the assertion in the situation that $C_3$ does not embed into~$D$.

\smallskip

Now we consider the case that $D$ contains a directed triangle.%
\begin{comment}{If we consider the situation of the previous case, the path $P$ consists of only one vertex and hence we cannot find the automorphism~$\beta$. And even if we require that this vertex does not lie on~$P$, we cannot show directly (at the moment I can't do it at all) that $P$ contains no other neighbour of~$y$ but $x$ and~$u$.}\end{comment}
For every edge $xy$ those successors of~$y$ that are predecessors of~$x$ lie in two common reachability digraphs.
As the intersection of two distinct reachability digraphs contains at most one vertex, we obtain that
\begin{txteq}\label{eq_bipReachInterGeq2_1}
every edge lies on precisely one directed triangle.
\end{txteq}
We distinguish whether $G(\Delta(D))$ is a semi-regular tree or not.
First, we consider the case $G(\Delta(D))\isom T_{k,\ell}$ for some $k,\ell\in\nat^\infty$ with $k,\ell\geq 2$.
Let $x\in VD$ and let $P$ be a shortest path in $G-x$ between any two successors $y$ and~$z$ of~$x$.
Since $P$ must contain some edge that does not lie in $\langle\AF(xy)\rangle$ and since any two distinct reachability digraphs intersect in at most one vertex, $P$ contains some vertex outside $\langle\AF(xy)\rangle$.
Thus and by the assumption on the intersection of any two distinct reachability digraphs, $P$ has at least three edges.
Let $z_2,z_1,z$ be the last three vertices of~$P$.
Let $a$ be a third successor of~$x$.
This vertex exists as $d^+=\omega$.
By minimality of~$P$, it contains no neighbour of~$a$ as otherwise we find a shorter path between $a$ and either~$y$ or~$z$, since neither $a$ and~$y$ nor $a$ and~$z$ have a common predecessor, as they lie in only one common reachability digraph.
Hence, the connected subdigraphs $zxyPz_2$ and $axyPz_2$ are isomorphic and we find an automorphism $\alpha$ of~$D$ that fixes $xyPz_2$ and interchanges~$a$ and~$z$, as $D$ is C-homogeneous.
So we obtain that $D':=D[z,z_1,z_2,z_1\alpha,a]$ consists of four edges and, with $z_1':=z_1\alpha$, we have $zz_1\in ED$ if and only if $az'_1\in ED$ and the same for $z_1z_2$ and $z'_1z_2$.
Since the intersection of any two reachability digraphs contains at most one vertex, the path $D'$ is not an alternating walk.
Thus, $D'$ consists of two induced $2$-arcs.
If these are $z_2z_1z$ and $z_2z_1'a$, then $z_1$ and $z_1'$ lie in the intersection of the two reachability digraphs $\langle\AF(xz)\rangle$ and $\langle\AF(z_2z_1)\rangle$.
Thus, these $2$-arcs must be $zz_1z_2$ and $az'_1z_2$.
If~$x$ and~$z_1$ are adjacent, then the edge between them must be $z_1x$ since $N^+(x)$ is independent.
But then, we have $z_1'x\in ED$, too, and $D[x,z_1,z_2,z_1']$ is a cycle in $\langle\AF(xz)\rangle$, which is impossible.
Similarly, $x$ and~$z_1'$ are not adjacent.
Hence, the digraph $D[x,z,z_1,a,z'_1]$ consists of only the two induced $2$-arcs $xzz_1$ and $xaz'_1$ and we can proceed as in the case that $C_3$ does not embed into~$D$ to obtain a contradiction with the additional requirement that $u'$ is not adjacent to~$x$, which is possible as only one successor of~$z$ is adjacent to~$x$ by~(\ref{eq_bipReachInterGeq2_1}) and $d^+=\omega$.

It remains to consider the case that $G(\Delta(D))$ is not a semi-regular tree.
Due to the structure of $G(\Delta(D))$, both sides of each reachability digraph have the same cardinality.
As $d^+=\omega$, we also have ${d^-=\omega}$.
Let $x\in VD$ and $y$ and~$z$ be two vertices in~$N^+(x)$.
Let $u$ and~$v$ be the unique successors of~$y$ and~$z$, respectively, that lie on a common directed triangle with~$x$, see~(\ref{eq_bipReachInterGeq2_1}).
Since each edge lies on a unique (directed) triangle, every common successor $w\neq x$ of~$u$ and~$v$ is adjacent to neither $y$ nor~$z$.
As $d^-=\omega$ and due to~(\ref{eq_bipReachInterGeq2_1}), we find $a\in N^-(v)$ that is adjacent to neither $w$ nor~$x$.
An edge $au$ implies that $u$ and~$v$ lie in two common reachability digraphs and an edge $ua$ leads to a cycle $D[a,u,w,v]$ witnessing that $\AF$ is universal.
As both situations are impossible, $a$ and~$u$ are not adjacent.
Furthermore, $az$ cannot be an edge because then $D[a,v,y,x,z]$ is a cycle witnessing that $\AF$ is universal.
As this is not the case, we have $az\notin ED$.
Let us suppose that $za$ is an edge of~$D$.
Then by C-homogeneity, we find an automorphism $\alpha$ of~$D$ that fixes $w$ and maps $zu$ to~$av$ and $v$ to~$u$.
Note that $b:=a\alpha\neq z$ since $za\in ED$ but $ba=(az)\alpha\notin ED$.
As $bu\in ED$, the digraph $D[a,b,u,z]$ is a cycle witnessing that $\AF$ is universal.
This contradiction shows that $z$ and~$a$ are not adjacent.
So we find an automorphism $\beta$ of~$D$ that fixes $z,u,w,v$ and maps $y$ to~$a$, as $D$ is C-homogeneous.
Thus, $x\beta\neq x$ is a common predecessor of~$a$ and~$z$.
So $a$ lies in $\langle\AF(xy)\rangle$ on the same side as~$z$.
Thus, $a$ and~$y$ lie in two common reachability digraphs in contradiction to the assumption.
\end{proof}

Now we are able to complete the investigation if $G(\Delta(D))$ is a semiregular tree:

\begin{lem}\label{lem_InWithT_kl}
If $G(\Delta(D))\isom T_{k,\ell}$ for some $k,\ell\in\nat^\infty$ with $k,\ell\geq 2$, then $D$ either is locally finite or has more than one end.
\end{lem}

\begin{proof}
Let us assume that $D$ is not locally finite.
By reversing the direction of each edge, we may assume $k=d^+=\omega$.
Let us suppose that $D$ has at most one end.
First, we show that
\begin{txteq}\label{eq_InWithT_kl1}
the intersection of two distinct reachability digraphs lies on the same side of each of them.
\end{txteq}
Let us suppose that this is not the case.
As $D$ is vertex-transitive, each two reachability digraphs with non-trivial intersection are a counterexample to~(\ref{eq_InWithT_kl1}).
Let $\Delta_1$ and~$\Delta_2$ be two distinct reachability digraphs with non-trivial intersection.
By Lemma~\ref{lem_bipReachInterGeq2}, their intersection contains at least two vertices.
Since $V(\Delta_1\cap\Delta_2)$ does not lie on the same side of~$\Delta_1$, we find two vertices $x,y\in V(\Delta_1\cap\Delta_2)$ of odd distance in~$\Delta_1$ such that $x$ has no successors in~$\Delta_1$.
Let $z$ be the predecessor of~$x$ on the unique $x$--$y$ path $P$ in~$\Delta_1$.
Since $d^+=\omega$, we find a successor $x'$ of~$z$ that does not lie on~$P$.
Then the digraph $x'zPy$ is isomorphic to~$P$ and, by C-homogeneity, we find an automorphism of~$D$ that fixes $zPy$ and maps $x$ to~$x'$.
So we conclude that $x'$ lies also in the same two reachability digraphs as~$y$.
Hence, the two vertices $x$ and~$x'$ of distance $2$ lie on the same side of~$\Delta_1$ and of~$\Delta_2$.
Inductively, all vertices of~$\Delta_1$ that lie on the same side of~$\Delta_1$ as~$x$, also lie in~$\Delta_2$.
In particular, this holds for some successor $y'$ of~$y$.
Hence, $\Delta_1$ and $\Delta_2$ share all vertices of~$D$.
For an edge $ab\in E\Delta_2$ the $a$--$b$ path in~$\Delta_1$ is an alternating walk.
Thus, $Q$ together with the edge $ab$ is a cycle witnessing that $\AF$ is universal.
This contradiction to the assumptions shows~(\ref{eq_InWithT_kl1}).

For the remainder of the proof, we fix two reachability digraphs $\Delta_1$ and $\Delta_2$ with non-trivial intersection such that the vertices in~$\Delta_1\cap\Delta_2$ have no successor in~$\Delta_1$.

With the same argument as in the proof of~(\ref{eq_InWithT_kl1}), just taking a path $P$ of even length, we obtain that
\begin{txteq}\label{eq_InWithT_kl2}
every vertex on the same side of~$\Delta_1$ as $V(\Delta_1\cap\Delta_2)$ lies in~$\Delta_2$.

The analogous property for~$\Delta_2$ holds as soon as $\ell\geq 3$.
\end{txteq}

For the remainder of the proof, let $x\in V\Delta_1\sm V\Delta_2$.
Next, we show that
\begin{txteq}\label{eq_InWithT_kl3}
no vertex of~$N^+(x)$ separates in~$\Delta_2$ any other two vertices of~$N^+(x)$.
\end{txteq}
To show this, we suppose that $y_1\in N^+(x)$ separates in~$\Delta_2$ the two vertices $y_2,y_3\in N^+(x)$.
By C-homogeneity and as $N^+(x)$ is independent, we find an automorphism of~$D$ that fixes $x$ and~$y_3$ and switches $y_1$ and~$y_2$.
This automorphism fixes $\Delta_2$ setwise and we obtain that $y_2=y_1\alpha$ separates in~$\Delta_2$ the vertices $y_1=y_2\alpha$ and $y_3=y_3\alpha$ which is clearly impossible.
This contradiction shows~(\ref{eq_InWithT_kl3}).

Let us show that
\begin{txteq}\label{eq_InWithT_kl4}
$D$ contains some directed triangle.
\end{txteq}
Let us suppose that $D$ contains no directed triangle.
Let $y\in N^+(x)$ and let $z_1,z_2\in N^+(y)$ such that $z_1$ is the neighbour of~$y$ in that component of~$\Delta_2-y$ that contains all other successors of~$x$.
Then the two $2$-arcs $xyz_1$ and $xyz_2$ are induced and we obtain an automorphism $\alpha$ of~$D$ that fixes $x$ and~$y$ and maps $z_1$ to~$z_2$, as $D$ is C-homogeneous.
Thus, $\alpha$ does not fix the unique component of $\Delta_2-y$ that contains all successors of~$x$.
This is impossible and hence we have shown~(\ref{eq_InWithT_kl4}).

Let $y\in N^+(x)$ and let $z\in N^+(y)$ such that $z$ lies in the unique component of~$\Delta_2-y$ that contains all successors of~$x$ but~$y$, see~(\ref{eq_InWithT_kl3}).
By the same argument as in the proof of~(\ref{eq_InWithT_kl3}) we obtain that
\begin{txteq}\label{eq_InWithT_kl5}
either $z$ is the only successor of~$y$ such that $D[x,y,z]$ is a directed triangle or $z$ is the only successors of~$y$ such that $D[x,y,z]$ is an induced $2$-arc.
\end{txteq}

If $D[x,y,z]$ is a directed triangle, then every edge of~$D$ lies on a unique directed triangle due to~(\ref{eq_InWithT_kl5}).
So the number of directed triangles that contain a given vertex is $d^+$ and it is also $d^-$.
Hence, we obtain $d^-=d^+=\omega$.
If $D[x,y,z]$ is an induced $2$-arc, then the edge $xy$ lies on infinitely many directed triangles as $D^+=\omega$ and by~(\ref{eq_InWithT_kl5}).
Thus, $x$ must have infinitely many predecessors and we obtain $d^-=d^+=\omega$ in this case, too.
Hence, we have $\ell\geq 3$ and the second part of~(\ref{eq_InWithT_kl2}) holds.
Thus, there are two reachability digraphs distinct from $\Delta_2$ that cover the vertices of~$\Delta_2$.
So the vertices of $\Delta_2-\Delta_1$ lie in a reachability digraph $\Delta_0\neq\Delta_1$. 
Since $C_3$ embeds into~$D$, we have
\[
\Delta_1-\Delta_2=\Delta_0\cap\Delta_1=\Delta_0-\Delta_2.\]
As $D$ is connected, we conclude that $\Delta_0,\Delta_1,$ and $\Delta_2$ are the only reachability digraphs of~$D$.

The next step is to show that $D[x,y,z]$ is not an induced $2$-arc:
\begin{equation}\label{eq_InWithT_kl6}
D[x,y,z]\isom C_3.
\end{equation}
If~(\ref{eq_InWithT_kl6}) does not hold, then $xyz$ is an induced $2$-arc and, by~(\ref{eq_InWithT_kl5}), unique with the property that $xy$ is its first edge.
Let $x'\in VD$ such that $yzx'$ is the unique induced $2$-arc with $yz$ as its first edge.
Then we have $x'\in V(\Delta_0\cap\Delta_1)$ and $x$ and~$x'$ lie on the same side of~$\Delta_1$.
Note that $xy$ already determines the vertex~$x'$.
So the stabilizer of the edge $xy$ must fix~$x'$.
Let $u$ be the first vertex on the unique $x$--$x'$ path in~$\Delta_1$ that is neither $x$ nor~$y$.
Let $v$ be another neighbour of~$x$, if $u$ is a neighbour of~$x$, and let $v$ be another neighbour of~$y$ otherwise.
Then we find an automorphism of~$D$ that fixes the edge $xy$ and maps $u$ to~$v$ which is clearly impossible as this automorphism does not fix~$x'$.
This shows~(\ref{eq_InWithT_kl6}).

Let us now show that $D[x,y,z]$ cannot be a directed cycle, either, which will be our desired contradiction.
To simplify notations, let $x_0=z$, $x_1=x$, $x_2=y$.
Let $F_i, G_i$ be the component of~$\Delta_i-x_ix_{i+1}$ that contains~$x_i, x_{i+1}$, respectively (we consider the indices modulo~$3$).
Let $u\in F_1\cap V(\Delta_1\cap\Delta_2)$.
Then we find a second vertex $v$ in~$F_1\cap V(\Delta_1\cap\Delta_2)$ that has distance $d_{\Delta_1}(x_2,u)$ to each of~$x_2$ and~$u$ because of $d^+\neq 2\neq d^-$, where $d_{\Delta_1}$ denotes the distance in~$\Delta_1$.
Let $w\in F_1$ be the unique vertex in~$F_1$ that has the same distance to each of $x_2,u,v$.
By C-homogeneity, we find an automorphism that fixes the unique $w$--$u$ path in~$\Delta_1$ and maps the unique $w$--$x_2$ path in~$\Delta_1$ onto the unique $w$--$v$ path in~$\Delta_1$ and vice versa.
As in the proof of~(\ref{eq_InWithT_kl3}), we obtain that $x_2$ does not separate $u$ and~$v$ in~$\Delta_2$.
So $u$ and~$v$ must lie in the same component $C$ of~$\Delta_2-x_2$.
Thus, all vertices $a$ of~$F_1\cap V(\Delta_1\cap \Delta_2)$ with $d_{\Delta_1}(a,x_2)=d_{\Delta_1}(x_2,u)$ lie in~$C$.
Let us suppose $C\sub F_2$.
Since there are infinitely many components of~$\Delta_2-x_2$ in~$F_2$, we find one neighbour $b_1$ of~$x_2$ in~$C$ and one neighbour $b_2$ in another component of~$F_2\cap V(\Delta_2-x_2)$.
Both digraphs $x_1x_2b_1$ and $x_1x_2b_2$ are induced $2$-arcs as neither $b_1$ nor~$b_2$ is~$x_0$ and due to~(\ref{eq_InWithT_kl5}).
By C-homogeneity, we find an automorphism $\alpha$ of~$D$ that fixes $x_1x_2$ and maps $b_1$ to~$b_2$.
Thus, $\alpha$ cannot fix $C$ setwise even though it fixes $F_1\cap V(\Delta_1\cap\Delta_2)$ setwise.
This contradiction shows $C\sub G_2$.
Thus, we have
\[F_1\cap V(\Delta_1\cap\Delta_2)\sub G_2\cap V(\Delta_1\cap\Delta_2).\]
By a symmetric argument, we obtain
\[F_1\cap V(\Delta_1\cap\Delta_2)= G_2\cap V(\Delta_1\cap\Delta_2).\]
Analogously, we obtain
\[F_i\cap V(\Delta_i\cap\Delta_{i+1})= G_{i+1}\cap V(\Delta_i\cap\Delta_{i+1})\]
for all $i$ and hence also
\[G_i\cap V(\Delta_i\cap\Delta_{i+1})= F_{i+1}\cap V(\Delta_i\cap\Delta_{i+1}).\]
Let $D[a,b,c]$ be a directed triangle with $a\in F_1\cap V(\Delta_0\cap\Delta_1)$ that is disjoint from $D[x,y,z]$.
Then we have
\[b\in F_1\cap V(\Delta_1\cap\Delta_2)=G_2\cap V(\Delta_1\cap\Delta_2)\]
and hence
\[c\in G_2\cap V(\Delta_2\cap\Delta_0)=F_0\cap V(\Delta_2\cap\Delta_0)\]
and
\[a\in F_0\cap V(\Delta_0\cap\Delta_1)=G_1\cap V(\Delta_0\cap\Delta_1).\]
So $ab$ is an edge in~$\Delta_1$ between vertices of distinct components of $\Delta-x_1x_2$, which is impossible.
This contradiction shows that $D$ has more than one end.
\end{proof}

Thus, we can go through the list of locally finite C-homogeneous digraphs, Theorem~2.1 in~\cite{FinConHomDigraphs}, and through the list of connected C-homogeneous digraphs with more than one end, Theorems 4.2 and 7.6 in~\cite{HH-ConHomDigraphs} and Theorem 6.2 in~\cite{GM-CHomDigraphs} by Gray and M\"oller, and obtain all possibilities if $G(\Delta(D))$ is a semi-regular tree.
Hence, in addition to $G(\Delta(D))\not\isom C_{2m}$ for any $m\in\nat$, we assume in the following $G(\Delta(D))\not\isom T_{k,\ell}$ for any $k,\ell\in\nat^\infty$.

\begin{lem}\label{lem_DeltaIntersectSameSide}
For each two distinct reachability digraphs $\Delta_1$ and $\Delta_2$ of~$D$, the set $V(\Delta_1\cap\Delta_2)$ lies on the same side of~$\Delta_1$.
\end{lem}

\begin{proof}
We may assume that $\Delta_1$ and $\Delta_2$ have non-trivial intersection.
Due to Lemma~\ref{lem_bipReachInterGeq2}, we have $|V(\Delta_1\cap\Delta_2)|\geq 2$.
Let us suppose that $V(\Delta_1\cap\Delta_2)$ does not lie on the same side of~$\Delta_1$.
Since $\Delta_1\cap\Delta_2$ contains no edge, $G(\Delta(D))$ is no complete bipartite graph.

If $G(\Delta(D))$ is the countable generic bipartite graph, then any two of its vertices have distance at most~$3$ in~$\Delta(D)$.
Since $V(\Delta_1\cap\Delta_2)$ does not lie on the same side of~$\Delta_1$, we find $x,y\in V(\Delta_1\cap\Delta_2)$ with $d_{\Delta_1}(x,y)=3$.
So any two vertices of distance three in~$\Delta_1$ lie in the intersection of two reachability digraphs by C-homogeneity, as we can extend them to an induced alternating path of length~$3$ within~$\Delta_1$.
This implies that all the vertices of~$\Delta_1$ lie in $\Delta_2$, which is impossible as we already saw in the proof of Lemma~\ref{lem_InWithT_kl}.
Thus, $G(\Delta(D))$ is not the countable generic bipartite graph.\looseness-1

So for the remainder of the proof, we may assume that $G(\Delta(D))\isom CP_k$ for some $k\in\nat^\infty$ with $k\geq 4$.
Since it suffices to consider the case $d^+=\omega$, we may assume $k=\omega$.
As ${\Delta_1\cap\Delta_2}$ contains two vertices of distinct sides of~$\Delta_1$ but no edge, it consists of precisely two vertices that are adjacent in the bipartite complement of~$\Delta_1$.
For the end vertices of any $2$-arc $x_1x_2x_3$, not necessarily induced, there is no $x'_2\in VD$ such that $x_1x'_2x_3$ is also a $2$-arc since otherwise $x_2$ and~$x'_2$ lie in two common reachability digraphs and on the same side of each of them, which is impossible.
In particular, every edge $y_1y_2$ lies on at most one directed triangle, since two directed triangle both of which contain $y_1y_2$ have different $2$-arcs from $y_2$ to~$y_1$.

Let $xy\in E\Delta_1$ with $y\in V\Delta_2$.
If $C_3$ embeds into~$D$, let $a$ be the unique vertex on a directed triangle with $xy$.
Otherwise, let $a$ be any successor of~$y$.
In both cases, let $a'$ (let $v$) be the unique neighbour of~$a$ (of~$y$, respectively) in the bipartite complement of~$\Delta_2$.
So we have $v\in V(\Delta_1\cap\Delta_2)$.
Since $k=\omega$ and each two distinct reachability digraphs have only two common vertices, we find a common successor~$u$ of~$x$ and~$v$ that is adjacent to neither~$a$ nor~$a'$.
Similar to the existence of~$u$, we find a vertex $b\in N^+(y)$ with $b\neq a$ such that $b$ and its unique neighbour~$b'$ in the bipartite complement of~$\Delta_2$ are adjacent to neither~$x$ nor~$u$.

Note that $\Delta_1$ contains rays avoiding $y$ and~$v$ and that the reachability digraph containing $a$ and~$a'$ that is distinct from~$\Delta_2$ contains rays avoiding $a$ and~$a'$.
As $D$ has at most one end, we find a path from each successor of~$a$ and each predecessor of~$a'$ to~$x$ such that the path avoids $a,a',b,b',y$, and~$v$.
Let $P$ be any such path of minimal length and let $c$ be its first vertex.
Note that if $C_3$ embeds into~$D$ then $P$ is the trivial path consisting only of~$x$.
By its minimality, $P$ contains no successor of~$b$ and no predecessor of~$b'$.
Indeed, if $P$ has such a vertex, then this is not~$c$, since neither $a$ and~$b$ nor $a'$ and~$b'$ lie in two common reachability digraphs and since $c\notin V\Delta_2$.
By C-homogeneity, we find an automorphism of~$D$ that fixes $xy$ and maps $b$ to~$a$ and~$b'$ to~$a'$.
This would contradict the minimality of~$P$.
Note that, if $P$ contains either a predecessor of~$b$ or a successor of~$b'$, then this is also a predecessor of~$a$ or a successor of~$a'$, respectively, and the analogue holds if $P$ contains either a predecessor of~$a$ or a successor of~$a'$.
Thus, if $ac\in ED$, we find an automorphism of~$D$ that fixes $P$ and $yxuv$ and maps $a'$ to~$b'$.
Then $yac$ and $ybc=(yac)\alpha$ are $2$-arcs with the same end vertices, which cannot exist as we already mentioned.
In the situation $ca'\in ED$, we obtain a similar contradiction by an automorphism that fixes $P$ and $yxuv$ and maps $a$ to~$b$, where we find the two $2$-arcs $ca'v$ and~$cb'v$.
\end{proof}

Now we are able to finish the situation for the cases that $G(\Delta(D))$ is either complete bipartite, or the bipartite complement of a perfect matching, or the countable generic bipartite graph.
Due to the previous classifications of C-homogeneous digraphs~\cite{GM-CHomDigraphs,FinConHomDigraphs,HH-ConHomDigraphs}, it suffices to describe those that have at most one end and are not locally finite.

\begin{lem}\label{lem_In_CP/K/Gen}
If $D$ has at most one end and is not locally finite, then it is isomorphic to one of the following digraphs:
\begin{enumerate}[\rm (i)]
\item $C_m[I_\omega]$ for some $m\in\nat^\infty$ with $m\geq 3$;
\item $Y_\omega$; or
\item $\RF_m$ for some $m\in\nat^\infty$ with $m\geq 3$.
\end{enumerate}
\end{lem}

\begin{proof}
Let us assume that $D$ has at most one end and is not locally finite.
Since $V(\Delta_1\cap\Delta_2)$ lies on the same side of~$\Delta_1$ by Lemma~\ref{lem_DeltaIntersectSameSide}, we may assume that the vertices in~$\Delta_1\cap\Delta_2$ have their predecessors in~$\Delta_1$ and their successors in~$\Delta_2$.
Let $\{A,B\}$ be the natural bipartition of~$V\Delta_1$ such that $V(\Delta_1\cap\Delta_2)\sub B$.
Since any two vertices in~$B$ have a common predecessor in~$A$, we conclude $B\sub V\Delta_2$ by C-homogeneity.
Indeed, we can map any two vertices in~$V(\Delta_1\cap\Delta_2)$ with a common predecessor onto any two vertices in~$B$ with a common predecessor, so any two vertices in~$B$ lie in two common reachability digraphs of~$D$ and hence $B\sub V\Delta_2$.
Thus, we have $B=V(\Delta_1\cap\Delta_2)$.
By an analogous argument, we obtain that every vertex on the same side of~$\Delta_2$ as~$B$ lies in~$B$.

Let $\sim$ be a relation on~$VD$ defined by
\begin{txteq}
$x\sim y\,\,\,\,:\,\Longleftrightarrow\,\,\,\, x$ and~$y$ lie on the same side of two reachability digraphs.
\end{txteq}
As we have just shown, $\sim$ is an equivalence relation on~$VD$, which is $\Aut(D)$-invariant.
Since each equivalence class is an independent set and since the reachability digraphs are bipartite, we conclude that $D_\ssim$ is a digraph.
Since every vertex of~$D$ lies in precisely two reachability digraphs, every vertex of~$D_\ssim$ has precisely one successor and one predecessor. Furthermore, $D_\ssim$ is connected.
Thus, we have $D_\ssim\isom C_m$ for some $m\in\nat^\infty$ with $m\geq 3$.
If $G(\Delta(D))\isom K_{k,\ell}$ for some $k,\ell\in\nat$, then we obtain $k=\ell$ because $B$ is one side of~$\Delta_1$ and one of~$\Delta_2$.
It is a direct consequence that $D\isom C_m[I_\omega]$ as $D$ is not locally finite.
Similarly, if $G(\Delta(D))$ is the countable generic bipartite graph, then we directly obtain $D\isom \RF_m$.
It remains to consider the case $G(\Delta(D))\isom CP_k$.
If $m\geq 4$, then we find two distinct types of induced $2$-arcs $xyz$: one whose end vertices are not adjacent to the same vertex $y'$ with $y'\sim y$ and one whose end vertices do not have this property.
Even though $D$ is C-homogeneous, we cannot map the first onto the second of these induced $2$-arcs by automorphisms of~$D$.
Thus, we have $m=3$.
Let $\overline{D}$ be the tripartite complement of~$D$.
Since the bipartite complement of each reachability digraph is a perfect matching, $\overline{D}$ is a disjoint union of directed cycles.
Let us suppose that the length of one of these cycles is more than~$3$.
Then it has length at least~$6$ and there are two $\sim$-equivalent vertices in~$\overline{D}$ that have distance~$3$ on that cycle.
Since these two $\sim$-equivalent vertices have a common predecessor, the same is true for any two $\sim$-equivalent vertices by C-homogeneity.
So each two $\sim$-equivalent vertices lie on a common directed cycle in~$\overline{D}$ and have distance $3$ on that cycle.
Hence, $\overline{D}$ consists of precisely one cycle of length at most~$9$ and $D$ is locally finite in contradiction to the assumption.
Thus, $\overline{D}$ is the disjoint union of directed triangles, which shows $D\isom Y_\omega$.\looseness-1
\end{proof}

Let us summarize the results of this section.
The following proposition follows directly from Proposition~\ref{prop_CPW} together with Lemmas~\ref{lem_GM4.3}, \ref{lem_In_DeltaCycle}, \ref{lem_InWithT_kl}, and~\ref{lem_In_CP/K/Gen}.

\begin{prop}
Let $D$ be a countable connected C-homogeneous digraph with $D^+\isom I_n$ for some $n\in\nat^\infty$ whose reachability relation is not universal.
If $D$ has at most one end and is not locally finite, then it is isomorphic to one of the following digraphs:
\begin{enumerate}[\rm (i)]
\item $C_m[I_\omega]$ for some $m\in\nat^\infty$ with $m\geq 3$;
\item $Y_\omega$; or
\item $\RF_m$ for some $m\in\nat^\infty$ with $m\geq 3$.\qed
\end{enumerate}
\end{prop}

\subsection{\boldmath Universal reachability relation}\label{sec_D+IndepAFUniversal}

Within this section, let $D$ be a countable connected C-homo\-geneous digraph with $D^+\isom I_n$ for some $n\in\nat^\infty$, with $D^-\isom I_{n'}$ for some $n'\in\nat^\infty$ and with at most one end.
We assume $n,n'\geq 2$ and that $\AF$ is universal.
Due to Lemma~\ref{lem_CycleWitnesses}, some cycle in~$D$ witnesses that $\AF$ is universal.
By Lemma~\ref{lem_UniversalInducedCycle}, we may assume that this is an induced cycle.

\begin{lem}\label{lem_EvenCycleWitnessing}
If $D$ contains an induced cycle of odd length witnessing that $\AF$ is universal, then it contains an induced cycle of length~$4$ witnessing that $\AF$ is universal.
\end{lem}

\begin{proof}
%This is copy\&paste from proof of (13) in \cite[Lemma 5.5]{FinConHomDigraphs}:
Let $C$ be an induced odd cycle witnessing that $\AF$ is universal.
Then $C$ contains a unique induced $2$-arc $xyz$.
The digraphs $C-x$ and $C-y$ are isomorphic induced alternating paths.
By C-homogeneity, we find an automorphism $\alpha$ of~$D$ that maps $C-x$ onto $C-y$.
Since $N^-(z)$ is independent and $x\alpha\in N^-(z)$, the digraph $D[x,y,z,x\alpha]$ is an induced cycle of length~$4$ witnessing that $\AF$ is universal.
\end{proof}

In the following, we fix an induced cycle $C$ of minimal length witnessing that $\AF$ is universal.
Due to Lemma~\ref{lem_EvenCycleWitnessing}, this cycle has even length.

\begin{lem}\label{lem_Dir4Cycle}
There is an isomorphic copy of~$C_4$ in~$D$.
\end{lem}

\begin{proof}
Let $xyz$ be a $2$-arc on~$C$.
Since $C$ has even length, $C-y$ has a non-trivial automorphism: one that maps $x$ to~$z$ and vice versa.
As~$C$ is induced, we can extend this automorphism of~$C-y$ to an automorphism $\alpha$ of~$D$ by C-homogeneity and obtain that $D[x,y,z,y\alpha]$ is a directed cycle of length~$4$.
Note that any directed cycle of length~$4$ is induced since $D^+$ and $D^-$ are edgeless.
\end{proof}

Let $xy\in ED$, let $X:= N^-(x)\sm N^+(y)$, and let $Y:=N^+(y)\sm N^-(x)$.
Obviously, $X$~and~$Y$ are disjoint.
In the following, we investigate the subdigraph $\Gamma:=D[X\cup Y]$ of~$D$.

\begin{lem}\label{lem_GammaHom2Partite}
The subdigraph $\Gamma$ is a non-empty homogeneous $2$-partite digraph.
\end{lem}

\begin{proof}
Let $A$ and~$A'$ be finite subdigraphs of $D[X]$ and let $B$ and~$B'$ be finite subdigraphs of $D[Y]$.
Because $V(B+B')\cap N^-(x)=\es$ and because $D^+(x)$ is edgeless, $x$ is adjacent to no vertex of~$B+B'$.
Similarly, because ${V(A+A')\cap N^+(y)=\es}$ and because $D^-(y)$ is edgeless, $y$ is adjacent to no vertex of~$A+A'$.
Hence, any isomorphism $\varphi$ from $A+B$ to $A'+B'$ extends to an isomorphism from $A+B+x+y$ to $A'+B'+x+y$, that fixes $x$ and~$y$, and thus by C-homogeneity it extends to an automorphism $\alpha$ of~$D$ with $X\alpha=X$ and $Y\alpha=Y$.
In particular, the restriction of~$\alpha$ to~$\Gamma$ is an automorphism of~$\Gamma$ that extends $\varphi$ and fixes both of~$X$ and~$Y$ setwise.
Thus, $\Gamma$ is homogeneous $2$-partite.
As $C_4$ embeds into~$D$, the subdigraph~$\Gamma$ is not empty.\looseness-1
\end{proof}

Having shown that $\Gamma$ is homogeneous $2$-partite, we can apply the classification of the countable such digraphs, Theorem~\ref{thm_ClassHom2Partite}.
So we can investigate the possible digraphs $\Gamma$ one by one, similar to the different possibilities for~$D^+$.
We start with the situation that $\Gamma$ is homogeneous bipartite and show that this cannot occur:

\begin{lem}\label{lem_GammaNotHomBip}
The subdigraph $\Gamma$ is not homogeneous bipartite.
\end{lem}

\begin{proof}
Let us suppose that $\Gamma$ is homogeneous bipartite.
Since $D$ contains some directed cycle of length~$4$ by Lemma~\ref{lem_Dir4Cycle}, we conclude that the edges of~$\Gamma$ are directed from~$Y$ to~$X$.
We consider all possibilities of Theorem~\ref{thm_GGK} one by one.
Note that due to Lemma~\ref{lem_GammaHom2Partite} the digraph $\Gamma$ is not empty.
So there are only four remaining possibilities for~$\Gamma$.

If $G(\Gamma)$ is complete bipartite, then $xy$ cannot be the inner edge of any induced $3$-arc.
As $\Aut(D)$ acts transitively on the $1$-arcs, we conclude that $D$ contains no induced $3$-arc at all.
Since every induced cycle of even length at least $6$ that witnesses that $\AF$ is universalcontains an induced $3$-arc, $C$ has length~$4$.
But as $xy$ is the inner edge of some $3$-arc in a cycle isomorphic to~$C$, the digraph $\Gamma$ must contain some edges that are directed from~$X$ to~$Y$.
This contradiction shows that $G(\Gamma)$ is not complete bipartite.

If $G(\Gamma)$ is a perfect matching, then we know that every induced $2$-arc lies on a unique induced directed cycle of length~$4$.
Due to the previous case, we may assume $|X|\geq 2$.
So every edge lies on at least two directed cycles of length~$4$.
Let $xyuv$ and $xyab$ be two distinct directed cycles of length~$4$ and let $yuwz$ be another directed cycle of length~$4$ containing~$yu$.
Then neither $v$ nor~$u$ is adjacent to any of $a,b,w,z$ since $G(\Gamma)$ is a perfect matching and the same holds for the subdigraph defined by the edge $yu$ instead of~$xy$.
Note that $|C|>4$, since $|C|=4$ implies the existence of some edge from~$X$ to~$Y$.
Thus, the digraph $D[y,z,b,x]$ cannot be a cycle of length~$4$ witnessing that $\AF$ is universal.
Hence, we have $zb\notin ED$.
If $bz\in ED$, then $a$ is not adjacent to~$z$ since neither $zy$ lie in~$D^-(a)$ nor $bz$ lies in $D^+(a)$.
Thus, $yab$ lies on two distinct induced directed cycles of length $4$, once together with $z$ and once together with~$x$.
This is impossible as we already mentioned.
Thus, $b$ and~$z$ are not adjacent.
Hence, C-homogeneity implies the existence of an automorphism $\alpha$ of~$D$ that fixes $x,y,z$ and interchanges $b$ and~$v$.
Since every induced $2$-arc lies on a unique induced directed cycle of length~$4$, we conclude $a\alpha=u$ and $u\alpha=a$.
As $u=a\alpha$ and~$z=z\alpha$ are not adjacent, $a$ and~$z$ are not adjacent, too.
Since $w$ and~$v$ are not adjacent, the same is true for $b$ and~$w\alpha$.
If either $bw\in ED$ or $wb\in ED$, then either $D[x,b,w,u,v]$ or $D[z,w,b,a,w\alpha]$ is a cycle of length~$5$ witnessing that $\AF$ is universal.
By Lemma~\ref{lem_EvenCycleWitnessing}, we conclude $|C|=4$, a contradiction.
Thus, we know that $b$ and~$w$ are not adjacent.
So due to C-homogeneity, $D$ has an automorphism~$\beta$ that fixes $x,y,z,w$ and maps~$v$ to~$b$.
Since $\beta$ fixes $y,z,w$, it must also fix $u$, the unique vertex that forms with the $2$-arc $wzy$ an induced directed cycle of length~$4$.
But we have $(uv)\beta=ub\notin ED$ as previously mentioned, even though $uv$ is an edge of~$D$.
This contradiction shows that $G(\Gamma)$ is not a perfect matching.

If $G(\Gamma)$ is the complement of a perfect matching, then we may assume $|X|\geq 3$ as otherwise $G(\Gamma)$ is also a perfect matching, which we treated before.
Let $z,u,v\in X$ and let $z'$ be the unique vertex in~$Y$ that is not adjacent to~$z$.
Considering the edge $ux$ instead of~$xy$, we obtain a unique vertex $z''\in N^+(x)\sm N^-(u)$ that is not adjacent to~$z'$.
Let us show that $z''$ is adjacent to neither $z$ nor~$v$.
By the structure of~$\Gamma$ applied to the edge $ux$ instead of~$xy$, we find a vertex $u^-\in N^-(u)\sm N^+(x)$ that is a common successor of~$y$ and~$z''$.
Since $u^-\in Y$ abd $u^-\neq z'$, we have $u^-z\in ED$.
Hence, $xz''u^-z$ is a directed cycle of length~$4$ and we conclude that $z$ is not adjacent to~$z''$ since $N^+(z)$ and $N^-(z)$ are independent sets.
If $u^-v\in ED$, then the same argument applies for $v$ and~$z''$ and hence they are not adjacent.
As $\Gamma$ is bipartite, we do not have $vu^-\in ED$.
So let us assume that $u^-$ and~$v$ are not adjacent.
Let us suppose that $v$ and~$z''$ are adjacent.
Since $D^+(v)$ is edgeless, we do not have $vz''\in ED$, so we have $z''v\in ED$.
Then $D[z'',v,z',u,u^-]$ is a cycle of length~$5$ witnessing that $\AF$ is universal.
As above, we conclude $|C|=4$ by Lemma~\ref{lem_EvenCycleWitnessing} and the minimality of~$C$, which is impossible as $\Gamma$ is bipartite.
Thus, $v$ and~$z''$ are also not adjacent if $u^-$ and~$v$ are not adjacent.
We have shown that $z''$ is adjacent to neither $v$ nor~$z$.
Hence, C-homogeneity implies the existence of an automorphism $\alpha$ of~$D$ that fixes $u,x,y,z''$ and maps $z$ to~$v$.
Since $\alpha$ fixes $u,x,z''$, it must also fix the uniquely determined vertex in $N^-(u)\sm N^+(x)$ that is not adjacent to~$z''$, which is~$z'$.
But then $\alpha$ must also fix $z$, the unique vertex in~$X=N^-(x)\sm N^+(y)$ that is not adjacent to~$z'$, in contradiction to the definition of~$\alpha$.
This shows that $G(\Gamma)$ is not the complement of a perfect matching.

It remains to consider the case that $G(\Gamma)$ is the generic bipartite graph.
As mentioned earlier, we have $|C|\neq 4$ as otherwise $\Gamma$ must contain edges from $X$ to~$Y$.
Let $abcd$ be the induced $3$-arc in~$C$.
Then $C-b$ is an induced alternating path and hence embeds into~$\Gamma$.
Let $P$ be an isomorphic copy of $C-b$ in~$\Gamma$.
As $D$ is C-homogeneous, we find an automorphism $\alpha$ of~$D$ with $(C-b)\alpha=P$.
Since both end vertices of~$P$ have successors on~$P$, they lie in~$Y$.
As $G(\Gamma)$ is generic bipartite, the end vertices of~$P$ have a common successor $z$ in~$X$.
Then $D[a\alpha,b\alpha,c\alpha,z]$ is a cycle of length~$4$ witnessing that $\AF$ is universal.
This contradiction to the minimality of~$C$ shows that $\Gamma$ is not homogeneous bipartite.
\end{proof}

Since $\Gamma$ is not homogeneous bipartite, we find an edge $uv\in E\Gamma$ with $u\in X$ and $v\in Y$.
So $D[x,y,u,v]$ is a cycle witnessing that $\AF$ is universal and the minimality of~$C$ implies $|C|=4$.
In the remainder of this section, we will concentrate on arguments that involve the diameter of~$D$.
First, we show that $D$ is homogeneous if its diameter is~$2$:

\begin{lem}\label{lem_Diam2}
If $\diam(D)=2$, then $D$ is homogeneous.
\end{lem}

\begin{proof}
First, let us show that
\begin{txteq}\label{eq_Diam2_1}
for every finite independent vertex set $A$, there are $u,v\in VD$ with $A\sub N^+(u)$ and $A\sub N^-(v)$.
\end{txteq}
We show~(\ref{eq_Diam2_1}) by induction:
If $|A|=2$, then we find a vertex $w$ with $A\sub N(w)$ because of $\diam(D)=2$.
Regardless which edges between $w$ and the elements of~$A$ lie in~$D$, we can use C-homogeneity and the cycle~$C$, into which every induced path of length~$2$ embeds, to conclude that some induced $2$-arc has the two elements of~$A$ as end vertices.
By the same reasons, we find some vertex $u$ with $A\sub N^+(u)$ and some vertex $v$ with $A\sub N^-(v)$.

Now, let us assume $|A|>2$.
First, we show the existence of some vertex with $A$ in its out-neighbourhood.
By induction, we find some $u\in VD$ and $a\in A$ with $A\sm\{a\}\sub N^+(u)$.
Let $a'\in N^+(u)\sm A$.
By induction, we find $z\in VD$ with $a,a'\in N^+(z)$ and such that all but at most two elements of~$A$ lie in~$N^+(z)$.
For all $b\in A\sm N(z)$, the first case $|A|=2$ gives us some $z_b\in VD$ with $b,z\in N^-(z_b)$.
Since $N^+(z)$ is independent, $z_b$ is adjacent neither to~$a$ nor to~$a'$.
Then the digraphs
\[D_1:=D[A\sm\{a\}\cup \{z_b\mid b\in A\sm N(z)\}\cup\{z,a'\}]\]
and
\[D_2:=D[A\cup \{z_b\mid b\in A\sm N(z)\}\cup\{z\}]\]
are isomorphic by an isomorphism $\varphi$ that maps $a'$ to~$a$ and fixes all other vertices.
By construction, $D_1$ and $D_2$ are connected, so $\varphi$ extends to an automorphism $\alpha$ of~$D$.
Since $(A\sm\{a\})\cup\{a'\}\sub N^+(u)$, we conclude $A\sub N^+(u\alpha)$.
By an analogous argument, we find some $v\in VD$ with $A\sub N^-(v)$.
Thus, we have shown~(\ref{eq_Diam2_1}).

Next, we show the following:
\begin{txteq}\label{eq_Diam2_2}
Let $A,B,A',B'$ be finite independent vertex sets of~$D$ such that some isomorphism $\varphi\colon D[A'\cup B']\to D[A\cup B]$ with $A'\varphi=A$ and $B'\varphi=B$ exists.
If $A$ is maximal independent in~$A\cup B$ and if $D$ has a vertex~$v$ with $A'\sub D^+(v)$ and $B'\sub D^-(v)$, then there exists some $u\in VD$ with $A\sub D^+(u)$ and $B\sub D^-(u)$.

\end{txteq}
If $D[A\cup B]$ is connected, then the assertion follows directly by C-homogeneity.
Since the case $B=\es$ is done by~(\ref{eq_Diam2_1}), we may assume $B\neq\es$.
By induction on~$|B|$ we find some vertex $v'\in VD$ with $A\sub N^+(v')$ and $B\sm\{b\}\sub N^-(v')$ for some $b\in B$.
Applying C-homogeneity, we may assume $A'=A$ and $B'\sm\{b'\}=B\sm\{b\}$.
Since $A$ is maximal independent in~$D[A\cup B]$, we know that $b$ has a neighbour $c$ in~$A\cup B$.
This neighbour is also a neighbour of~$b'$ with $b\in N^+(c)$ if and only if $b'\in N^+(c)$.
So $b$ and $b'$ are not adjacent as both lie either in $N^+(c)$ or in $N^-(c)$.
Let $Z$ be a vertex set containing precisely one vertex from each component of $D[A\cup B]$ that does not contain~$b$.
Then $Z\cup\{b,b'\}$ is an independent set and we find a vertex $z$ with $Z\cup\{b,b'\}\sub N^+(z)$ by~(\ref{eq_Diam2_1}).
Then the digraphs $D[A\cup B\cup\{z\}]$ and ${D[A\cup (B\sm\{b\})\cup\{b',z\}]}$ are isomorphic by an isomorphism $\psi$ that maps $b$ to~$b'$ and fixes all other vertices.
Since both digraphs are connected, $\psi$ extends to an automorphism $\alpha$ of~$D$.
Then we have $A\sub N^+(v\alpha)$ and $B\sub N^-(v\alpha)$, which shows~(\ref{eq_Diam2_2}).\looseness-1

To show that $D$ is homogeneous, let $F$ and $H$ be finite isomorphic induced subdigraphs of~$D$ and let $\varphi\colon F\to H$ be an isomorphism.
Let $A\sub VF$ be a maximal independent subset and let $B\sub VF\sm A$ be maximal independent, too.
By~(\ref{eq_Diam2_2}), we find a vertex $u$ with $A\sub N^+(u)$ and $B\sub N^-(u)$.
We have $N(u)\cap VF=A\cup B$ by maximalities of~$A$ and~$B$.
Analogously, we find $v$ with $A\varphi\sub N^+(v)$ and $B\varphi\sub N^-(v)$.
Then $F+u$ and $H+v$ are connected and isomorphic via an isomorphism~$\varphi'$ that extends~$\varphi$.
By C-homogeneity, $\varphi'$ extends to an automorphism of~$D$.
This shows that $D$ is homogeneous.
\end{proof}

The previous lemma enables us to prove that $D$ is homogeneous if $\Gamma$ is not the generic orientation of the countable generic bipartite graph:

\begin{lem}\label{lem_GammaIsCP'kOrGen2Partite}
If $\Gamma$ is either the generic $2$-partite digraph or $CP'_k$ for some $k\in\nat^\infty$, then $D$ is homogeneous.
\end{lem}

\begin{proof}
Up to isomorphism and/or reversing the direction of every edge, the only paths $abcd$ of length~$3$ in a digraph are of the form:
\begin{enumerate}[(a)]
\item $ab,bc,cd\in ED$;
\item $ab,bc,dc\in ED$;
\item $ba,bc,dc\in ED$.
\end{enumerate}
If we can show that in each of these three cases the end vertices $a$ and~$d$ have distance at most~$2$, then we have $\diam(D)=2$ and the assertion follows from Lemma~\ref{lem_Diam2}.
If in any of these three cases $a$ is adjacent to~$c$ or~$b$ is adjacent to~$d$, we can conclude $d(a,d)\leq 2$ directly.
So we may assume that this is not the case.
In case (a), we may assume $bc=xy$ as $\Aut(D)$ acts transitively on the $1$-arcs of~$D$.
Since $a$ and~$c$ are not adjacent, we have $a\in X$ and, since $b$ and~$d$ are not adjacent, we have $d\in Y$.
As $G(\Gamma)$ is a complete bipartite graph in both possibilities for~$\Gamma$, we obtain $d(a,d)=1$.
In cases (b) and (c), we may assume $c=x$, $b\in X$, and $a\in Y$ by C-homogeneity.
Then either $d\in N^-(x)\sm N^+(y)=X$ and $d(a,d)=1$ or $d\in N^-(x)\cap N^+(y)$ and $d(a,d)=2$ because of $a,d\in N(y)$.
This proves $\diam(D)=2$ and hence that $D$ is homogeneous.
\end{proof}

In the following, we assume due to Lemmas~\ref{lem_GammaNotHomBip} and~\ref{lem_GammaIsCP'kOrGen2Partite} and by Theorem~\ref{thm_GGK} that $\Gamma$ is the generic orientation of the countable generic bipartite graph.

\begin{lem}\label{lem_DiamLeq3}
We have $\diam(D)\leq 3$.
\end{lem}

\begin{proof}
Seeking for a contradiction, let us suppose $\diam(D)\geq 4$.
Let $P=x_0\ldots x_4$ be a shortest (not necessarily directed) path between two vertices $x_0$ and~$x_4$ with $d(x_0,x_4)=4$.
Then $P$ embeds into~$\Gamma$, as every finite $2$-partite digraph embeds into~$\Gamma$.
Hence, we find an automorphism~$\alpha$ of~$D$ that maps~$P$ into~$\Gamma$.
Then either $x_0\alpha$ and $x_4\alpha$ lie in~$X$ or they lie in~$Y$.
In both cases, they have a common neighbour, either~$x$ or~$y$.
Thus, $x_0$ and~$x_4$ have a common neighbour.
This contradiction to $d(x_0,x_4)=4$ shows $\diam(D)\leq 3$.
\end{proof}

Since we already investigated the case $\diam(D)=2$, the only remaining situation is $\diam(D)=3$.
We shall prove that in this situation $D$ and $\Gamma$ are isomorphic.

\begin{lem}\label{lem_Diam3GenOrient}
If $\diam(D)\neq2$, then $D$ is the generic orientation of the countable generic bipartite graph.
\end{lem}

\begin{proof}
By Lemma~\ref{lem_Diam2} and Lemma~\ref{lem_DiamLeq3}, we may assume $\diam(D)=3$.
Let $D_i(x)$ be the set of those vertices of~$D$ whose distance to~$x$ is~$i$.
The first observation in this proof is that
\begin{txteq}\label{eq_Diam3GenOrient1}
there are non-adjacent vertices $a\in D_1(x)$ and $b\in D_2(x)$.
\end{txteq}
Indeed, if all vertices $a\in D_1(x)$ and $b\in D_2(x)$ are adjacent, then every vertex in $VD=\{x\}\cup D_1(x)\cup D_2(x)\cup D_3(x)$ has distance at most $2$ to~$a$ and we obtain $\diam(D)=2$, a contradiction to our assumption.

Let us show that
\begin{txteq}\label{eq_Diam3GenOrient2}
the end vertices of any induced path of length~$3$ have distance~$3$.
\end{txteq}
Let $P_1$ be a path of length~$3$ whose end vertices have distance~$3$ and let $P_2$ be another induced path of length~$3$.
By using C-homogeneity and the cycle~$C$, we can modify $P_1$ and obtain a path $P_3$ with the same end vertices like~$P_1$ and such that $P_2$ and~$P_3$ are isomorphic.
Hence,~(\ref{eq_Diam3GenOrient2}) holds.

Next, we show that
\begin{txteq}\label{eq_Diam3GenOrient3}
$D$ contains no triangle.
\end{txteq}
Let us suppose that $D$ contains some triangle.
Since $N^+(x)$ and $N^-(x)$ are independent sets, this triangle is a directed triangle.
Let $a\in Y$, $b\in X$, $x$, and $d\in N^-(x)\cap N^+(y)$.
Then $D[a,b,x,d]$ is an induced path of length~$3$ as $N^+(y)$ and $N^-(x)$ are independent vertex sets.
Due to~(\ref{eq_Diam3GenOrient2}), we have $d(a,d)=3$, but $y$ is a common neighbour of~$a$ and~$d$.
This contradiction shows~(\ref{eq_Diam3GenOrient3}).

A direct consequence of~(\ref{eq_Diam3GenOrient3}) is that $D_1(x)$ is an independent set.
Let us show that
\begin{txteq}\label{eq_Diam3GenOrient4}
$D_2(x)$ is an independent set.
\end{txteq}
If this is not the case, then two vertices $a,b\in D_2(x)$ are adjacent.
Let $c$ be a common neighbour of~$b$ and~$x$.
By~(\ref{eq_Diam3GenOrient3}), we know that $a$ and~$c$ are not adjacent.
Hence $D[a,b,c,x]$ is an induced path of length~$3$.
So its end vertices have distance~$3$ by~(\ref{eq_Diam3GenOrient2}) in contradiction to the choice of~$a$.

We have almost proved that $D$ is $2$-partite.
The only edges that might be an obstacle for this are those with both its incident vertices in~$D_3(x)$.
So let us exclude such edges:
\begin{txteq}\label{eq_Diam3GenOrient5}
$D_3(x)$ is an independent set.
\end{txteq}
Let us suppose that some edge $ab$ has both its incident vertices in~$D_3(x)$.
Let $P$ be a path of length~$3$ from~$x$ to~$a$.
Due to (\ref{eq_Diam3GenOrient3}), $Pab$ is induced and its end vertices have distance~$3$.
As~$\Gamma$ is the generic orientation of the countable generic bipartite graph, we also find an isomorphic copy $P'$ of~$P$ in~$\Gamma$.
By C-homogeneity, we find an automorphism~$\alpha$ of~$D$ that maps~$P$ to~$P'$.
Since the end vertices of~$P'$ lie either both in~$X$ or both in~$Y$, they have a common neighbour, either~$x$ or~$y$, respectively, and thus they have distance~$2$.
Therefore, the distance between the end vertices of $P=P'\alpha\inv$ must be~$2$, too.
This contradiction to the choice of~$b$ shows~(\ref{eq_Diam3GenOrient5}).

As mentioned earlier, we obtain from~(\ref{eq_Diam3GenOrient3}), (\ref{eq_Diam3GenOrient4}), and (\ref{eq_Diam3GenOrient5}) that $D$ is a $2$-partite digraph with partition sets $U:=\{x\}\cup D_2(x)$ and $W:=D_1(x)\cup D_3(x)$.
Let $A,B$, and~$C$ be finite subsets of~$U$.
Then we find a finite set~$F\sub VD$ such that
\[H:=D[A\cup B\cup C\cup F]\]
is connected.
As $H\sub D$ is $2$-partite and~$\Gamma$ is the gerneric orientation of the countable generic bipartite graph, we find an isomorphic copy of~$H$ in~$\Gamma$.
By C-homogeneity, there is an automorphism $\alpha$ of~$D$ with $H\alpha\sub\Gamma$ such that either $(A\cup B\cup C)\alpha\sub X$ or $(A\cup B\cup C)\alpha\sub Y$.
As $\Gamma$ is the generic orientation of the countable generic bipartite graph, there is a vertex $v$ either in~$Y$ or in~$X$ with $A\alpha\sub N^+(v)$ and $B\alpha\sub N^-(v)$ and $C\alpha\cap N(v)=\es$.
Then $v\alpha\inv$ is a vertex we are searching for.
An analogous argument shows the existence of such a vertex if $A,B$, and~$C$ are finite subsets of~$W$.
Hence, we have shown that $D$ is the generic orientation of the countable generic bipartite graph.
\end{proof}

Let us summarize the results of this section:

\begin{prop}
Let $D$ be a countable connected C-homogeneous digraph whose reachability relation is universal.
If $D^+\isom I_n$ for some $n\in\nat^\infty$, then $D$ is either homogeneous or the generic orientation of the countable generic bipartite graph.\qed
\end{prop}

\subsection{A second result}

By summarizing the propositions with Section~\ref{sec_ProofPartII}, we obtain the following theorem:

\begin{thm}\label{thm_Part2}
Let $D$ be a countable connected C-homogeneous digraph such that $D^+\isom I_n$ for some $n\in\nat^\infty$.
If $D$ has at most one end and is not locally finite, then it is isomorphic to one of the following digraphs:
\begin{enumerate}[\rm(i)]
\item a homogeneous digraph;
\item $C_m[I_\omega]$ for some $m\in\nat^\infty$ with $m\geq 3$;
\item $Y_\omega$;
\item $\RF_m$ for some $m\in\nat^\infty$ with $m\geq 3$; or
\item the generic orientation of the countable generic bipartite graph.\qed
\end{enumerate}
\end{thm}

Theorems~\ref{thm_Part1} and~\ref{thm_Part2} together with~\cite[Theorem~2.1]{FinConHomDigraphs} and~\cite[Theorems~4.2 and~7.6]{HH-ConHomDigraphs} and~\cite[Theorem~6.2]{GM-CHomDigraphs} imply our main result, Theorem~\ref{thm_main}.

\bibliographystyle{amsplain}
\bibliography{Bibs}

\providecommand{\bysame}{\leavevmode\hbox to3em{\hrulefill}\thinspace}
\providecommand{\MR}{\relax\ifhmode\unskip\space\fi MR }
\providecommand{\MRhref}[2]{%
  \href{http://www.ams.org/mathscinet-getitem?mr=#1}{#2}
}
\providecommand{\href}[2]{#2}
\begin{thebibliography}{10}

\bibitem{C-DistanceTransitive}
P.J. Cameron, \emph{A census of infinite distance-transitive graphs}, Discrete
  Math. \textbf{192} (1998), no.~1-3, 11--26.

\bibitem{CPW}
P.J. Cameron, C.E. Praeger, and N.C. Wormald, \emph{Infinite highly arc
  transitive digraphs and universal covering digraphs}, Combinatorica
  \textbf{13} (1993), no.~4, 377--396.

\bibitem{Cherlin-HomImprimitive}
G.L. Cherlin, \emph{Homogeneous directed graphs. the imprimitive case}, Logic
  colloquium '85 ({O}rsay, 1985), Stud.\ Logic Found.\ Math., vol. 122,
  North-Holland, Amsterdam, 1987, pp.~67--88.

\bibitem{Cherlin-CountHomDigraphs}
\bysame, \emph{The classification of countable homogeneous directed graphs and
  countable homogeneous $n$-tournaments}, vol. 131, Mem.\ Amer.\ Math.\ Soc.,
  no. 621, Amer.\ Math.\ Soc., 1998.

\bibitem{DGMS-SetHomogeneous}
M.~Droste, M.~Giraudet, D.~Macpherson, and N.~Sauer, \emph{Set-homogeneous
  graphs}, J.\ Combin.\ Theory (Series B) \textbf{62} (1994), no.~1, 63--95.

\bibitem{Enomoto}
H.~Enomoto, \emph{Combinatorially homogeneous graphs}, J.\ Combin.\ Theory
  (Series B) \textbf{30} (1981), no.~2, 215--223.

\bibitem{Gard-HomogeneousGraphs}
A.~Gardiner, \emph{Homogeneous graphs}, J.\ Combin.\ Theory (Series B)
  \textbf{20} (1976), no.~1, 94--102.

\bibitem{Gard-HomConditions}
\bysame, \emph{Homogeneity conditions in graphs}, J.\ Combin.\ Theory (Series
  B) \textbf{24} (1978), no.~3, 301--310.

\bibitem{GGK}
M.~Goldstern, R.~Grossberg, and M.~Kojman, \emph{Infinite homogeneous bipartite
  graphs with unequal sides}, Discrete Math. \textbf{149} (1996), no.~1-3,
  69--82.

\bibitem{GK-kHomogeneous}
Ja.Ju. Gol'fand and M.H. Klin, \emph{On $k$-homogeneous graphs}, Algorithmic
  Stud.\ Combin., vol. 186, 1978, pp.~76--85 (in Russian).

\bibitem{GrayMacpherson}
R.~Gray and D.~Macpherson, \emph{Countable connected-homogeneous graphs}, J.\
  Combin.\ Theory (Series B) \textbf{100} (2010), no.~2, 97--118.

\bibitem{GM-CHomDigraphs}
R.~Gray and R.G. M\"oller, \emph{Locally-finite connected-homogeneous
  digraphs}, Discrete Math. \textbf{311} (2011), no.~15, 1497--1517.

\bibitem{FinConHomDigraphs}
M.~Hamann, \emph{The classification of finite and locally finite
  connected-homogeneous digraphs}, submitted, arXiv:1101.2330.

\bibitem{Hom2PartiteDi}
\bysame, \emph{Homogeneous $2$-partite digraphs}, submitted, arXiv:1311.5056.

\bibitem{HH-ConHomDigraphs}
M.~Hamann and F.~Hundertmark, \emph{The classification of connected-homogeneous
  digraphs with more than one end}, Trans.\ Am.\ Math.\ Soc. \textbf{365}
  (2013), no.~1, 531--553.

\bibitem{HP-Transitivity}
M.~Hamann and J.~Pott, \emph{Transitivity conditions in infinite graphs},
  Combinatorica \textbf{32} (2012), no.~6, 649--688.

\bibitem{HedmanPong}
S.~Hedman and W.Y. Pong, \emph{Locally finite homogeneous graphs},
  Combinatorica \textbf{30} (2010), no.~4, 419--434.

\bibitem{Henson-CountHom}
C.W. Henson, \emph{Countable homogeneous relational structures and
  {$\aleph_0$}-categorical theories}, J.\ Symbolic Logic \textbf{37} (1972),
  494--500.

\bibitem{L-FiniteHomDigraphs}
A.H. Lachlan, \emph{Finite homogeneous simple digraphs}, Proceedings of the
  {H}erbrand symposium ({M}arseilles, 1981) (J.~Stern, ed.), Stud.\ Logic
  Found.\ Math., vol. 107, North-Holland, 1982, pp.~189--208.

\bibitem{L-Tournaments}
\bysame, \emph{Countable homogeneous tournaments}, Trans.\ Am.\ Math.\ Soc.
  \textbf{284} (1984), no.~2, 431--461.

\bibitem{LW-CountUltrahomGraphs}
A.H. Lachlan and R.~Woodrow, \emph{Countable ultrahomogeneous undirected
  graphs}, Trans.\ Am.\ Math.\ Soc. \textbf{262} (1980), no.~1, 51--94.

\bibitem{DistanceTransitive}
H.D. Macpherson, \emph{Infinite distance transitive graphs of finite valency},
  Combinatorica \textbf{2} (1982), no.~1, 63--69.

\bibitem{Macpherson-Survey}
\bysame, \emph{A survey of homogeneous structures}, Discrete Math. \textbf{311}
  (2011), no.~15, 1599--1634.

\bibitem{Schmerl-HomogeneousPO}
J.H. Schmerl, \emph{Countable homogeneous partially ordered sets}, Algebra
  Universalis \textbf{9} (1979), no.~3, 317--321.

\bibitem{J-SmoothlyEmbeddable}
J.~Sheehan, \emph{Smoothly embeddable subgraphs}, J.\ London Math.\ Soc.\ (2)
  \textbf{9} (1974), 212--218.

\end{thebibliography}

\end{document}